\documentclass[11pt]{article}
\usepackage{enumerate}
\usepackage{amsmath}
\usepackage{amsthm}
\usepackage{amsfonts}
\usepackage{amssymb}
\usepackage[numbers]{natbib}
\usepackage{xcolor}
\usepackage{graphicx}
\usepackage{tikz}
\usetikzlibrary{matrix,arrows.meta}
\usepackage{fullpage}
\usepackage{xcolor}
\usepackage{hyperref}
\usepackage{bbm}

\newcommand*\samethanks[1][\value{footnote}]{\footnotemark[#1]}

\newcommand{\cP}{ \mathcal{P} }
\newcommand{\bfsigma}{ \mathbf{\sigma} }
\renewcommand{\Pr}{ \mathbb{P} }

\newtheorem{question}{Question}
\numberwithin{question}{section}

\newtheorem{corollary}[question]{Corollary}
\newtheorem{theorem}[question]{Theorem}
\newtheorem{proposition}[question]{Proposition}
\newtheorem{lemma}[question]{Lemma}
\newtheorem{claim}[question]{Claim}

\newtheorem{definition}[question]{Definition}

\newtheorem{remark}[question]{Remark}

\numberwithin{equation}{section}

\title{On first-order limiting probabilities for small permutation classes}
\date{}

\author{Victor Falgas--Ravry\thanks{Institutionen f\"or matematik och matematisk statistik, Ume\aa\text{ } Universitet, Sweden.\\	Emails: $\{$\texttt{victor.falgas-ravry}, \texttt{klas.markstrom}, \texttt{eero.raty}$\}$\texttt{@umu.se}. Research supported in part by a postdoctoral grant 213-0204 from the Olle Engkvist Foundation, and by Vetenskapsr{\aa}det grant VR 2025-03758 (E.R.), by a postdoctoral grant from the Osk. Huttunen
Foundation (E.R.), and by Vetenskapsr{\aa}det grant VR 2021-03687 (V.F.-R.).} \and Klas Markstr\"om\samethanks[1] \and Eero R\"aty\samethanks[1] \and Fiona Skerman\thanks{Institutionen f\"or matematik, Uppsala Universitet, Sweden. Email: \texttt{fiona.skerman@math.uu.se}. F.S. is supported by the Wallenberg AI, Autonomous Systems and Software Program (WASP) funded by the Knut and Alice Wallenberg Foundation.}}

\begin{document}
	\maketitle
	
\begin{abstract}
    Given a small permutation class $\cP$, we consider the set of possible limits of the probability $\Pr\left(\bfsigma_n \models \Psi\right)$ that a uniformly chosen random permutation $\bfsigma_n\in \cP_n$ satisfies a fixed first-order sentence in Theory of Two Orders (TOTO). In the case of small permutation classes with polynomial growth, we show this limit set is finite, and characterize the set explicitly. In particular, we prove that many natural small permutation classes of polynomial growth arising from permutation matrices satisfy a $0$-$1$ convergence law, and have an infinite model. 

    In the case of small permutation classes with exponential growth we show that the limit set is dense in $[0,1]$, thereby extending a result  of Albert--Bouvel--F\'eray--Noy, using different methods, to this range.
\end{abstract}

\section{Introduction}
One way to characterize the typical features of the elements of some class of  combinatorial objects is to draw such an element at random according to some probability measure, and study its properties. In particular, one may ask for the probability that such a randomly chosen element satisfies some property that can be described using a suitable version of first-order logic. Starting from the work of Glebskii, Kogan,  Liogon'kii and Talanov~\cite{glebskii1969range} and Fagin~\cite{fagin1976probabilities}, there have been extensive investigations of logical limit laws for random graph models. The monograph~\cite{spencer2013strange} by Spencer gives an excellent overview of the subject.

In this paper, we shall be interested in analogous questions for permutations drawn uniformly at random from a permutation class (that is, from sets of permutations $\cP$ that are closed under taking subpatterns, see Section~\ref{section: preliminaries} for precise definitions). Such questions have recently received much  attention and relatively little is known. The study of first-order logical limit laws for permutations in general was initiated by Foy and Woods~\cite{foy1990probabilities}, who showed that a permutation drawn \emph{uniformly at random} from $\mathfrak{S}_{n}$ does not satisfy a logical limit law. Their work was a strengthening of earlier work by Kaufmann and Shelah \cite{KAUFMANN1985285} who showed non-convergence for monadic second-order logic. Henceforth, we shall refer to a permutation drawn uniformly at random from a permutation class simply as a random permutation from that class.

In contrast to these older results it has recently been shown that for some non-trivial permutation classes there is a first-order limit law, and for some of them the set of limiting probabilities has also been described. In particular, both of these were established for 231-avoiding permutations in~\cite{albert2024logical}, and the first-order limit law was subsequently established for 321-avoiding permutations in~\cite{ozdemir2023first}, which up to isomorphism covers all classes avoiding a single pattern of length 3.    Another positive result was given in~\cite{braunfeld2021logical}, where it was shown that the class of layered permutations satisfies a limit law. They also conjectured that every proper permutation class satisfies a limit law.  Going in another direction, \cite{muller2023logical} identified for which parameter values the  non-uniform Mallows distribution on the set of all permutations has a limit law.

Rather than studying individual permutation classes, the aim of this paper is to consider the entire collection of what are known as \emph{small permutation classes}.  As shown by Marcus and Tardos~\cite{marcus2004excluded} for every proper permutation class there exists an (upper) growth rate $\tau\geq 1$ so that the number of permutations of order $n$ in the class grows like (at most) $\tau^{n+o(n)}$, with $\tau=1$ entailing polynomial growth in $n$.  It has been known for some time now (see the discussion in the next subsection) that a qualitative shift in the behaviour of permutation classes occurs at the growth rate $\kappa=2.205\ldots$, the unique positive root of the polynomial $1+2X^2-X^3$. A permutation class with growth rate smaller than this is called a small permutation class.  

In~\cite{eeroklas} the second and third authors of this paper show that every small permutation class has a bijection to a regular language which also acts as a bijection between the  first-order logics on either side. In fact, the encoding and the structural bijection extend to monadic second order logic as well. Together with some further structural knowledge of the regular language, this implies that  for every small permutation class, the probability of a first-order property holding for an order $n$ random permutation from that class converges to a limit as $n\rightarrow \infty$. See Theorem~\ref{theorem: eeroklas} for more precise statements of these results. This shows that the convergence results from~\cite{albert2024logical} and~\cite{ozdemir2023first} hold for all small permutation classes, and proves the conjecture from~\cite{braunfeld2021logical} mentioned above for all small permutation classes.

The broad aim of this paper is to determine the set of possible limit probabilities for small permutation classes. For each class of polynomial growth we will show that  the limit set is finite, and characterize those for which there is a $0$-$1$ law. In the latter case, we also show that there is an infinite model for the class, analogous to the infinite random graph in Fagin's work~\cite{fagin1976probabilities}.  For small permutation classes with superpolynomial growth we instead show that the  set of limiting probabilities is dense in $[0,1]$. To properly state our results and set them in their context, a fuller background is necessary; this is covered in the next subsection, with our own contributions  described in full in Section~\ref{subsection: results}.

\subsection{Background and motivation}

\subsubsection{Permutation classes}

Permutation classes have been extensively studied since the pioneering work of Knuth~\cite[Section 2.2.1]{knuth1968art} in the late 1960s. Suppose that a permutation class $\cP$ is \emph{proper}, in the sense that it is not the class of all permutations. What can one say about the growth-rate of $\vert \cP_n\vert =\vert \cP\cap\mathfrak{S}_{n}\vert $ as $n \rightarrow \infty$? The collection $\mathfrak{S}_{n}$ of all permutations of $[n]$ has size $n!$; however, the celebrated Stanley--Wilf conjecture held that for any proper permutation class the growth is at most exponential in $n$. Following work of many authors, notably Alon and Friedgut~\cite{alon2000number} and Klazar~\cite{klazar2000furedi}, the Stanley--Wilf conjecture was finally proved by Marcus and Tardos~\cite{marcus2004excluded} in 2004.

Given Marcus and Tardos's result, we have an exponential upper bound for the size of a proper permutation class.  When $ \limsup_{n\rightarrow \infty} \vert \cP_n\vert^{1/n} = \liminf_{n\rightarrow \infty} \vert \cP_n\vert^{1/n}$ we refer to their common value as the \emph{Stanley--Wilf limit} or \emph{growth rate} $\mathrm{gr}(\cP)$ of the permutation class $\cP$.

The Stanley--Wilf limit is conjectured to exist for all permutation classes $\cP$ (see e.g.~\cite[Conjecture 1.1]{vatter2015permutation}, though this appears to be an older folklore conjecture). A natural question to ask is then: which values can be attained as the Stanley--Wilf limit of some permutation class? When $\cP$ is the set $\mathrm{Av}(\pi)$ of all permutations avoiding the pattern $\pi$, the Stanley--Wilf limit exists, by a remark of Arratia~\cite{arratia1999stanley} in combination with \cite{marcus2004excluded}, but finding the value is hard even for simple patterns.

Given the difficulty of determining growth rates even for some seemingly simple permutation classes, it is natural to approach the problem from a slightly different angle and ask instead: given a (small) constant $c > 0$, what can be said about the structural properties of permutation classes with growth-rate less than $c$, or to determine whether $c$ is the Stanley--Wilf limit of some class? 
Questions of this type have seen significantly more progress than the ones concerned with determining the growth rates. One of the first results in this direction was obtained by Kaiser and Klazar~\cite{kaiser2002growth}, who showed in 2002 that if for some arbitrary $n$ we have $\vert \cP_n\vert < F_n$, where $F_n$ denotes the $n$-th Fibonacci number, then $\vert \cP_n\vert $ grows at most polynomially in $n$. Thus for all proper permutation classes $\cP$, the growth rate $gr(\cP)$ is either $0$, $1$ or at least $(1+\sqrt{5})/2$. They also obtained a precise classification of the possible growth rates/Stanley--Wilf limits below $2$ (which correspond to certain generalized Fibonacci numbers). 

After several improvements~\cite{huczynska2006grid, polygrowth}, a remarkable breakthrough was made by Vatter~\cite{vatter2011small}, who showed that a certain type of phase transition occurs at the growth-rate $\kappa \approx 2.205$, where $\kappa$ is the least positive root of $1+2X^2-X^3$. We say that a permutation class $\mathcal{P}$ is a \emph{small permutation class} if $\mathrm{\overline{gr}}(\mathcal{P}) < \kappa$. In his influential work, Vatter~\cite{vatter2011small} provided a rather specific structural classification for small permutation classes (which we will state in Section~\ref{sec: Inflations of geometric grid classes}), proved that there are only countably many small permutation classes, and also proved that there are already uncountably many permutation classes with growth-rate precisely $\kappa$ and  pairwise distinct generating functions.

Other natural phase transitions occurring at higher growth-rates were subsequently identified by Pantone and Vatter~\cite{pantonevatter}, and Vatter~\cite{vatter2010permutation}: Pantone and Vatter~\cite{pantonevatter} determined an explicit constant $\xi$ such that there are only countably many Stanley--Wilf limits below $\xi$, but such that for every $\varepsilon>0$ there are uncountably many Stanley--Wilf limits in every half-open interval $[\xi, \xi+\varepsilon)$. Vatter~\cite{vatter2010permutation}, on the other hand, showed that there exists a constant $\lambda \approx 2.48$ such that every real number greater than $\lambda$ is the Stanley--Wilf limit of some proper permutation class.

Vatter~\cite{vatter2011small} also conjectured that every small permutation class has a rational generating function. The conjecture was later proved by Albert, Ru{\v{s}}kuc, and Vatter~\cite{albert2015inflations}; in particular, their result implies that the Stanley--Wilf limit exists for every small permutation class. Recently, the second and third authors gave an alternative proof of this conjecture as a consequence of more general results~\cite{eeroklas}. 

There is a much broader literature on permutation classes and pattern-avoiding permutations, including their connections to other well-studied objects in combinatorics and computer science such as stacks, ordered graphs and matrices, and in this short introduction we can only touch upon the tip of the iceberg. We refer the reader to the monographs and surveys of Bona~\cite{bona2008combinatorics}, Kitaev~\cite{kitaev2010survey}, Linton~\cite{linton2010permutation}, Steingrimsson~\cite{steingrimsson2013some} and Vatter~\cite{vatter2015permutation} for much more on this rich area of combinatorics.

\subsubsection{Logical limit laws}
A key difference between the study of logical limit laws for random permutations and for random graphs is that unlike in the graph case, there are at least two genuinely different and very natural logics for permutations, each of which has its own advantages. The first logic appeared in~\cite{foy1990probabilities} and is now known as the Theory of Two Orders (abbreviated as TOTO). This logic concerns a set equipped with two independent linear orders. Informally speaking, one should view permutations in TOTO as a set of points in the plane with no two points placed on the same horizontal or vertical line, and the two orders can be chosen to be the usual orders comparing the x- and y-coordinates of the respective points. TOTO is generally more suitable for problems about pattern-avoidance in permutations, and is the logic we shall consider in this paper; see Section~\ref{subsection: permutations} for a more detailed introduction to TOTO. The second logic, which we will not define here,  Theory of One Bijection (TOOB), is more suitable when working with problems that concern the cycle-structure of permutations. For an excellent account of both of these, see e.g.\ the recent paper of Albert, Bouvel and F\'eray~\cite{albert2020two}.

For uniform random permutations Foy and Woods~\cite{Foy1990} showed that in general TOTO formulae do not satisfy a limit law.  In 2022 Braunfeld and Kukla~\cite{braunfeld2021logical} showed, among other results, that the class of so-called `layered' permutations  satisfies a first-order limit law in TOTO. They also made the rather strong conjecture that every proper permutation class satisfies a first-order limit law in TOTO. 

More recently, Albert, Bouvel, F\'eray and  Noy~\cite{albert2024logical} showed that a uniformly chosen random permutation from the class $\mathrm{Av}(231)$ of $231$-avoiding permutations satisfies a first-order logical limit law in TOTO. Furthermore, they showed that the set $\mathrm{Lim}(\mathrm{Av}(231))$ of limits of probabilities that a fixed TOTO sentence is satisfied is dense in $[0,1]$.  Shortly thereafter, \"Ozdemir~\cite{ozdemir2023first} obtained a first-order logical limit law in TOTO for $\mathrm{Av}(321)$; using symmetries, these results combined imply that all six of the permutation classes $\mathrm{Av}(\sigma)$, $\sigma \in \mathfrak{S}_{3}$, satisfy a first-order logical limit law.  

Recall that very recently the second and third authors proved that any small permutation class is first-order isomorphic to a regular language.  This means that there is a bijection from the permutation class to the language, which also extends to a bijection of sentences in first-order logic.  Together with further structural properties of this regular language,  they concluded that every small permutation class satisfies a first-order convergence law for sentences in TOTO. 
\begin{theorem}[See~\cite{eeroklas}] \label{theorem: eeroklas}
    Let $\mathcal{F}$ be a small permutation class. 
    \begin{enumerate}
        \item Then there exists a regular language $L_{\mathcal{F}}$ and a length-preserving bijection
        \[f_{\mathcal{F}}: L_{\mathcal{F}} \rightarrow \mathcal{F}\]
        satisfying the following property: for every first-order logic sentence $\Psi$ on $\mathcal{F}$ there exists a first-order logic  sentence $\Phi$ on $L_{\mathcal{F}}$ so that 
        \[\{w \in L_\mathcal{F}: w \models \Phi\} = \{w \in L_\mathcal{F}: f_\mathcal{F}(w) \models \Psi\}.\]
        The same conclusion also holds for sentences in monadic second-order logic.
        \item{(sketched)} The bijection $f_\mathcal{F}$ is in fact a monadic second-order isomorphism: that is, for every sentence $\Phi$ in monadic second-order logic on $L_\mathcal{F}$ there exists a corresponding sentence in monadic second-order logic on $\mathcal{F}$ for which the conclusion above holds. In addition, it is also possible to construct a regular language $L_\mathcal{F}'$ and a length-preserving bijection $f'_\mathcal{F}: L_\mathcal{F}' \rightarrow \mathcal{F}$ which is a first-order isomorphism. 
        \item For the last three assertions, suppose that $\mathcal{F}$ is not finite. There exist a positive real number $\alpha$, a real number $\beta \geq 1$, and a non-negative integer $c$ so that 
        \[\vert \mathcal{F}_n\vert = (\alpha+o(1))\cdot \beta ^n\cdot n^c.\]
        \item Given a monadic second-order sentence $\Psi$ on $\mathcal{F}$, write \[\mathcal{F}^\Psi_n := \{\sigma \in \mathcal{F}_n: \sigma \models \Psi\}.\]
        Then the sets $\mathcal{F}^\Psi_n$ have a rational generating function.
        \item For every first-order $\mathrm{TOTO}$ sentence $\Phi$ the limit  $\lim_{n\rightarrow \infty}\Pr(\pi \models \Phi)$ exists, where $\pi$  is taken uniformly at random from $\mathcal{F}_n$. 
    \end{enumerate}
\end{theorem}

We are now ready to prove the main theorem. 

\subsection{Contributions of this paper}\label{subsection: results}
As we have seen, small permutation classes are of interest within the generally challenging theory of permutation classes. They are one of the few significant families of permutation classes for which researchers have been able to derive some general results; in particular small permutation classes have been shown to have a pleasant geometric structure. They are thus a natural first family to study and to try to understand more fully if one is to grapple with permutation classes in their full generality.

For small permutation classes in the polynomial growth regime, we show that the set of limiting probabilities is finite, and determine the set explicitly. As a byproduct of our proof, we also obtain a short alternative proof of the first-order convergence law from \cite{eeroklas} for permutation classes of polynomial growth. Before we state our result, let us note that, as a consequence of the work of Kaiser and Klazar~\cite{kaiser2002growth}, if $\mathcal{P}$ is a small permutation class with polynomial growth then there exist positive integers $b$ and $k$ such that $\left\vert \mathcal{P}_{n}\right\vert =b{n \choose k-1}+O\left(n^{k-2}\right)$ (see the discussion after~\eqref{eq: size of P}).
\begin{theorem}
	\label{thm:poly growth} Let $\mathcal{P}$ be a permutation class with polynomial growth for which  $\mathcal{P}_{n}\neq\emptyset$ for
	all $n \in \mathbb{Z}_{>0}$. Let $b$ and $k$ be positive integers such that $\left\vert \mathcal{P}_{n}\right\vert =b{n \choose k-1}+O\left(n^{k-2}\right)$. Then 
	\begin{enumerate}
		\item[(i)] for all first-order sentences $\Psi$ in TOTO, the probability $\Pr\left(\bfsigma_n \models \Psi\right)$ that a random permutation $\bfsigma_n$  from $\mathcal{P}_n$ satisfies $\Psi$ tends to a limit as $n \rightarrow \infty$ with convergence rate $O(n^{-1})$. 
		\item[(ii)] the set of limiting probabilities 
        \[\mathrm{Lim}\left(\mathcal{P}\right):=\] \[ \left\{\lim_{n\rightarrow \infty} \mathbb{P}(\bfsigma_n \models \Psi): \ \Psi \textrm{ first-order TOTO sentence},\ \bfsigma_n \textrm{ drawn uniformly at random from }\cP_n\right\} \]
        consists of all the fractions $\left\{ \frac{b'}{b}: \ b'\in\left[0,b\right]\cap\mathbb{Z}\right\} $, where the integer $b$ is defined as above. In particular, this set of limiting probabilities is always finite, and $\cP$ satisfies a $0$-$1$ law if and only if $b=1$.
	\end{enumerate}
\end{theorem}

\noindent In the special case where $\cP$ is the geometric grid class of a signed permutation matrix (see Section~\ref{subsection: geom grid classes} for a formal definition of these terms), our methods allow us to construct an infinite TOTO model for the class $\cP$. 
\begin{theorem}\label{theorem: infinite model}
    Let $\mathcal{P}$ be the geometric grid class of a signed permutation matrix. Then there is a permutation $\mathfrak{P}$ on an infinite set of points that gives an infinite model for the set of TOTO axioms defining $\mathcal{P}$ and that is asymptotically almost surely TOTO first-order indistinguishable, as $n\rightarrow \infty$, from a permutation $\bfsigma_n$ drawn uniformly at random from $\cP_n$ (i.e.\ for any fixed $\Psi$, the probability that $\bfsigma_n \models \Psi$ converges to $1$ as $n\rightarrow \infty$ if $\mathfrak{P}\models \Psi$, and to $0$ otherwise).
\end{theorem}

For small permutation classes in the exponential growth regime we show, analogously to the result from~\cite{albert2024logical} for $231$-avoiding permutations,  that the set of limiting probabilities is dense in $[0,1]$. However, our proof uses very different methods in order to handle the much more general setting we work in.
\begin{theorem}\label{thm:exp growth}
    Let $\mathcal{P}$ be a small permutation class with exponential growth. Then $\mathrm{Lim}(\mathcal{P})$ is a dense subset of $[0,1]$. 
\end{theorem}
\noindent Note that in this regime we rely on the convergence result from \cite{eeroklas} and our contribution is thus the density of the set $\mathrm{Lim}(\cP)$.

\section{Preliminaries}\label{section: preliminaries}
\subsection{Permutations and TOTO}\label{subsection: permutations}
We denote by $\mathfrak{S}_{n}$ the collection of permutations of $[n]:=\{1,2, \ldots, n\}$, and we write $\vert \sigma\vert =n$ if $\sigma \in \mathfrak{S}_{n}$. We often represent a permutation $\sigma \in \mathfrak{S}_{n}$  as the ordered sequence $\sigma(1)\sigma(2)\ldots \sigma(n)$. Given permutations $\tau\in \mathfrak{S}_{m}$ and $\sigma \in \mathfrak{S}_{n}$ for $n\geq m$, we say that $\tau$ is a \emph{subpattern} of $\sigma$ if there exists an order-preserving injection $f: \ [m]\rightarrow[n]$ such that for all $i,j\in [m]$,  $\tau(i)\leq \tau(j)$ if and only if $\sigma(f(i))\leq \sigma(f(j))$.  A \emph{permutation class} is a collection of permutations $\cP \subseteq \bigcup_{n\in \mathbb{N}} \mathfrak{S}_{n}$ which is closed under taking subpatterns. Given a fixed permutation $\tau$, we denote by $\mathrm{Av}(\tau)$ the permutation class of all \emph{$\tau$-avoiding permutations}, that is, the collection of permutations $\sigma$ not containing $\tau$ as a subpattern.

For a general background in first-order logic we refer the reader to \cite{RAUTENBERG}. As already mentioned we will be working with TOTO,  the Theory Of Two Orders.  In this theory, we think of permutations as finite sets of points in the plane with the additional property that no two points lie on the same horizontal or vertical line. Indeed, given such a set of points $\{(x_i, y_i): \ i \in [n]\}$ with $x_1<x_2<\ldots <x_n$, there is a unique associated permutation $\sigma \in\mathfrak{S}_{n}$ with $y_{\sigma^{-1}(1)}<y_{\sigma^{-1}(2)}<\ldots <y_{\sigma^{-1}(n)}$.  So, TOTO is concerned with sets equipped with two linear orders, one for each coordinate,  and isomorphism classes of finite TOTO models are in one-to-one correspondence with permutations, for a proof see e.g.~\cite[Proposition 4]{albert2020two}. 

Formally TOTO consists of a countably infinite number of variables $x_{1},x_{2},\dots$, and a set of relations\footnote{Elsewhere, the set of relations $R$ is often chosen to contain relations $\left\{<_1, <_2\right\}$. However, certain technical steps in our proof become slightly simpler if we work with the orders $\leq_1$ and $\leq_2$ instead, where $x \leq_1y$ if and only if $x <_1 y$ or $x = y$. These two versions of TOTO are equivalent.} $R=\left\{ \leq_{1},\le_{2}\right\} $ corresponding to the linear orders $\leq_1$ and $\leq_2$. An \textit{atomic formula} in TOTO is a formula of the form $x_{i}=x_{j}$, $x_{i}\leq_1 x_{j}$ or $x_i \leq_2 x_j$. The set $\mathrm{FO}$ of \textit{first-order formulae} in TOTO is the least collection of expressions satisfying the following:
 \begin{enumerate}[(i)]
 	\item all atomic formulae lie in $\mathrm{FO}$,
    \item $\bot$ and $\top$ are in $\mathrm{FO}$,
    \item if $p$ is in $\mathrm{FO}$, then $\neg p$ is in $\mathrm{FO}$,
 	\item if $p$ and $q$ are in $\mathrm{FO}$, then so is $\alpha(p,q)$, where $\alpha$ is any of the standard logical connectives of arity $2$ (for example, this means that $p \lor q$, $p \land q$, $p \rightarrow q$ are in $\mathrm{FO}$),
 	\item if $p\in\mathrm{FO}$ and $x$ is a free variable\footnote{Informally, $x$ is a \textit{free variable} in a sentence $p\left(x\right)$ if $p$ does not contain any quantifiers bounding the variable $x$.}  occurring in $p$,
 	then $\forall x: \ p\left(x\right)$ and $\exists x:\ p\left(x\right)$ also belong to $\mathrm{FO}$. 
 \end{enumerate}
 In addition, we say that $p \in \mathrm{FO}$ is a \emph{(first-order) sentence} if $p$ does not contain any free variables. We write $\mathrm{FOS} \subseteq \mathrm{FO}$ for the set of first-order sentences. 
 
 However, in our proofs it is more convenient to work directly with formulae that are written in \textit{disjunctive normal form}. A \textit{literal} $\ell$ is a formula that is either an atomic formula or the negation of an atomic formula. We say that a sentence $\Psi$ is written in \textit{disjunctive normal form }if $\Psi$ is given by
 \[
 Q_{1}x_{1}\dots Q_{k}x_{k}\left(\left(p_{1,1}\land p_{1,2}\land \dots\land p_{1,a_{1}}\right)\lor\left(p_{2,1}\land p_{2,2}\land\dots\land p_{2,a_{2}}\right) \lor \dots\lor\left(p_{b,1}\land p_{b,2}\land \dots\land p_{b,a_{b}}\right)\right),
 \]
 where each $Q_{i}\in\left\{ \forall,\exists\right\} $ is a quantifier,
 $x_{1},\dots,x_{k}$ are variables, and each $p_{i,j}=p_{i,j}\left(\underline{x}\right)$
 is an atomic formula or a negation of an atomic formula using (a subset of the) variables $x_{1},\dots,x_{k}$.

 Note that a sentence $\Psi$ as above is said to have \emph{quantifier depth} (see e.g.~\cite{albert2020two}) at most $k$. It is well-known that every sentence can be converted to one in disjunctive normal form, and thus is logically equivalent to a sentence in disjunctive normal form. 
 \begin{proposition}
 	\label{prop: every FOS can be written in DFN}  Given a first-order sentence $\Psi$, there exists a first-order sentence $\Psi_{DNF}$ that is written in disjunctive normal form, and for which $\Psi$ and $\Psi_{DNF}$ are logically equivalent. 
 \end{proposition}
\noindent For a more comprehensive background on the first-order logic TOTO, see~\cite{albert2020two, albert2024logical}. 

\subsection{Geometrically griddable permutation classes and geometric grid classes}	\label{subsection: geom grid classes} 
In this subsection, we define geometric grid classes and geometrically griddable classes of permutations, which play a central role in the theory of permutation classes with polynomial growth. 

Given a matrix $M$ with entries in $\{-1, 0, 1\}$, we shall consider two specific permutation classes that can be associated to it, namely the class of permutations induced by $M$-compatible gridded figures, denoted by $\text{Grid}\left(M\right)$, and the geometric grid class of $M$, which we denote by $\text{Geom}\left(M\right)$.  As we are primarily interested in  geometric grid classes, we present here a slightly simplified definition of the gridded figures treated by Albert, Ru{\v{s}}kuc and Vatter in~\cite{albert2015inflations}.

Given a matrix $M$, we index the rows from bottom to top, and the columns from left to right, and write $M_{i,j}$ for the entry in the $i^{th}$ column from the left and $j^{th}$ row from the bottom. This non-standard choice of indexing has a geometric motivation: given a unit square grid on the rectangular subset $\left[0,\ell\right]\times\left[0,m\right]$ of the plane $\mathbb{R}^2$, the unit square whose top-right corner is $\left(i,j\right)$ can be associated in a natural way with the entry $M_{i,j}$ of a matrix with $\ell$ columns and $m$ rows. In this spirit, we shall abuse notation and refer to such a matrix $M$ as an $\ell\times m$ matrix in the remainder of this paper, in order to better reflect the geometric nature of the structures we study.

To make this association formal, we let the standard $\ell\times m$-\emph{gridding}  $G_{\ell\times m}$ be the partition of the rectangle $R=\left[0,\ell \right]\times\left[0,m\right]$ into $\ell m$ interior-disjoint unit squares whose corners are the lattice points of $\mathbb{Z}^{2}$. We write $C_{i,j}$ for the unit square (or \emph{cell}) whose top-right corner is $\left(i,j\right)$. 
\begin{definition}
    An $\ell\times m$ \textit{gridded figure }$\left(F,G_{\ell\times m}\right)$ consists of the standard $\ell\times m$ gridding $G_{\ell\times m}$ together with a finite set of points $F\subseteq\left[0, \ell\right]\times\left[0,m\right]$ such that no two points of $F$ lie on the same horizontal or vertical line, and no point of $F$ lies on the grid-lines of $G_{\ell\times m}$.

\end{definition}
Clearly each gridded figure $\left(F,G_{\ell\times m}\right)$ induces a permutation $\sigma$: let $x_1<x_2<\ldots < x_{\vert F\vert}$ and $y_1<y_2<\ldots <y_{\vert F\vert}$ denote the $x$- and $y$- coordinates of points in $F$. Then there is a unique permutation $\sigma \in \mathfrak{S}_{\vert F\vert}$ such that $F=\{(x_i, y_{\sigma(i)}): \ i\in [\vert F\vert ] \}$. 

\begin{definition}
    We say that an $\ell\times m$ gridded figure $\left(F,G_{\ell\times m}\right)$ is \textit{compatible} with a matrix $M$ if $M$ is an $\ell \times m$ matrix with entries in $\{-1,0,1\}$ such that for each $\left(i,j\right)\in\left[\ell \right]\times\left[m\right]$ we have that
    \begin{align*}
    F\cap C_{i,j} \textrm{ is }\left\{\begin{array}{ll}
	   \text{empty } & \text{if }M_{i,j}=0,\\
      \text{a (possibly empty) increasing segment} & \text{if }M_{i,j}=1,\\
	   \text{a (possibly empty) decreasing segment} & \text{if }M_{i,j}=-1.
    \end{array}\right.
    \end{align*}
    Here we say that a set of points $\{(x_i, y_i): i \in [t]\}$ in $\mathbb{R}^2$ with $x_1<x_2<\ldots <x_t$ forms an \emph{increasing segment} if $y_1<y_2<\ldots <y_t$, and a \emph{decreasing segment} if $y_1>y_2>\ldots > y_t$.
\end{definition}

We say that two gridded figures $(F,G)$ and $(F',G)$ with the same gridding $G$ are \emph{equivalent} if they induce the same permutation $\sigma$, and for each $i \in [|\sigma|]$, the points $p_i \in F$ and $p'_i \in F'$ corresponding to the $i$-th point in $\sigma$ belong to the same cells in $G$. We write $\textrm{Grid}^{\#}\left(M\right)$ for the set of equivalence classes of all gridded figures $\left(F,G_{\ell\times m}\right)$ that are compatible with $M$. Further, we let $\textrm{Grid}\left(M\right)$ denote the set of
all permutations $\sigma$ induced by the gridded figures $\left(F,G_{\ell\times m}\right)$ which belong to any equivalence class contained in $\textrm{Grid}^{\#}\left(M\right)$. It is easily checked that $\mathrm{Grid}(M)$ is a permutation class.

The \emph{standard figure} of $M$ is the subset of $\left[0,\ell\right]\times\left[0,m\right]$ consisting of the union of (a) the increasing diagonals $\left(i-1,j-1\right)+\{(t,t): \ t\in (0,1)\}$ in the cells $C_{i,j}$ with $M_{i,j}=1$, and (b) the decreasing diagonals $\left(i-1,j-1\right)+\{(t,1-t): \ t\in (0,1)\}$ in the cells $C_{i,j}$ with $M_{i,j}=-1$. Note that the standard figure thus contains no point inside the cells $C_{i,j}$ for which $M_{i,j}=0$.
\begin{definition}
    Let $M$ be an $\ell\times m$ matrix and let $G$ denote the standard $\ell\times m$ gridding. We write $\textrm{Geom}^{\#}\left(M\right)$ for the set of equivalence classes in $\textrm{Grid}^{\#}\left(M\right)$ which contain a gridded figure $\left(F,G\right)$ as a representative for which $F$ is a finite subset of the standard figure of $M$. The \emph{geometric grid class} of $M$, denoted by $\textrm{Geom}\left(M\right)$, is the collection of all permutations $\sigma$ induced by the equivalence classes in $\textrm{Geom}^{\#}\left(M\right)$. 
\end{definition}
\noindent See Figure~\ref{fig: left greedy} 
for an example of a permutation in $\text{Geom}\left(M\right)$ for some $3\times 4$ matrix $M$. It is clear that $\textrm{Geom}\left(M\right)$ is a permutation class. We say that a permutation class $\mathcal{P}$ is \emph{geometrically griddable} if we have $\mathcal{P}\subseteq\text{Geom}\left(M\right)$
for some matrix $M$ whose entries are in the set $\left\{ -1,0,1\right\} $.

We trivially have $\text{Geom}^{\#}\left(M\right)\subseteq\text{Grid}^{\#}\left(M\right)$
and $\text{Geom}\left(M\right)\subseteq\text{Grid}\left(M\right)$.
However, these inclusions are not always equalities. Indeed,
as pointed out in~\cite{albert2013geometric}, for $M=\left(\begin{array}{rr}
	-1 & 1\\
	1 & -1
\end{array}\right)$, the permutation $2413$ lies in
$\textrm{Grid}(M)$ but not in $\textrm{Geom}(M)$.
However the authors of~\cite{albert2013geometric} showed one does get equality in certain special cases, which turn out to be exactly the ones needed for our purposes. Given an $\ell\times m$ matrix $M$, define the  \emph{row-column graph} $H_{M}$ to be the graph with vertices $V\left(H_{M}\right):=\left[\ell \right]\sqcup\left[m\right]$ and edges consisting of those pairs $ij$ for which $M_{i,j}\neq0$.  
\begin{proposition}[Albert, Atkinson, Bouvel, Ru{\v{s}}kuc and Vatter~\cite{albert2013geometric}]\label{prop: equality grid and geometric grid when HM is a forest}
	\label{thm:Forest theorem}If $H_M$ is a forest, then $\textrm{Grid}^{\#}\left(M\right)=\text{Geom}^{\#}\left(M\right)$ 	(and thus $\textrm{Grid}\left(M\right)=\text{Geom}\left(M\right)$). 
\end{proposition}
Another useful result, due to Huczynska and Vatter~\cite[Corollary 3.4]{huczynska2006grid}, tells us that every permutation class with polynomial growth is contained inside a geometric grid class with highly specific structure.

\begin{proposition}[Huczynska and Vatter]\label{prop: polygrowth implies contained in geometric}
	If $\mathcal{P}$ is a permutation class with polynomial growth, then there exists a signed permutation matrix $M$ such that $\mathcal{P}\subseteq \mathrm{Geom}(M)$.
\end{proposition}
\noindent We recall here that a signed permutation matrix $M$ is a square $m\times m$ matrix with entries from $\{-1, 0,1 \}$ such that there exists a permutation $\tau=\tau_M \in \mathfrak{S}_{m}$ with the property that $M_{i,j}\neq 0$ if and only if $j=\tau(i)$. Note in particular that for such a matrix $M$, the graph  $H_M$ is a matching, and hence the hypothesis in Proposition~\ref{prop: equality grid and geometric grid when HM is a forest} is satisfied.

In most of our arguments, we study permutations via gridded figures. Each gridded figure induces a permutation, as we have already remarked. However, for a given matrix $M$ there may be many fundamentally different non-equivalent $M$-compatible gridded figures giving rise to the same permutation, which   is a problem if we hope to use such gridded figures to count permutations. To resolve this difficulty, Albert, Atkinson, Bouvel, Ru{\v{s}}kuc and Vatter~\cite{albert2013geometric} introduced the \textit{left-greedy order}, a linear order providing a way for comparing two gridded figures which give rise to the same permutation. 
\begin{definition}\label{def: left greedy}
    Given a permutation $\sigma$ and a gridding $G$, we define the \emph{left-greedy order} $<_{LG}$ on equivalence classes of gridded figures with gridding $G$ that induce $\sigma$ by letting $(F, G)<_{LG}(F', G)$ if one of the following occurs:
    \begin{enumerate}[(a)]
	   \item there exists an index $i_0$ such that for all $i<i_0$ and for all $j$, the number of points of $F$ and $F'$ in the $(i,j)$-th cells of their respective griddings are equal, but there are strictly more points of $F$ than of $F'$ in the cells located in the $i_0$-th column;
	   \item there exist indices $i_0$ and $j_0$ so that for every $(i,j)$ with either $i < i_0$, or $i = i_0$ and $j < j_0$,  the number of points of $F$ and $F'$ in the $(i,j)^{th}$ cell of their respective griddings are equal, the number of points of $F$ and $F'$ in the cells located in the $i_0$-th column are equal,     
       but there are strictly more points of $F$ than of $F'$ in the $(i_0,j_0)$-th cell.
    \end{enumerate}
\end{definition}

\noindent Intuitively, $(F, G)<_{LG} (F', G)$ if $(F,G)$ is `more compressed towards the left, or equally left-compressed and more compressed towards the bottom' of their respective griddings -- see also Figure~\ref{fig: left greedy}. It is easily seen that for any fixed gridding $G$, $<_{LG}$ interplays well with our notion of equivalent griddings (e.g.\ if $(F, G)$ and $(F',G)$ are equivalent and $(F',G)<_{LG}(F'', G)$ then $(F,G)<_{LG}(F'',G)$), and thus $<_{LG}$ induces a linear order on $\textrm{Grid}^{\#}\left(M\right)$. For convenience, we write $(F,G) \leq_{LG} (F',G)$ if $(F,G) <_{LG} (F',G)$, or $(F,G)$ and $(F',G)$ are equivalent. 

We remark that $<_{LG}$ is easily seen to be a linear order. Indeed, given a gridded figure $(F,G)$, let $n_{i,j}$ denote the number of points in the $(i,j)$-th cell, and let $n_i$ denote the number of points in the $i$-th column. Then the left-greedy order simply coincides precisely with the descending lexicographic order on the tuples
\[(n_1,n_{1,1},\dots,n_{1,m},n_2,n_{2,1},\dots,n_{2,m},\dots,n_\ell,n_{\ell,1},\dots,n_{\ell,m}).\]
\noindent Observe that if $(F,G)$ and $(F',G)$ are two gridded figures with identical tuples, they are easily seen to be equivalent. Thus the fact that $<_{LG}$ is a linear order follows from the fact that the lexicographic order is a linear order.

\begin{figure}
    \centering
    \includegraphics[scale=0.7]{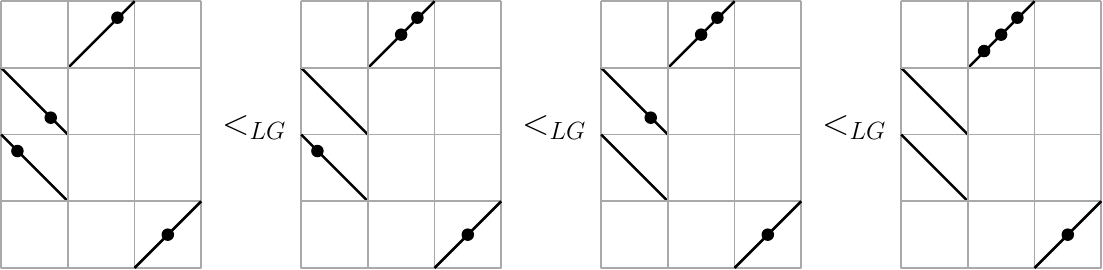}
     \caption{
     An example of gridded figures all inducing the permutation $\sigma=2341$ and their left-greedy ordering -- see Definition~\ref{def: left greedy}. Note these are all $M$-compatible for $M$ with $M_{3,1}=M_{2,4}=1$, $M_{1,2}=M_{1,3}=-1$ and $M_{i,j}=0$ otherwise. (Recall here our non-standard indexing convention for matrices.)
    }\label{fig: left greedy}
\end{figure}

Given a permutation $\sigma \in \mathrm{Grid}(M)$, we write $\sigma^{\#}$ for the $<_{LG}$-minimal  equivalence class of $M$-compatible gridded figures (with the standard gridding for $M$) that induce $\sigma$. In a slight abuse of notation, we may in certain cases also use $\sigma^{\#}$ to denote a representative from that equivalence class. When $M$ is clear from context, we refer to $\sigma^{\#}$ as the \emph{left-greedy representation} of $\sigma$. Note that  for every pair $(i,j)$, all gridded figures 
in $\sigma^{\#}$ have precisely the same number of points lying in the $(i,j)$-th cell of their gridding, and we denote this number by $\vert \sigma^{\#}_{i,j}\vert $.
\begin{definition}
    Given an $\ell\times m$ matrix $M$ with entries in $\{-1,0,1\}$ and another $\ell\times m$ matrix $A$ with entries in $\mathbb{Z}_{\geq 0}\cup\{\infty\}$, we write $\mathrm{Grid}^{\#}(M\vert A)$ for the collection of all equivalence classes that contain a representative gridded figure $(F,G)$ which is $M$-compatible and  satisfies $\vert F\cap C_{i,j}\vert \leq A_{i,j}$ for all pairs $(i,j)\in [\ell]\times [m]$. We denote by $\mathrm{Grid}(M\vert A)$ the collection of all permutations induced by the equivalence classes in $\mathrm{Grid}^{\#}(M\vert A)$, and refer to this permutation class as the \emph{$A$-restricted} grid class of $M$.

    We let $\mathrm{Geom}^{\#}(M\vert A)$ denote the collection of 
those equivalence classes that contain a representative $(F,G)$ which satisfies the conditions described above, and is also a subset of the standard figure of $M$. We let $\mathrm{Geom}(M\vert A)$ denote the set of permutations induced by such equivalence classes. We refer to the permutation class $\mathrm{Geom}(M\vert A)$ as the \emph{$A$-restricted} geometric grid class of $M$.
\end{definition}
We point out that the sets $\mathrm{Grid}^{\#}(M\vert A)$ coincide with the notion of \textit{canonical sets} introduced by Kaiser and Klazar~\cite{kaiser2002growth} (proof of Lemma 4.3). However, as specifying the sequence $A$ is often important for our purposes, we chose to introduce this notation for the canonical sets.

\subsection{Inflations of geometric grid classes}\label{sec: Inflations of geometric grid classes}
Next, we will define inflations of geometric grid classes, a generalization of the geometric grid classes based on  matrices $M$ having entries in $\{-1, 0, 1, \star\}$, satisfying certain additional properties.
\begin{definition}\label{def: suitable matrix}
	A matrix $M$ is \emph{suitable} if $M_{i,j}\in\left\{ -1,0,+1,\star\right\} $
	for all $i,j$, and $M$ satisfies the following conditions:
	\begin{enumerate}[(i)]
		\item if $M_{i,j}=\star$, then  
        $M_{i,j'}=0$ for all $j'\neq j$ and $M_{i',j}=0$ for all $i'\neq i$; 
		\item if there exist $j_1\neq j_2$ such that $M_{i,j_1},M_{i,j_2}\in\left\{ -1,+1\right\} $, 
		then for all $j'\notin\{j_1, j_2\}$ we have $M_{i,j'}=0$ and for all $i'\neq i$ we have $M_{i',j_1}=M_{i',j_2}=0$;
		\item Similarly, if there exists $i_1\neq i_2$ for which we have $M_{i_1,j},M_{i_2,j}\in\left\{ -1, +1\right\} $,
		then $M_{i',j}=0$ whenever $i'\not\in\left\{ i_1,i_2\right\} $, and $M_{i_1,j'}=M_{i_2,j'}=0$
		whenever $j'\neq j$. 
	\end{enumerate}
\end{definition}
\noindent In other words, $M$ is suitable if (i) $\star$-entries of $M$ are the only non-zero entries of $M$ in their row and column; (ii) if some column contains two non-zero entries, then these two entries are the only non-zero entries in that column, and the only non-zero entries in their respective rows; and (iii) if some row contains two non-zero entries, then these two entries are the only non-zero entries in that row and in their respective columns. To each such matrix $M$, we shall associate a certain class of griddable permutations. To define these properly, we need to make a number of technical definitions.

An \emph{interval} inside a permutation $\sigma$ is an integer interval $I=\left[a,b\right]\cap\mathbb{Z}_{>0}$ such that its image $\sigma\left(I\right):=\left\{ \sigma\left(i\right):\ i\in I\right\} $ under $\sigma$ is also an integer interval. If the restriction of $\sigma$ to $I$ is a monotone permutation, then we refer to $I$ as a \emph{monotone interval} of $\sigma$.
\begin{definition}
    Given a permutation $\tau \in \mathfrak{S}_k$ and permutations $\sigma_{1},\dots,\sigma_{k}$, we define the \emph{inflation} $\tau\left[\sigma_{1},\dots,\sigma_{k}\right]$ of $\tau$ by the $k$-tuple of permutations $(\sigma_{1},\ldots,\sigma_{k})$ to be the permutation obtained by replacing each entry $\tau\left(i\right)$ by an interval that is order-isomorphic to the permutation $\sigma_{i}$. 
\end{definition}
\noindent As an example to illustrate this definition, the inflation of $\tau=4213$ by $(132, 12, 1,231)$ is the permutation
\[
4213\left[132,12,1,231\right]=798\,23\,1\,564
\]
where the spaces in the expression on the right hand-side have been added to highlight the transitions between each of the four intervals in the inflation of $\tau$ --- see also Figure~\ref{fig:inflation-example} for a geometric representation of this inflation. 
\begin{figure}
    \centering
    \begin{tikzpicture}[scale=1.2]

        \draw[step=1,black,thin] (0,0) grid (4,4);

        \foreach \x/\y in {0.25/3.25, 0.50/3.75, 0.75/3.50}
            \fill[black] (\x,\y) circle (2pt);

        \foreach \x/\y in {1.33/1.33, 1.67/1.67}
            \fill[black] (\x,\y) circle (2pt);

        \fill[black] (2.50,0.50) circle (2pt);

        \foreach \x/\y in {3.25/2.50, 3.50/2.75, 3.75/2.25}
            \fill[black] (\x,\y) circle (2pt);
    \end{tikzpicture}
    \caption{Example of an inflation $4213\left[132,12,1,231\right]$}
    \label{fig:inflation-example}
\end{figure}
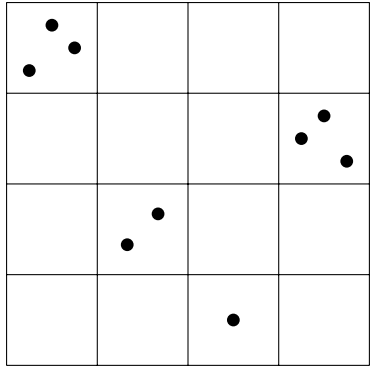
\noindent One may think of inflations in a natural way as permutations that can be decomposed into smaller  permutations.

\begin{definition}
	Given two permutation classes $\mathcal{P}$ and $\mathcal{Q}$, we
	define the \emph{inflation} $\mathcal{P}\left[\mathcal{Q}\right]$ of $\mathcal{P}$ by $\mathcal{Q}$
	to be:
	\[
	\mathcal{P}\left[\mathcal{Q}\right]:=\left\{ \tau\left[\sigma_{1},\dots,\sigma_{k}\right]:\ \tau\in\mathcal{P}_{k}, \ \sigma_{1},\ldots,\sigma_{k}\in\mathcal{Q}\right\} .
	\]
	Given an inflation of the form $\sigma=\tau\left[\sigma_{1},\dots,\sigma_{k}\right]\in\mathcal{P}\left[\mathcal{Q}\right]$ where $\tau \in \mathcal{P}$ and $\sigma_1, \ldots, \sigma_k \in \mathcal{Q}$, we say that $\tau\left[\sigma_{1},\dots,\sigma_{k}\right]$ is a \emph{$\mathcal{Q}$-decomposition} of $\sigma$, and we call $\tau$ a \emph{$\mathcal{P}$-profile} of the expression $\sigma = \tau\left[\sigma_{1},\dots,\sigma_{k}\right]$.
\end{definition}
It is easily checked that for any pair of permutation classes $\mathcal{P}, \mathcal{Q}$, the inflation $\mathcal{P}\left[\mathcal{Q}\right]$ also forms a permutation class. Also, note that a permutation $\sigma\in\mathcal{P}\left[\mathcal{Q}\right]$ may have multiple \emph{$\mathcal{Q}$}-decompositions and $\mathcal{P}$-profiles; there is no uniqueness implied in our definition. However, for reasons of convenience we include the notion of $\mathcal{P}$-profiles into inflations, instead of permutations themselves, as in our applications, we often care about a particular choice of inflation, and the profile of that particular inflation.  
Despite having non-uniqueness of \emph{$\mathcal{Q}$}-decompositions, one can get around this by using a decomposition analogue of the left-greedy ordering.
\begin{definition}
    Let $\mathcal{P}$ and $\mathcal{Q}$ be two permutation classes. Let $\tau\left[\sigma_{1},\dots,\sigma_{k}\right]$ and $\tau'\left[\sigma'_{1},\dots,\sigma'_{k'}\right]$ be two $\mathcal{Q}$-decompositions of the same permutation in $\mathcal{P}[\mathcal{Q}]$. We say that $\tau\left[\sigma_{1},\dots,\sigma_{k}\right]<_L\tau'\left[\sigma'_{1},\dots,\sigma'_{k'}\right]$ if one of the following holds:
    \begin{enumerate}[(a)]
        \item $\vert \tau \vert =k <k' =\vert \tau'\vert$;
        \item $k=k'$ and there exists $i_0\in [k]$ such that for all $i\in [i_0-1]$ we have $\vert \sigma_{i}\vert =\vert \sigma'_{i} \vert$ but $\vert \sigma_{i_0}\vert >\vert \sigma'_{i_0}\vert$.
    \end{enumerate}
\end{definition}
\noindent As an example for $\mathcal{Q}=\mathrm{Av}(12)$ and $\mathcal{P}$ being the set of all permutations, the $\mathcal{Q}$-decompositions of the permutation $1432$ satisfy \[ 12[1,321]<_L 132[1,21,1] <_L 132[1,1,21] <_L 1432[1,1,1,1]. \]
\noindent It is easy to verify that $<_{L}$ induces a  linear order. For every $\sigma \in \mathcal{P}[\mathcal{Q}]$, it is easily seen that there exists a unique $<_L$-minimal $\mathcal{Q}$-decomposition of $\sigma$. We refer to this $<_L$-minimal decomposition as the \emph{left-greedy decomposition}. Its uniqueness is a key fact we use in our arguments.

\begin{definition}
Given a permutation class $\mathcal{P}$, we define the sequence of permutation classes $\left(\mathcal{P}^{\left[d\right]}\right)_{d\in \mathbb{Z}_{\geq1}}$ 
inductively by setting $\mathcal{P}^{\left[1\right]}:=\mathcal{P}$
and $\mathcal{P}^{\left[d+1\right]}:=\mathcal{P}\left[\mathcal{P}^{\left[d\right]}\right]$. For notational convenience, we also define $\mathcal{P}^{\left[0\right]}:=\mathfrak{S}_1$, the family consisting of the single permutation on $1$ element.	
\end{definition}
\noindent In our arguments, we shall be naturally led to consider two particular families of inflations, namely inflations of the identity permutation and inflations of the monotone decreasing permutation.
\begin{definition}
	Given permutations $\sigma_1, \ldots, \sigma_k$, their \emph{direct sum} is the permutation $\sigma_1 \oplus \sigma_2\oplus \cdots \oplus \sigma_k:= 123\ldots k[\sigma_1, \sigma_2, \ldots , \sigma_k]$. Similarly, their \emph{skew sum} is the permutation $\sigma_1 \ominus \sigma_2\ominus \cdots \ominus \sigma_k:= k(k-1)(k-2)\ldots 21[\sigma_1, \sigma_2, \ldots , \sigma_k]$.
\end{definition} 
\begin{definition}
    The \emph{increasing oscillating sequence} is the integer sequence $(s_n)_{n\in \mathbb{Z}_{\geq 0}}$ defined by $s_{n}=n+4$ if $n$ is even and $s_n=n$ otherwise; thus its first few terms are $4, 1, 6, 3, 8, 5, \ldots$. We let $F_k:=\{(i, s_{i-1}): \ i \in [k]\}$, and denote by $\omega_k$ the permutation associated to $F_k$. We define the \emph{decreasing oscillating sequence} similarly by taking $D_k := \{(i, -s_{i-1}): \ i \in [k]\}$, and denoting by $\varpi_k$ the permutation associated to $D_k$.
    
    We let $\mathcal{O}_k$ denote the finite permutation class consisting of all subpermutations of $\omega_k$ and $\varpi_k$ and let $\widetilde{\mathcal{O}}_{k}$ denote the permutation class that is the union of $\mathcal{O}_k$, and the set containing all monotone increasing and decreasing permutations of finite length. 
\end{definition}
\begin{definition}
    Given  a suitable $\ell\times m$ matrix $M$, the standard $\ell \times m$ gridding $G$ and integers $d,k\in \mathbb{Z}_{\geq 0}$, we define $\mathrm{Geom}^{\#}(M)_{d,k}$ to be the collection of all equivalence classes which contain a gridded figure $(F,G)$ satisfying the following:
    \begin{enumerate}[(i)]
        \item if $M_{i,j}=0$ then $F\cap C_{i,j}=\emptyset$;
        \item if $M_{i,j}=+1$, we have $F\cap C_{i,j} \subseteq \left\{(i-1, j-1)+(t,t): \ t\in (0,1)\right\}$, while if $M_{i,j}=-1$ we have $F\cap C_{i,j} \subseteq \left\{(i-1, j-1)+(t,1-t): \ t\in (0,1)\right\}$;
       \item if $M_{i,j}=\star$, then $F\cap C_{i,j}$ induces a permutation from ${\widetilde{\mathcal{O}}_{k}}^{[d]}$.
    \end{enumerate}
    Finally we let $\mathrm{Geom}(M)_{d,k}$ denote the collection of all permutations induced by an equivalence class in $\mathrm{Geom}^{\#}(M)_{d,k}$.
\end{definition}
\noindent As Definition~\ref{def: left greedy} of the left-greedy order $<_{LG}$ was given for gridded figures in general, without any reference to a matrix, it applies in this setting also, and given $\sigma \in \mathrm{Geom}(M)_{d,k}$ we can again write $\sigma^{\#}$ for the $<_{LG}$-minimal  equivalence class in $\mathrm{Geom}^{\#}(M)_{d,k}$ that induces $\sigma$. With these notions in place, we can now state the fundamental theorem of Vatter~\cite[Theorem 5.4]{vatter2011small} on the structure of small permutation classes; we state it here in its reformulation by Albert, Ru{\v{s}}kuc and Vatter~\cite[Theorem 8.4]{albert2015inflations}.
\begin{proposition}[Vatter]\label{prop: small implies contained inside Geom (M) dk}
    Let $\cP$ be a small permutation class. Then there exists a suitable matrix $M$ and integers $d,k\in \mathbb{Z}_{\geq 0}$ such that $\cP\subseteq \mathrm{Geom}(M)_{d,k}$.
\end{proposition}

\section{The polynomial growth regime}	


\subsection{0-1 laws and  infinite TOTO models for small geometric grid classes}

We begin our investigations of the polynomial growth regime by proving Theorem~\ref{theorem: infinite model}. We shall in fact prove a slightly more general version of Theorem~\ref{theorem: infinite model}, Theorem~\ref{theorem: infinite model, technical version}, for reasons that will become clear in subsequent sections, where we shall combine it with results on the structure of general permutation classes with polynomial growth to prove Theorem~\ref{thm:poly growth}.

We now set things up for the statement of Theorem~\ref{theorem: infinite model, technical version}. Let $M$ be a signed $m\times m$ permutation matrix, and let $\tau \in \mathfrak{S}_m$ be the associated permutation. Given a non-empty set $I \subseteq [m]$, let $\mathbf{a}\in {\left(\mathbb{Z}_{\geq 0}\right)}^{[m]\setminus I}$ be a non-negative integer-valued vector with entries indexed by $[m]\setminus I$.

Set $F^{\infty}$ to be the infinite collection of points in $[0,m]^2$ given by the intersection of the standard figure of $M$ with the infinite collection of vertical lines $\{ x= i-1/2 \pm 1/n:\ n \in \mathbb{Z}_{\geq 3}\}_{i\in I}$ and  the $\sum_{i\in [m]\setminus I}a_i$ lines $\{ x= i-1 +\frac{n}{a_i+1}: \ n\in [a_i]\}_{i\in [m]\setminus I}$. (Thus the intersection $(F^{\infty})_{i,\tau(i)}$ of $F^{\infty}$ with the cell $C_{i,\tau(i)}$ is infinite if $i\in I$, and has size precisely $a_i$ otherwise.) We let $\mathfrak{P}$ denote the permutation induced by the infinite point-set $F^{\infty}$.  
\begin{claim}
   The set $F^{\infty}$ has the following key property: for every point $p\in F^{\infty}$, either $p$ is the $\leq_1$-least or $\leq_1$-greatest element of $F^{\infty}$, or it has a unique $<_1$-predecessor and a unique $<_1$-successor in $F^{\infty}$. 
\end{claim}
\begin{proof} Indeed, we observe first of all that for every $i \in [m]$, the intersection  of $F^{\infty}$ with the infinite rectangular strip $R_i=\{(x,y): \ i-1\leq x\leq i\}$ is either empty (if $i\in [m]\setminus I$ and $a_i=0$) or has a unique $\leq_1$-minimal and a unique $\leq_1$-maximal element: if $i\in [m]\setminus I$, this is clear from the fact $F^{\infty}\cap R_i$ contains a finite number $a_i$ of points , while if $i \in I$ we have that the $\leq_1$-extremal points in $F^{\infty}\cap R_i$ are the unique points of $F^{\infty}$ on the lines $x=i-5/6$ and $x=i-1/6$.

If $i\in [m]\setminus I$, this easily implies all of the finitely many points in $F^{\infty}\cap R_i$ either have a unique $\leq_1$-predecessor and unique $<_1$-successor or are $\leq_1$-extremal in $F^{\infty}$. On the other hand if $i\in I$, then every point $p$ in $F^{\infty}\cap R_i$ has an $x$-coordinate of the form either $i-5/6$ (in which case it is $\leq_1$-minimal in $F^{\infty}\cap R_i$ and its $\leq_1$-successor in $F^{\infty}$ has $x$-coordinate $i-3/4$) or $i -1/2-1/n$ for some $n\geq 4$ (in which case its $\leq_1$-predecessor and  $\leq_1$-successor in $F^{\infty}$ have $x$-coordinates $i-1/2-1/(n-1)$ and $i-1/2-1/(n+1)$ respectively), or $i-1/2+1/n$ for some $n\geq 3$ (which we can deal with similarly). Together with the existence of unique $\leq_1$-extremal points in $F^{\infty}\cap R_i$ for every $i\in [m]$, this establishes our claim.
\end{proof}
Fix some arbitrary constant $T_0\geq 3$, and let $\mathcal{F}_{I, \mathbf{a}}= \mathcal{F}_{I, \mathbf{a}, T_0}$ denote the collection of permutations $\sigma$ in $\mathrm{Geom}(M)$ such that $\vert \sigma^{\#}_{i,\tau(i)}\vert$ is equal to $a_i$ if $i\in [m]\setminus I$, and is at least $T_0$ otherwise (we suppress the dependence on $T_0$ when $T_0$ is clear from context). Note that $\mathcal{F}_{I, \mathbf{a}}$ is not a permutation class (it is not closed under taking patterns), and that by properties of the left-greedy order, $\mathcal{F}_{I, \mathbf{a}}$ is either empty or infinite.  We let $\left(\mathcal{F}_{I, \mathbf{a}}\right)_n$ denote the order $n$ permutations in $\mathcal{F}_{I, \mathbf{a}}$. The main result in this section is the following theorem, which states that $\mathfrak{P}$ is an infinite asymptotically TOTO-first order indistinguishable model for $\mathcal{F}_{I, \mathbf{a}}$.
\begin{theorem}\label{theorem: infinite model, technical version}
If $\mathcal{F}_{I, \mathbf{a}}=\mathcal{F}_{I, \mathbf{a}, T_0}$ is infinite then the permutation $\mathfrak{P}$ is an infinite model for the set of TOTO axioms defining $\mathcal{F}_{I, \mathbf{a}}$, and is asymptotically TOTO first-order indistinguishable from $\mathcal{F}_{I, \mathbf{a}}$; in particular, for any fixed $\Psi\in \mathrm{FOS}$, and for $\sigma_n$ drawn uniformly at random from $\left(\mathcal{F}_{I, \mathbf{a}}\right)_n$ we have
\[  \mathbb{P}( \sigma_n \models \Psi) = \begin{cases} 1-O(n^{-1}) & \mbox{ if $\mathfrak{P}\models \Psi$}\\
O(n^{-1}) & \mbox{otherwise}.\end{cases}\]
\end{theorem}

\noindent Let $k$ be fixed. Consider a first-order TOTO sentence $\Psi\in \mathrm{FOS}$ of the form $Q_1x_1 \ \ldots Q_k x_k \Psi_1(\mathbf{x})$, where $Q_i\in \{\forall, \exists\}$ for all $i\in [k]$, $\mathbf{x}=(x_1, \ldots, x_k)$ and $\Psi_1(\mathbf{x})$ is the quantifier-free part of $\Psi$. In the key lemma below, which is the heart of the proof of Theorem~\ref{theorem: infinite model, technical version}, we show that $F^{\infty}$ satisfies $\Psi$ if and only if $F'$ also does, where $F'$ is any $M$-compatible finite gridded figure having a sufficiently large (but finite) number of points in the cell $C_{i,\tau(i)}$ if $i\in I$, and having exactly $a_i$ points  in that cell otherwise. In other words, the first order TOTO sentence $\Psi$ cannot tell $F^{\infty}$ and $F'$ apart.
\begin{lemma}\label{lemma: infinite model ER game}
Let $F'$ be an $M$-compatible finite gridded figure such that $F'_{i,\tau(i)}:= F'\cap C_{i,\tau(i)}$ satisfies $\vert F'_{i,\tau(i)}\vert \geq 2^{k}$ if $i\in I$, and $\vert F'_{i,\tau(i)}\vert =a_i$ if $i\in[m]\setminus I$. Then $F^{\infty}\models \Psi$ if and only if $F'\models \Psi$.
\end{lemma}
\noindent Lemma~\ref{lemma: infinite model ER game} is proved by playing an instance of the \emph{Ehrenfeucht--Fra\"iss\'e game} from first-order logic, which we now recall. In our setting, the $k$-round Ehrenfeucht--Fra\"iss\'e game is a two-player game played in $k$ rounds on a pair of gridded figures $(P, G_{m\times m})$ and $(Q, G_{m\times m})$ (we omit $G_{m\times m}$ from the notation in what follows). 
 The two players are called Spoiler and Duplicator. In round~$i$, Spoiler picks either a point $p_i\in P$ or a point $q_i\in Q$. If Spoiler picked $p_i\in P$, then Duplicator responds by picking some $q_i\in Q$, while if Spoiler picked $q_i\in Q$, then Duplicator responds by picking some $p_i \in P$, and round $i$ concludes. Thus at the end of round $i$, the two players have together specified points $p_i\in P$ and $q_i \in Q$. The game ends at the conclusion of round $k$. Duplicator wins if the ordered $k$-tuples of points $(p_1,\ldots,p_k)$ from $P$  and $(q_1,\ldots,q_k)$ from $Q$ satisfy the same internal TOTO relations --- that is, for every $i\in \{1,2\}$ and $j,j'\in [k]$, $p_j\leq_i p_{j'}$ holds if and only if $q_j\leq_i q_{j'}$ holds --- and otherwise Spoiler wins.

It is well-known, see e.g.~\cite{MR126370}, that if Duplicator has a winning strategy for a given $k$ then the two permutations induced by the point-sets $P$ and $Q$ satisfy the same set of first-order formulae of quantifier depth $k$ --- and in particular $P\models \Psi$ if and only if $Q\models \Psi$, whence the connection to Lemma~\ref{lemma: infinite model ER game}. The existence of such a winning strategy for Duplicator is folklore; for the sake of completeness and readers less familiar with Ehrenfeucht--Fra\"iss\'e games, we include a proof of this (and thus of Lemma~\ref{lemma: infinite model ER game}) in the appendix.
\begin{corollary}\label{cor: zero one law when enough points}
Let $\Psi$ be as above. Either every $\sigma \in \mathcal{F}_{I, \mathbf{a}}$ with $\vert \sigma^{\#}_{i,\tau(i)}\vert\geq 2^k$ for all $i\in I$ satisfies $\sigma \models \Psi$ or none of them does.\  \qedsymbol
\end{corollary}
\noindent The next ingredient we need in the proof of Theorem~\ref{theorem: infinite model, technical version} is the following statement regarding the number of permutations in $\left(\mathcal{F}_{I, \mathbf{a}}\right)_n$.
\begin{lemma}\label{lemma: balls and bins}
Let $ k>0$ be fixed. We have $\vert \left(\mathcal{F}_{I, \mathbf{a}}\right)_n\vert= \binom{n}{\vert I\vert -1} +O\left(n^{\vert I\vert -2}\right)$. Further, all but an $O(n^{-1})$ proportion of $\sigma \in \left(\mathcal{F}_{I, \mathbf{a}}\right)_n$ satisfy $\vert \sigma^{\#}_{i,\tau(i)}\vert\geq 2^{ k}$.
\end{lemma}
\begin{proof}
The proof is similar to that of~\cite[Lemma 4.3]{kaiser2002growth} (see also the discussion between Lemma~4.3 and Theorem~4.4 in that paper), so we only give an outline here. A permutation $\sigma \in \mathcal{F}$ is entirely determined by the number of points $\vert \sigma^{\#}_{i,\tau(i)}\vert$ of $\sigma^{\#}$ belonging to each of the cells $C_{i,\tau(i)}$, $i\in [m]$. For $i\in [m]\setminus I$, this number is fixed and equal to $a_i$. A standard balls-and-bins counting argument shows that the number of ways of distributing the remaining $n'=n-\sum_{i\in [m]\setminus I}a_i$ points of $\sigma^{\#}$ over the $\vert I\vert$ remaining cells is $\binom{n'+\vert I\vert -1}{\vert I\vert-1}=\binom{n}{\vert I\vert -1}+O(n^{\vert I\vert-2})$. Further, a balls-and-bins argument combined with a union bound implies that all but at most a $O(n^{-1})$ proportion of these partitions assign at least $2^k$ points to each of the cells $C_{i,\tau(i)}$, $i\in I$. The lemma follows.
\end{proof}
\begin{proof}[Proof of Theorem~\ref{theorem: infinite model, technical version}]
Since $\mathfrak{P}$ is a subset of the standard figure of $M$,  contains exactly $a_i$ points in cell $C_{i,\tau(i)}$ for each $i\in [m]\setminus I$ and at least $T_0$ points in cell $C_{i,\tau(i)}$ for each $i\in I$, $\mathfrak{P}$ exhibits precisely the same finite patterns as $\cP=\mathrm{Geom}(M)$. It is thus an infinite model for the set of TOTO axioms defining $\mathcal{F}_{I, \mathbf{a}}$.

Consider now a fixed $\Psi\in \mathrm{FOS}$ of the form $Q_1x_1 \ldots Q_kx_k \Psi_1(\mathbf{x})$. 
By Lemma~\ref{lemma: balls and bins}, there is at most a $O(n^{-1})$ proportion of $\sigma \in \left(\mathcal{F}_{I, \mathbf{a}}\right)_n$ such that $\vert \sigma^{\#}_{i,\tau(i)}\vert<2^{ k}$ for some $i \in I$. By Lemma~\ref{lemma: infinite model ER game}, all other permutations $\sigma \in \left(\mathcal{F}_{I, \mathbf{a}}\right)_n$ satisfy $\sigma \models \Psi$ if and only if $\mathfrak{P}\models \Psi$. The theorem follows.
\end{proof}
\noindent We now get Theorem~\ref{theorem: infinite model} as an easy corollary of Theorem~\ref{theorem: infinite model, technical version} and the properties of the left-greedy order.
\begin{proof}[Proof of Theorem~\ref{theorem: infinite model}]
Let $M$ be a signed $m\times m$ permutation matrix, and let $\tau\in \mathfrak{S}_m$ be the associated permutation. Assume without loss of generality that $M$ does not contain as a proper submatrix a signed permutation matrix $M'$ such that $\mathrm{Geom}(M)=\mathrm{Geom}(M')$. By properties of the left-greedy order, this is easily seen to imply that for every fixed $T_0>0$, there exist $\sigma\in \mathrm{Geom}(M)$ with $\vert \sigma^{\#}_{i,\tau(i)}\vert \geq T_0$ for all $i\in [m]$. Setting $I=[m]$, $\mathbf{a}=()$ and $\mathcal{P}=\mathrm{Geom}(M)$, Theorem~\ref{theorem: infinite model} follows immediately from Theorem~\ref{theorem: infinite model, technical version}.
\end{proof}
\begin{remark}
Let $M$ be a signed permutation matrix not containing a proper signed permutation submatrix $M'$ with $\mathrm{Geom}(M)=\mathrm{Geom}(M')$. Given some gridded configuration $\sigma \in \mathrm{Geom}(M)$, any non-empty cell $C$ of $M$, and any integer $k\in \mathbb{Z}_{>0}$, it is fairly easy to construct a first-order sentence in TOTO which can be used to `identify' both the $k$ $\leq_1$-least and the $k$  $\leq_1$-greatest points in $\sigma^{\#}\cap C$, and to ensure that these sets do not overlap.

 Thus any infinite point-set ${(F')}^{\infty}$ providing a putative infinite model for $\mathrm{Geom}(M)$ must contain in every cell $C$ an infinite $<_1$-increasing sequence of points $(x_n)_{n\in \mathbb{Z}_{>0}}$ and an infinite  $<_1$-decreasing sequence of points $(y_n)_{n\in \mathbb{Z}_{>0}}$ with $x_{n}<_1y_{n'}$ for all $n,n'\in \mathbb{Z}_{>0}$. In particular, up to translations ${(F')}^{\infty}$ must contain as a subset a copy of the infinite point-set $F^{\infty}$ we constructed.
\end{remark}
\noindent We showed in Theorem~\ref{theorem: infinite model}, which corresponded to the case $I=[m]$, $\mathbf{a}=()$ of Theorem~\ref{theorem: infinite model, technical version}, that geometric grid classes have infinite models which with high probability are indistinguishable from random permutations of large order, and in particular must satisfy an asymptotic $0$-$1$ law.

What would happen instead if we considered a permutation subclass of $\mathrm{Geom}(M)$ containing the union of several families of the form $\mathcal{F}_{I, \mathbf{a}}$ for some proper non-empty subsets $I\subsetneq[m]$ and non-empty $\mathbb{Z}_{\geq 0}$-valued vectors $\mathbf{a}$?  Clearly, this will yield new permutation classes of polynomial growth, with a significantly more complicated structure.  Depending on the relative orders of their constituent parts, such classes need no longer satisfy a $0$-$1$ law; however, as we shall show, there can only be a finite set of limiting probabilities for TOTO sentences in this more general setting --- and what is more, this more general setting in fact covers all permutation classes with polynomial growth.

In order to establish this, we shall use the rich theory on the structure of small permutation classes, which is the topic of the next subsection, Section~\ref{subsection: structure of polygrowth classes}, before essentially revisiting the arguments from this subsection but in the context of general permutation classes with polynomial growth in Section~\ref{subsection: logical limit laws poly growth}.

\subsection{The structure of permutation classes with polynomial growth}\label{subsection: structure of polygrowth classes}
In this section, we introduce the necessary structural results on permutation classes with polynomial growth from \cite{kaiser2002growth}, and restate them within the terminology used in this paper. Throughout this section, let $\mathcal{P}$ be a permutation class with polynomial growth. By Proposition~\ref{prop: polygrowth implies contained in geometric}, there exists an $m\times m$ signed permutation matrix $M$ such that $\mathcal{P}\subseteq \mathrm{Geom}(M)$. Our main goal is to understand, on an abstract level, for which matrices $A$ the reverse inclusion $\mathrm{Geom}(M\vert A)\subseteq \mathcal{P}$ holds. To this end, let $\tau=\tau_M$ denote the permutation of $[m]$ associated with the (signed) permutation matrix~$M$. 

\begin{definition}
A set $I\subseteq [m]$ is \emph{popular} if for all $t\in \mathbb{Z}_{>0}$ there exists a permutation $\sigma \in \mathcal{P}$ such that $\vert \sigma^{\#}_{i, \tau(i)}\vert \geq t$ for all indices $i\in I$.
\end{definition}
\begin{remark}
	Since $\mathcal{P}_n\neq \emptyset$ for all $n$ and $\mathcal{P} \subseteq \mathrm{Geom}(M)$, it must be the case that there is at least one non-empty popular subset of $[m]$.
\end{remark}
\begin{definition}
Given a set $I$, we say that a sequence $\mathbf{a}\in (\mathbb{Z}_{\geq 0})^{[m]\setminus I}$ is \emph{$I$-admissible} if for every $t\in \mathbb{Z}_{>0}$ there exists a permutation $\sigma \in \mathcal{P}$ such that $\vert \sigma^{\#}_{i, \tau(i)}\vert \geq t$ for all indices $i\in I$ and $\vert \sigma^{\#}_{i, \tau(i)}\vert =a_i$ for all other indices $i\in [m]\setminus I$. We denote the (possibly empty) collection of $I$-admissible sequences by $A_I$.
\end{definition}
\begin{remark} Clearly, $A_I\neq \emptyset$ implies that $I$ is a popular set. Also, note that $A_I$ (viewed as a collection non-negative integer sequences) need not be a down-set, since the left-greedy order is not necessarily stable under the deletion of points from gridded figures. 
\end{remark}	

\begin{lemma}\label{lemma: inclusion-maximal popular sets}
Let $I\subseteq [m]$ be an inclusion-maximal popular set. Then the set $A_I$ is finite and non-empty.
\end{lemma}
\begin{proof}
The proof follows analogously from the proof of Lemma 4.3 in \cite{kaiser2002growth}. 
\end{proof}

For each non-empty set $I\subseteq [m]$ which is not popular, 
let $t_I$ be the minimal integer such that
 there is no permutation $\sigma \in \mathcal{P}$ with $\vert \sigma^{\#}_{i, \tau(i)}\vert \geq t_I$ for all $i \in I$. Further, for every non-empty set $I\subseteq [m]$ which is popular and inclusion-maximal, let $t_I$ be the minimal positive integer such that 
  for every permutation $\sigma \in \mathcal{P}$, if $\vert \sigma^{\#}_{i, \tau(i)}\vert \geq t_I$ for all $i\in I$ then the sequence $(\vert \sigma^{\#}_{i, \tau(i)}\vert)_{i\in [m]\setminus I}$ must belong to $A_I$. (This follows from a standard compactness argument, we omit the details.)  Set 
\begin{align}\label{eq:T}T:=1+ \max_{I\subseteq [m]: \ \textrm{$I$ not popular}}t_I \ + \max_{I\subseteq [m]: \ \textrm{$I$ $\subseteq$-maximal \& popular}}\left(t_I +\max_{\mathbf{a}\in A_I}\| \mathbf{a}\|_{\infty} \right).\end{align} 
Given an inclusion-maximal popular set $I$ and an $I$-admissible sequence $\mathbf{a}\in A_I$, let $\mathcal{P}_{I, \mathbf{a}}$ denote the collection of all permutations $\sigma\in \mathcal{P}$ with 
$\vert \sigma^{\#}_{i, \tau(i)}\vert \geq T$  for all $i\in I$ and $\vert \sigma^{\#}_{i, \tau(i)}\vert =a_i$ for all other indices $i\in [m]\setminus I$. 
\begin{lemma}\label{lemma: disjointness}
Let $I, I'$ be inclusion-maximal popular sets, and let $\mathbf{a}\in A_I$, $\mathbf{a'}\in A_{I'}$ be $I$- and $I'$-admissible sequences respectively. If $(I, \mathbf{a})\neq (I', \mathbf{a'})$, then $\mathcal{P}_{I, \mathbf{a}}\cap \mathcal{P}_{I', \mathbf{a}'}=\emptyset$.
\end{lemma}
\begin{proof}
If $I=I'$ and $\mathbf{a}\neq \mathbf{a}'$ (which in particular implies $I\neq [m]$), then $\mathcal{P}_{I, \mathbf{a}}$ and $\mathcal{P}_{I, \mathbf{a}'}$ are clearly disjoint (consider the number of points of a putative $\sigma^{\#} \in \mathcal{P}_{I, \mathbf{a}}\cap \mathcal{P}_{I, \mathbf{a'}}$ falling in cell $C_{i,\tau(i)}$ for some $i\notin I$ with $a_i\neq a'_i$).

We may thus assume $I\neq I'$. Suppose $\sigma \in \mathcal{P}_{I, \mathbf{a}}\cap \mathcal{P}_{I', \mathbf{a}'}$. By inclusion-maximality, we have $I'\setminus I\neq \emptyset$. Now for every $i\in I'\setminus I$, we have by definition of $\mathcal{P}_{I, \mathbf{a}}$ that $\vert \sigma^\#_{i, \tau(i)}\vert =a_i<T$, contradicting the fact that $\sigma \in \mathcal{P}_{I', \mathbf{a}'}$.
\end{proof}
\noindent The size of  $\left(\mathcal{P}_{I, \mathbf{a}}\right)_n:= \mathcal{P}_n\cap \mathcal{P}_{I, \mathbf{a}}$  can be estimated by standard combinatorial arguments.
\begin{lemma}\label{lemma: size of inclusion-max pop sets}
    Let $I$ be an inclusion-maximal popular set. Then $\left\vert\left(\mathcal{P}_{I,\mathbf{a}}\right)_{n}\right\vert=\binom{n}{\left\vert I\right\vert-1}+O\left(n^{\left\vert I\right\vert -2}\right)$.
\end{lemma}
\begin{proof}
Follows directly from the proof of Lemma~\ref{lemma: balls and bins}.
\end{proof}
Given an inclusion-maximal popular set $I$ and an $I$-admissible sequence $\mathbf{a}\in A_I$, let  $A_{\mathbf{a}}$ denote the $m \times m$ matrix with $A_{i,\tau(i)}=\left\{ \begin{array}{ll}a_i & i\notin I\\
	\infty &  i\in I
\end{array}\right.$ and $A_{i,j}=0$ for all $i\in [m]$ and $j\neq \tau(i)$. Then it follows from the definition of popular sets and the fact that $\mathcal{P}$ is a permutation class that
\begin{align}\label{eq: contain restricted geometric grid}
\mathcal{P}_{I,\mathbf{a}}\subseteq \mathrm{Geom}(M\vert A_{\mathbf{a}}) \subseteq \mathcal{P}.
\end{align}
Now let $\mathcal{Q}:= \bigcup_{I: \textrm{ $\subseteq$-maximal popular}}\  \left( \bigcup_{\mathbf{a}\in A_I}\mathcal{P}_{I, \mathbf{a}}\right)$, and let $P\in [m]$ denote the size of a largest popular set. We say that a popular set is \textit{size-maximal} if its size equals $P$, and we note that size-maximal popular sets are trivially inclusion-maximal. It follows from Lemmas~\ref{lemma: disjointness} and~\ref{lemma: size of inclusion-max pop sets} that
\begin{align}\label{eq: size of Q}
	\vert\mathcal{Q}_n\vert = b\binom{n}{P-1}+O(n^{P-2})
\end{align}
where $b:=\sum_{I\ \textrm{size-maximal}} \vert A_I\vert$ is strictly positive (by Lemma~\ref{lemma: inclusion-maximal popular sets}). As we show below, the set $\mathcal{Q}_n\subseteq \mathcal{P}_n$ contains most of $\mathcal{P}_n$.
\begin{lemma}\label{lemma: approx P by Q}
    We have $\vert \left(\mathcal{P}\setminus \mathcal{Q}\right)_n\vert =O(n^{P -2})$.
\end{lemma}
\begin{proof}
By construction of $T$, if $\sigma \in \left(\mathcal{P}\setminus \mathcal{Q}\right)_n$, then at most $P-1$ of the cells of its left-greedy decomposition $\sigma^{\#}$ can contain $T$ points. Indeed, otherwise $\sigma$ would have to belong to $\mathcal{P}_{I, \mathbf{a}}$ for some size-maximal popular set $I$ and $\mathbf{a}\in A_I$. By a standard balls-and-bins argument and a union bound, it follows that $\vert \left(\mathcal{P}\setminus \mathcal{Q}\right)_n\vert \leq \binom{m}{P-1}T^{m-P+1 }\binom{n+P-3}{P -2}=O(n^{P-2})$.
\end{proof}
\noindent One may observe that taken together, \eqref{eq: size of Q} and Lemma~\ref{lemma: approx P by Q} have the following instant corollary:
\begin{align}\label{eq: size of P}
\vert\cP_n\vert =b\binom{n}{P-1}+O(n^{P-2})       
\end{align}
\noindent We note, however, that in their 2002 paper Kaiser and Klazar~\cite{kaiser2002growth} obtained the much stronger statement that for every permutation class $\mathcal{P}$ with polynomial growth, there exist strictly positive integers $P_1, P_2, n_0$, and integers $a_{i,j}$ for $0\leq i \leq P_1$ and $0\leq j \leq P_2$ such that for every $n \geq n_0$ we have 
\[\vert \mathcal{P}_n\vert = \sum_{j= 0}^{P_2}\sum_{i=0}^{P_1}a_{i,j} \binom{n-i}{j}.\]
Thus~\eqref{eq: size of P} is by no means new, but it follows as a by-product of the new work in this subsection, namely the combinatorial interpretation we have derived for the leading order coefficient $b$ in~\eqref{eq: size of Q} in terms of inclusion-maximal popular sets of maximal order. This combinatorial interpretation, together with the facts we established  regarding the family $\mathcal{Q}$, namely that it covers all but a $O(n^{-1})$ proportion of $\mathcal{P}_n$ and that it decomposes into a disjoint union of well-behaved subfamilies $\mathcal{P}_{I,\boldsymbol{a}}$, is  crucial for our work in the next subsection. Since our aims differ from~\cite{kaiser2002growth} we have also deviated slightly from their terminology and adapted the notation  to our purposes.


\subsection{General logical limit laws in the polynomial growth regime}\label{subsection: logical limit laws poly growth}
Let $T'\in \mathbb{Z}_{\geq 0}$. Given an inclusion-maximal popular set $I$ and an $I$-admissible sequence $\mathbf{a}\in A_I$, let $\mathcal{P}^{T'}_{I, \mathbf{a}}$ denote the collection of all permutations $\sigma \in\mathcal{P}_{I, \mathbf{a}}$ with  $\vert \sigma^{\#}_{i, \tau(i)}\vert \geq T+T'$ for each $i\in I$, where $T$ is the constant defined in equation~\eqref{eq:T}. Our main tool in proving Theorem~\ref{thm:poly growth} is the following lemma, which follows easily from Lemma~\ref{lemma: infinite model ER game}. 
\begin{lemma}\label{lemma: Psi 0-1 satisfied by super restricted geom grid}
	Let $\Psi\in \mathrm{FOS}$. Then there exists $T'\in \mathbb{Z}_{\geq 0}$ such that for all inclusion-maximal sets $I$ and all $\mathbf{a}\in A_I$, either all permutations in $\mathcal{P}^{T'}_{I, \mathbf{a}}$ satisfy $\Psi$ or none of them does.
\end{lemma}
\begin{proof}
Immediate from Lemma~\ref{lemma: infinite model ER game} by taking $T'\geq 2^k$, where $k$ is the quantifier depth 
in $\Psi$.
\end{proof}
\noindent In order to prove that every fraction claimed in Theorem~\ref{thm:poly growth} can be achieved as a limiting probability, we also need the next lemma, which in some sense, states that there exists an `indicator function-like'  first-order sentence for each of the sets $\mathcal{P}^{T'}_{I, \mathbf{a}}$, allowing us to tell them apart. 
\begin{lemma}\label{lemma: indicator FOSs}
Fix a constant $C$. For every size-maximal popular set $I$ and every $\mathbf{a} \in A_I$, there exists a first-order sentence $\Psi_{\mathbf{a},I}$ such that $\sigma \models \Psi_{\mathbf{a},I}$ if and only if $\sigma^\#$ contains exactly $\mathbf{a}_i$ points or at least $C$ points in its $(i,\tau(i))$-th cell depending on whether $i \in I^c$ or $i \in I$ respectively. 
\end{lemma}
\noindent As the proof of Lemma~\ref{lemma: indicator FOSs} is constructive and the techniques used in the proof differ from the ones mostly used in this section, we defer it to the Appendix. Given Lemmas~\ref{lemma: Psi 0-1 satisfied by super restricted geom grid} and~\ref{lemma: indicator FOSs}, it is straightforward to derive zero--one logical limit laws for the sets $\mathcal{P}_{I, \mathbf{a}}$, which combined with the results of the previous subsection yields Theorem~\ref{thm:poly growth}. We fill out the details below.
\begin{proof}[Proof of Theorem~\ref{thm:poly growth}]
Fix $\Psi\in \mathrm{FOS}$. Let $T'=T'(\Psi)\in \mathbb{Z}_{\geq 0}$ be such that the conclusion of Lemma~\ref{lemma: Psi 0-1 satisfied by super restricted geom grid} holds. Let $P$ denote the size of a largest popular subset of $[m]$, and 
$b$ as defined after~\eqref{eq: size of Q}. Let $b_{\Psi}$ denote the number of pairs $(I, \mathbf{a})$ such that $I$ is a popular set of size $P$, $\mathbf{a}\in A_I$ and every permutation $\sigma \in \mathcal{P}^{T'}_{I, \mathbf{a}}$ satisfies $\psi$; note this quantity is well-defined by Lemma~\ref{lemma: Psi 0-1 satisfied by super restricted geom grid} and our choice of $T'$, and that $b_{\Psi}\in [0, b]\cap\mathbb{Z}$. By a standard balls-and-bins argument, we have that $\left\vert\left(\mathcal{P}_{I, \mathbf{a}} \setminus \mathcal{P}^{T'}_{I, \mathbf{a}}  \right)_n \right\vert =O(n^{ \vert I\vert -2})$ and hence $\vert (\mathcal{P}^{T'}_{I, \mathbf{a}}  )_n \vert =\binom{n}{\vert I\vert -1}+ O(n^{\vert I\vert -2})$.

Combining this information with Lemma~\ref{lemma: approx P by Q}, 
equality~\eqref{eq: size of Q} and the disjointness of the $(\mathcal{P}^{T'}_{I, \mathbf{a}}  )_n$ following from Lemma~\ref{lemma: disjointness}, we have with probability $1-O(n^{-1})$ a uniformly chosen random permutation  $\sigma \in \mathcal{P}_n$ lies inside the union of the $b$ pairwise disjoint sets $(\mathcal{P}^{T'}_{I, \mathbf{a}}  )_n$, where $I$ is a popular set of size $P$ and $\mathbf{a}\in A_I$. Since each of these sets $(\mathcal{P}^{T'}_{I, \mathbf{a}}  )_n$ has size $\binom{n}{\vert I\vert -1}+ O(n^{\vert I\vert -2})$, and since precisely  $b_{\Psi}$ of them consist solely of permutations satisfying $\psi$ and $(b-b_{\Psi})$ of them consist solely of permutations not satisfying $\psi$, it follows that the probability $\sigma \models \Psi$ is precisely $b_{\Psi}/b+O(n^{-1})$. This completes the proof of Theorem~\ref{thm:poly growth} part (i), and also establishes half of part (ii), namely that $\mathrm{Lim}(\mathcal{P})\subseteq \{b'/b: \ b'\in [0,b]\cap \mathbb{Z}\}$.

To complete the proof of part (ii), we shall apply Lemma~\ref{lemma: indicator FOSs}. Let $S$ denote the set of pairs $(\mathbf{a},I)$ with $\mathbf{a} \in A_I$ and $I$ a size-maximal popular set. Recall that we have $b = \vert S\vert $, and that $\vert \mathcal{P}_n \setminus \bigcup_{(\mathbf{a},I) \in S} \mathcal{P}^{T'}_{I, \mathbf{a}}\vert  = o\left(\vert\mathcal{P}_n\vert\right)$ for any fixed integer $T'$. Let $T'$ be a positive integer chosen such that the conclusion of Lemma~\ref{lemma: Psi 0-1 satisfied by super restricted geom grid} is satisfied. Given $b' \in [0,b] \cap \mathbb{Z}$, let $S'$ denote an arbitrary subset of $S$ of size $b'$, and let $\Psi_{b'} = \bigvee_{(\mathbf{a},I) \in S'} \Psi_{\mathbf{a},I}$. Using Lemmas~\ref{lemma: indicator FOSs} with $C = T+T'$ and \ref{lemma: disjointness}, it is easy to verify that we have $\lim_{n\rightarrow \infty} \mathbb{P}\left(\sigma_n \models \Psi_{b'}\right) = \frac{b'}{b}$, as required. Thus $\mathrm{Lim}(\mathcal{P})= \{b'/b: \ b'\in [0,b]\cap \mathbb{Z}\}$, as claimed.
\end{proof}
\section{The exponential growth regime}
Throughout this section, let $\cP$ denote a small permutation class with exponential growth. Markstr\"om and R\"aty proved in~\cite{eeroklas} that such a permutation class satisfies a TOTO logical limit law, i.e.\ that for every $\Psi \in \mathrm{FOS}$ we have $\mathbb{P}\left(\bfsigma_n \models \Psi\right)$ converging to a limit as $n\rightarrow \infty$, where $\bfsigma_n$ denotes a random permutation from $\cP_n$, see Theorem~\ref{theorem: eeroklas}. Our goal in this section is to prove that the set of limiting probabilities $\mathrm{Lim}(\cP)$ is dense in $[0,1]$. We shall do this by proving the following technical covering result, which easily implies Theorem~\ref{thm:exp growth}:
\begin{theorem}\label{theorem: technical cover}
    Let $\varepsilon>0$ be fixed. Then there exist $N_1, N_2 \in \mathbb{Z}_{\geq 0}$ and first-order sentences $\Psi_{1}, \Psi_2, \ldots ,\Psi_{N_1}$ such that 
    \begin{enumerate}[(i)]
        \item (small probability) $\lim_{n \rightarrow \infty}\mathbb{P}(\bfsigma_n \models \Psi_i) <\varepsilon$ for all $i \in [N_1]$, where $\bfsigma_n$ is a  random permutation from $\cP_n$;
        \item (covering property) for all $n\geq N_2$, every permutation in $\cP_n$ satisfies some $\Psi_i$, $i\in [N_1]$.
    \end{enumerate} 
\end{theorem}
\begin{proof}[Proof of Theorem~\ref{thm:exp growth} from Theorem~\ref{theorem: technical cover}]
    Given any real $r\in (0,1)$ and any $\varepsilon > 0$, let $\Psi_1, \ldots, \Psi_{N_1}\in \mathrm{FOS}$ be a family of first-order TOTO sentences satisfying the conclusion of Theorem~\ref{theorem: technical cover}. Define $\Phi_i := \bigvee_{j=1}^{i} \Psi_j$, and let $\Phi_0$ be false, and let $i\geq 0$ be the largest index such that $\lim_{n \rightarrow \infty} \mathbb{P}(\bfsigma_n \models \Phi_i) < r$. Note that by condition (ii) $\lim_{n \rightarrow \infty} \mathbb{P}(\bfsigma_n \models \Phi_{N_1}) = 1$, and thus we must have $i < N_1$. Then by condition (i) and a union bound,
    $$\lim_{n \rightarrow \infty} \mathbb{P}(\bfsigma_n \models \Phi_{i+1}) \leq \lim_{n \rightarrow \infty} \mathbb{P}(\bfsigma_n \models \Phi_i) + \lim_{n \rightarrow \infty} \mathbb{P}(\bfsigma_n \models \Psi_{i+1}) < r+ \varepsilon,$$
    and by monotonicity we certainly have $\lim_{n \rightarrow \infty} \mathbb{P}(\bfsigma_n \models \Phi_{i+1}) \geq r$. This immediately implies that $\mathrm{Lim}(\cP)$ is dense in $[0,1]$.
\end{proof}
\noindent  In the remainder of the section, we may thus focus on proving Theorem~\ref{theorem: technical cover}. Before outlining our strategy, we introduce some useful notions and notation.

Given a permutation $\sigma \in \mathrm{Geom}(M)_{d,k}$, we write $ \sigma^{\#}_{i,j}$, $ \sigma^{\#}_{i, -}$ and $\sigma^{\#}_{-,j}$ for the set of points of $\sigma^\#$ that lie inside the $(i,j)$-th cell, the $i$-th column of cells and the $j$-th row of cells of the standard $\ell\times m$ gridding, respectively. Analogously, we write $(F,G)\in \sigma^{\#}$ as a shorthand for the fact that the gridded figure $(F,G)$ lies in the equivalence class of $\sigma^{\#}$, and we write $F_{i,j}$, $F_{i,-}$, $F_{-,j}$ for the collection of points of $F$ in the  $(i,j)$-th cell, the $i$-th column of cells and the $j$-th row of cells of the gridding $G$, respectively.

Given a permutation $\sigma$ and points $a,b\in \sigma$ with $a\leq_1 b$, we write $[a,b]_1$ for the collection of $x\in \sigma$ with $a\leq_1 x\leq_1 b$. We define $[a,b]_2$ similarly with respect to $\leq_2$ instead of $\leq_1$, and extend these notions analogously to gridded figures and their equivalence classes. Further, we write $a_{1,i}$ for the $i$-th successor of $a$ in the $\leq_1$-order (if it exists), that is to say $a_{1,0}=a$ and for $i\geq 1$, $a_{1,i}$ is the point such that $[a, a_{1,i}]_1$ contains exactly $i+1$ points (if such a point exists). Finally, given a first-order sentence $\Psi$, we write $\cP^\Psi$ for the set of permutations $\sigma \in \cP$ that satisfy $\Psi$. 
\subsection{Strategy}\label{subsection: strategy}
\noindent The challenge in the proof of Theorem~\ref{theorem: technical cover} is to manage the balancing act between on the one hand choosing first-order sentences having probability less than $\varepsilon$ of being satisfied (as in condition (i)), and on the other hand ensuring that the number $N_1$ of sentences we need to cover $\cP_n$ (as in condition (ii)) remains bounded as a function of $\varepsilon$.

As the starting point for our argument, we recall Vatter's fundamental result (Proposition~\ref{prop: small implies contained inside Geom (M) dk}) that there exists a suitable matrix $M$ and $d,k\in \mathbb{Z}_{\geq 0}$ such that $\cP \subseteq \mathrm{Geom}(M)_{d,k}$. In the remainder of this section we assume that $M$ is an $\ell\times m$ matrix, and let $G$ denote the standard $\ell\times m$ gridding. 

Given $\sigma \in\cP_n$, we consider $\sigma^{\#}$, the $<_{LG}$-minimal equivalence class of gridded figures $(F,G) \in \mathrm{Geom}^{\#}(M)_{d,k}$ inducing $\sigma$. By the pigeon-hole principle, for $n$ sufficiently large, some row or column in $G$ must contain `many' points of $\sigma^{\#}$. (The reason we consider full rows or columns rather than individual cells is that otherwise we may miss some essential information about the permutation $\sigma$, such as e.g.\ oscillation patterns due to the presence of two non-empty cells in the same column.)

As a consequence of the following result of Markstr\"{o}m and R\"{a}ty  (Proposition~\ref{prop: eeroklas} below), given a row or column of $M$ one may construct a first-order sentence $\widetilde{\Psi}$ such that $\sigma\in \cP$ satisfies $\sigma \models \widetilde{\Psi}$ if and only if the specified row or column is `rich in points' in $\sigma^{\#}$. Going over all rows and columns, the collection of such $\widetilde{\Psi}$ will certainly satisfy the covering condition (ii) from the conclusion of Theorem~\ref{theorem: technical cover} for all $n$ sufficiently large. However these first-order sentences $\widetilde{\Psi}$ will likely not satisfy condition (i), so we shall need to `refine' them further to produce suitable first-order sentences having a smaller probability of being satisfied.

To achieve this, we shall decompose them into first-order sentences for which certain  `structured deletion functions' exist. This is the heart of our proof, and we will now outline our approach to it in the case of columns (the argument for rows follows similarly).

Fix constants $C_0$ and $L_0$ and let $(i,-)$ be a fixed column in $M$. 
We shall consider first-order sentences $\Psi$ that are of the form $\Psi = \exists x: \psi (x)$, satisfying the following conditions. 
\begin{enumerate}
    \item For every $\sigma$ with $\sigma \models \Psi$, there exists a unique $a$ for which $\sigma \models \psi(a)$, and this unique $a$ corresponds to the $C_0$-th least point in $\sigma^\#_{i,-}$ with respect to the $\leq_1$-order. 
    \item There exists a constant $C_1$ depending only on $C_0$, 
    $L_0$, $d$ and $k$ so that the points in $[a_{1,0}, a_{1,C_1-1}]_1$ are all contained in 
    $\sigma^\#_{i,-}$
    and furthermore satisfy certain additional (somewhat technical) structural properties that we do not specify at this stage. 
\end{enumerate}
Now, suppose that $\Psi$ satisfies the conditions described above, and let $\cP^{\Psi}$ denote the set of permutations in $\cP$ that satisfy $\Psi$. Let $s$ denote a constant chosen so that $1 \leq s \leq C_1 - L_0 - 1$. We define a \textit{deletion function} $f_{\Psi}: \cP^{\Psi}_n \rightarrow \cP_{n-L_0}$  to be the function that removes the points in $[a_{1,s},a_{1,s+L_0-1}]_1$ from $\sigma$. Further we also define the \textit{gridded deletion function} $f_{\Psi}^{\#}: \cP^\Psi \rightarrow \text{Geom}^\#(M)_{d,k}$ as follows: given $\sigma \in \cP^\Psi$, we write $\sigma^\#$ for the left-greedy decomposition of $\sigma$. Let $(F,G)$ denote a representative in the equivalence class of $\sigma^\#$. We then set $f^\#_{\Psi}(\sigma)$ to be the equivalence class containing the gridded figure $(F',G)$, where $F'$ is obtained by deleting the points (corresponding to) $[a_{1,s},a_{1,s+L_0-1}]_1$ from $F$. It is easy to check that the function $f_{\Psi}^\#$ is well-defined, as the definition is independent of the choice of representative $(F,G)$. 

Write $\pi^\#: \text{Geom}^\#(M)_{d,k} \rightarrow \text{Geom}(M)_{d,k}$ for the natural projection map which removes the grid-lines, and write $\Pi^\#:\text{Geom}(M)_{d,k} \rightarrow \text{Geom}^\#(M)_{d,k}$ for the function mapping $\sigma$ to  $\sigma^\#$, the left-greedy minimal equivalence class of $M$-compatible gridded figures inducing $\sigma$ . 
\begin{definition}\label{def: structured deletion}
    We say that a deletion function $f_{\Psi}$ is a \textit{structured deletion function} if the gridded deletion function $f_{\Psi}^{\#}$ is an injection, and if for every $\sigma \in \cP^\Psi$, we have 
    \begin{align} \label{eq: structured}
        \Pi^\#(f_\Psi(\sigma)) = f_{\Psi}^\#(\sigma).
    \end{align}
\end{definition}
\noindent  All this may at a first glance seem rather abstruse and unnecessarily complicated, and the reader may wonder: why are structured deletion functions useful? 
From Definition~\ref{def: structured deletion}, and more specifically the injectivity of gridded deletion functions and the crucial equality~\eqref{eq: structured}, it is not hard to show that  structured deletion functions themselves are injective (Lemma~\ref{lemma: injective}). Combining this with the facts that we are free to choose the parameter $L_0$ and that the permutation class $\cP$ has exponential growth, we may ensure that the condition (i) holds for $\Psi$: we have
\begin{align*}
    \vert \cP^{\Psi}_n\vert =\vert f_{\Psi}(\cP^{\Psi}_ n)\vert \leq \vert \cP_{n-L_0}\vert \leq c^{-L_0}\vert \cP_n\vert
\end{align*}
for some constant $c>1$, so taking $L_0$ sufficiently large as a function of $\varepsilon$ we can ensure $\mathbb{P}(\bfsigma_n\models \Psi)<\varepsilon$.

We will use such structured first-order sentences $\Psi$ equipped with deletion functions to `refine' our previously discussed collection of first-order sentences $\widetilde{\Psi}$ about rich rows or columns, breaking down each $\cP^{\widetilde{\Psi}}$ into a union of smaller pieces $\cP^{\Psi}$, each satisfying condition (i).

What remains to be done then is to construct these $\Psi$ and the associated deletion functions, and show that the latter are indeed structured deletion functions. At this stage, we shall say rather vaguely that the first-order sentences $\Psi$ and the structured point-set $[a_{1,0},a_{1,{C_1-1}}]_1$ will arise from certain finite patterns occurring in a fixed column. In general, it will be fairly straightforward to show that the gridded deletion function $f^{\#}_{\Psi}$ is injective. It will however be much more challenging to show that the crucial equality~\eqref{eq: structured} holds. 

To overcome this difficulty, we shall ensure that the structured set of $\leq_1$-consecutive points $a_{1,0}, a_{1,1},\dots,a_{1,C_1-1}$ that can be uniquely identified by $\Psi$ inside every $\sigma \in \cP^\Psi$ contains $C_1$ points, of which only $L_0$ (which will be much smaller than $C_1$) are deleted by $f_{\Psi}$, thus ensuring that some amount of structure is preserved after the deletion. In particular, this remaining structure ensures first of all that the points $a_{1,s-1}$ and $a_{1,s+L_0}$ which precede and succeed the deleted points with respect to the $\leq_1$-order are contained in the same column in the left-greedy decomposition of $f_{\Psi}(\sigma)$, and secondly that the `deleted points' can be added back to this column 
to recover a gridded figure that represents $\sigma$. 

Unavoidably, the description and discussion of deletion functions above may feel rather vague at this stage, so before we dive into the details let us mention an alternative perspective on what we aim to achieve with Definition~\ref{def: structured deletion} in terms of a `commutative' diagram, see Figure~\ref{fig:fig 2}.

Let $(F_{0,0},G)$ be an arbitrary representative of the equivalence class of $\sigma^\#$, and let $(F_{1,0},G)$ denote an arbitrary representative of the equivalence class $f_\Psi^{\#}(\sigma)$, and for convenience write $\theta := f_{\Psi}(\sigma)$. Let $(F_{1,1},G)$ be a gridded figure from the equivalence class of the left-greedy representation of $f_\Psi^{\#}(\sigma)$, and recall from above that it is possible to add back the deleted points in some column to obtain a gridded figure $(F_{0,1},G)$ which induces the permutation $\sigma$. 

We now record some important properties of these gridded figures that can be used to deduce the condition~\eqref{eq: structured}; the formal argument is presented in Lemma~\ref{lemma: square comparision}. As in each instance, we use the standard gridding $G$, we omit it below; also for simplicity's sake we only describe what happens in the case where we have a rich column (the case of rich rows can be dealt with analogously). 
\begin{enumerate}
    \item The gridded figures $F_{0,0}$ and $F_{1,0}$ are somewhat `similar', in the sense that $F_{1,0}$ is obtained from $F_{0,0}$ by deleting points from at most two cells (which are located in the same column), and the same holds for the pair $F_{0,1}$ and $F_{1,1}$. 
    \item If points are deleted from two distinct cells in the same column, 
    then the number of points in the upper cell of $F_{0,0}\setminus F_{1,0}$ is the same as the number of points in the upper cell of $F_{0,1}\setminus F_{1,1}$, and similarly the number of points  in the lower cell of $F_{0,0}\setminus F_{1,0}$ is the same as the number of points in the lower cell of $F_{0,1}\setminus F_{1,1}$.
    \item The gridded figures $F_{0,0}$ and $F_{0,1}$ both represent $\sigma$, and we have $F_{0,0} \leq_{LG} F_{0,1}$ as $F_{0,0}$ is a representative of the equivalence class of  $\sigma^\#$. Similarly, the gridded figures $F_{1,0}$ and $F_{1,1}$ both represent $\theta$, and we have $F_{1,1} \leq_{LG} F_{1,0}$. 
\end{enumerate}
Thus, within the columns of Figure~\ref{fig:fig 2} the gridded figures are somewhat similar in structure, while in both rows we know the status of the left-greedy order, and their directions are not aligned. A rather tedious technical argument (Lemma~\ref{lemma: square comparision}) can be used to prove that actually both of the left-greedy relations must be equalities, and in particular, that both $F_{1,0}$ and $F_{1,1}$ represent the same equivalence class, which is equivalent to the key equality~\eqref{eq: structured}. To carry out this technical argument, we need to add several minor assumptions, such as requiring that the differences in points between both pairs correspond to the same labeled set of points (that is, the points inserted back are inserted in their original gaps - in our application this  condition is trivially satisfied as the set of deleted points is an interval which is re-added into the appropriate gap). In addition, we need to ensure that within each cell of the gridded figure, we are never deleting or inserting the least point under the order $\leq_1$. 

Having outlined our approach we build up the necessary machinery to carry it out. In the following subsection, we recall some results from~\cite{eeroklas},  and then move on to proving Lemmas~\ref{lemma: injective} and \ref{lemma: square comparision}. In Section~\ref{subsection: properties of permutations in Okd} we derive certain structural properties of the permutation class $\widetilde{\mathcal{O}}_{k}^{[d]}$ needed for our proofs. We then split our argument into three distinct parts: dealing with monotone sub-intervals (Section~\ref{subsection: monotone intervals}) dealing with columns and rows containing two non-zero entries (Section~\ref{subsection: alternating intervals}) and dealing with $\star$-entries (Section~\ref{subsection: star entries}). Finally, we combine all these parts together and prove Theorem~\ref{theorem: technical cover} in Section~\ref{subsection: putting it together}. 

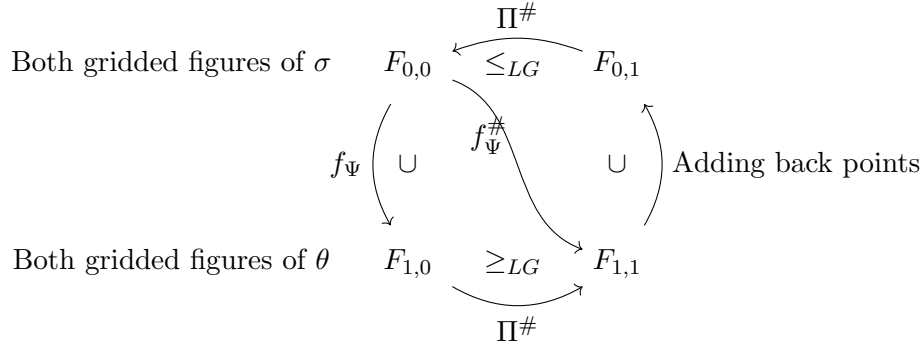
\begin{figure}
    \centering
    \begin{tikzpicture}
\matrix (m) [
  matrix of nodes,
  nodes={minimum width=1cm, minimum height=1cm, align=center},
  row sep=3mm,
  column sep=3mm
] {
  Both gridded figures of $\sigma$ & $F_{0,0}$ & $\leq_{LG}$ & $F_{0,1}$ \\
    & $\cup$ &  & $\cup$ \\
  Both gridded figures of $\theta$  & $F_{1,0}$ & $\geq_{LG}$  & $F_{1,1}$ \\
};

\draw[->, bend right=30]
  (m-1-2) to node[midway, left] {$f_{\Psi}$} (m-3-2);

\draw[<-, out=20, in=160]
  (m-1-2) to node[midway, above] {$\Pi^{\#}$} (m-1-4);

\draw[->, bend right=30]
  (m-3-2) to node[midway, below] {$\Pi^{\#}$} (m-3-4);

\draw[<-, bend left=30]
  (m-1-4) to node[midway, right] {Adding back points} (m-3-4);

\draw[->, out=-20, in=160]
  (m-1-2) to node[midway, above left] {$f_{\Psi}^{\#}$} (m-3-4);
\end{tikzpicture}
    \caption{The containment relation within each column holds in the following strong sense: for each column, the point sets in each cell of the gridded figures in the bottom row are contained in the point sets of the gridded figures in the top row. Moreover, this inclusion is an equality for all but at most two cells, which must lie in the same column of the gridding $G$.}
    \label{fig:fig 2}
\end{figure}
\subsection{Setting up}

The following useful result due to the second and third authors~\cite{eeroklas} provides us a way to use first-order sentences to describe the condition that a certain row or column is `rich in points'. While such sentences are not too hard to construct, we will not state them explicitly at this point, but defer their (technical) statements to the subsections dealing with the various subcases of our argument, and defer their (tedious) construction to the Appendix. 
\begin{proposition}[Markstr\"om--R\"aty] \label{prop: eeroklas}
    Let $M$ be a suitable matrix, $d$ and $k$ be positive integers, and let $I$ be a set of positions of non-zero entries of $M$. There exists a first-order formula $\psi_I(x_{u,v,i,j}: u\in\{1,2\}, v\in\{0,1\}, (i,j)\in I)$ such that, for every $\sigma\in\mathrm{Geom}(M)_{d,k}$, we have $\sigma\models\psi_I(a_{u,v,i,j}: u\in\{1,2\}, v\in\{0,1\}, (i,j)\in I)$ if and only if $I$ is precisely the set of non-empty cells of $\sigma^\#$ and, for every $(i,j)\in I$ and $u\in\{1,2\}$, the points $a_{u,0,i,j}$ and $a_{u,1,i,j}$ are respectively the $\leq_u$-least and $\leq_u$-greatest points of $\sigma^\#_{i,j}$.
\end{proposition}

We now move on to proving some important preliminary results, which give some motivation for considering structured deletion functions. 
\begin{lemma} \label{lemma: injective}
    Let $\Psi$ be a first-order sentence, and let $f: \mathcal{P}^{\Psi}_n \rightarrow \mathcal{P}_{n-L}$ be a structured deletion function. Define the particular constant $\gamma := \frac{1+\sqrt{5}}{2}$. Then $f$ is an injection, and we have $\lim_{n \rightarrow \infty} \mathbb{P}(\sigma_n \models \Psi) \leq \gamma^{-L}$. 
\end{lemma}
\begin{proof}
    Let $f^\#$ denote the gridded version of $f$, and recall that by Definition~\ref{def: structured deletion} $f^\#$ must be an injection. If $f(\sigma_1) = f(\sigma_2)$ for some $\sigma_1,\sigma_2 \in \mathcal{P}^\Psi_n$, then by~\eqref{eq: structured}, we have
    \[f^\#(\sigma_1) = \Pi^\#(f(\sigma_1)) = \Pi^\#(f(\sigma_2)) = f^\#(\sigma_2),\]
    whence $\sigma_1=\sigma_2$ by injectivity of $f^{\#}$. Thus $f$ must be injective as well. Hence, we conclude that  
    \[\lim_{n \rightarrow \infty} \mathbb{P}(\sigma_n \models \Psi)  = \lim_{n\rightarrow \infty} \frac{\vert \cP^\Psi_n\vert}{\vert \cP_n\vert}\leq \lim_{n\rightarrow \infty} \frac{\vert \cP_{n-L} \vert}{\vert \cP_n\vert} \leq \gamma^{-L},\]
    by using the third claim in Theorem~\ref{theorem: eeroklas} and recalling that $\cP$ has growth rate at least $\gamma$ by the results of Kaiser and Klazar~\cite{kaiser2002growth}. 
\end{proof}
As the statement of Lemma~\ref{lemma: square comparision} is rather technical, we refer the reader to Figure~\ref{fig:fig 2} which provides a rough overview of the statement. 
\begin{lemma}\label{lemma: square comparision}
Let $M$ be a suitable matrix. Suppose that $(F_{u,v},G)$,
for $(u,v)\in\{0,1\}\times\{0,1\}$ are gridded
figures in $\mathrm{Geom}(M)^\#_{d,k}$ such that $(F_{0,0},G)$ and $(F_{0,1},G)$ induce
$\sigma$, while $(F_{1,0},G)$ and $(F_{1,1},G)$ induce $\theta$. In addition, assume that $(F_{0,0},G)$ belongs to the equivalence class of the left-greedy representation $\sigma^\#$. 

Let $\mathcal{S}$ and $\mathcal{S'}$ denote collections of cells, all of which correspond to non-zero entries of $M$, taking one of the following three forms:
    \begin{enumerate}
        \item[(a)] $\mathcal{S} = \{(i,j)\}$ and $\mathcal{S'} = \{(i',j')\}$, or
        \item[(b)] $\mathcal{S} = \{(i,j_0),(i,j_1)\}$ and $\mathcal{S'} = \{(i',j'_0),(i',j'_1)\}$ with $j_0 < j_1$ and $j'_0 < j'_1$, or
        \item[(c)] $\mathcal{S} = \{(i_0,j),(i_1,j)\}$ and $\mathcal{S'} = \{(i'_0,j'),(i'_1,j')\}$ with $i_0 < i_1$ and $i'_0 < i'_1$.
    \end{enumerate}
\noindent Suppose that the gridded figures satisfy the following conditions: 
    \begin{align} \label{eq: most cells equal 1}
        (F_{0,0})_{u,v} &= (F_{1,0})_{u,v} \text{ for all } (u,v) \not \in \mathcal{S}, \\
 \label{eq: most cells equal 2}
        (F_{0,1})_{u,v} &= (F_{1,1})_{u,v} \text{ for all } (u,v) \not \in \mathcal{S'}, \\
  \label{eq: inclusions 1}
        (F_{1,0})_{u,v} &\subseteq (F_{0,0})_{u,v} \text{ for all } (u,v) \in \mathcal{S}, \text{ and }\\
    \label{eq: inclusions 2}
        (F_{1,1})_{u,v} &\subseteq (F_{0,1})_{u,v} \text{ for all } (u,v) \in \mathcal{S'}. 
    \end{align}
\noindent Label the points of $F_{0,0}$ and $F_{0,1}$ by their
corresponding entries in the permutation $\sigma$ (i.e.\  the $i$-th least point in the $\leq_1$-order is labeled $(i, \sigma(i))$), and give $F_{1,0}$ and
$F_{1,1}$ the labels inherited through the above inclusions.

Suppose that under the inheritance of labels, $F_{1,0}$ and $F_{1,1}$ have the same label set. Further, suppose that in each of the cells $(u,v) \in \mathcal{S}$ the $\leq_1$-least point of $(F_{0,0})_{u,v}$ also belongs to $(F_{1,0})_{u,v}$, and that for every $(u,v) \in \mathcal{S'}$ the $\leq_1$-least point of $(F_{0,1})_{u,v}$ also belongs to $(F_{1,1})_{u,v}$. Furthermore, suppose the following are satisfied:
        \begin{align} \label{eq: three conditions 1}
             & \text{if in case (a), then } |(F_{0,1})_{i',j'}| - |(F_{1,1})_{i',j'}| = |(F_{0,0})_{i,j}| - |(F_{1,0})_{i,j}|; \\
             & \text{if in case (b), then } |(F_{0,1})_{i',j'_s}| - |(F_{1,1})_{i',j'_s}| = |(F_{0,0})_{i,j_s}| - |(F_{1,0})_{i,j_s}| \text{ for } s\in \{0,1\};\label{eq: three conditions 2}\\
             & \text{if in case (c), then } |(F_{0,1})_{i'_s,j'}| - |(F_{1,1})_{i'_s,j'}| = |(F_{0,0})_{i_s,j}| - |(F_{1,0})_{i_s,j}| \text{ for } s\in \{0,1\}.\label{eq: three conditions 3}
        \end{align}
 Finally, suppose that we have $(F_{1,1},G) \leq_{LG} (F_{1,0},G)$, and recall that $(F_{0,0},G) \leq_{LG} (F_{0,1},G)$ since $(F_{0,0},G)$ belongs to the equivalence class $\sigma^\#$. Then $(F_{1,0},G)$ and $(F_{1,1},G)$ belong to the same
equivalence class in $\mathrm{Geom}(M)^\#_{d,k}$.
\end{lemma}
\begin{proof}
We only present the proof when the assumptions of case (b) hold, as the other cases follow by using very similar arguments. For convenience, we drop the grid $G$ from the notation. Viewing $\sigma$ as a labeled permutation, let $D$ be the labeled set of points removed from $F_{0,0}$ to obtain $F_{1,0}$. Thus by our assumption $D$ is also the labeled set of points removed from $F_{0,1}$ to obtain $F_{1,1}$. Note that the result follows trivially if $D =\emptyset$, so we may assume that $D$ is non-empty. 

Suppose for the sake of contradiction that we have $F_{1,1}<_{LG} F_{1,0}$. Write $h = \min(i,i')$, and observe that all differences between the pairs $(F_{0,0},F_{1,0})$ and $(F_{0,1},F_{1,1})$ occur only in columns whose index is at least $h$. Note that we must have $|(F_{1,0})_{u,v}| = |(F_{1,1})_{u,v}|$ for every $u < h$, as otherwise the first difference between such cell counts witnessing $F_{1,1}<_{LG}F_{1,0}$ would immediately translate to the pair $(F_{0,0},F_{0,1})$ giving $F_{0,1}<_{LG}F_{0,0}$, which is a contradiction. Analogously we must have $|(F_{0,1})_{u,v}| = |(F_{0,0})_{u,v}|$ for every $u< h$. In particular, we have 
\begin{equation} \label{eq: all equal}
        |(F_{0,0})_{u,v}| = |(F_{0,1})_{u,v}| = |(F_{1,0})_{u,v}| = |(F_{1,1})_{u,v}| \text{ for all $(u,v)$ with $u < h$.}
\end{equation}
Recall we identify points in our gridded figures with the corresponding elements of $\sigma$.
Let $a$ denote the largest point in $\sigma$ with respect to the $\leq_1$-order that is not in $D$, and is smaller than all the points in $D$ with respect to the order $\leq_1$. Note that $a$ exists, and by our assumptions $a$ belongs to the $i$-th column of both $F_{0,0}$ and $F_{1,0}$, and to the $i'$-th column of both $F_{0,1}$ and $F_{1,1}$. 

If $i < i'$, we note that $a$ belongs to the $i$-th column of $F_{1,0}$ but $a$ is not contained in the $i$-th column of $F_{1,1}$, which together with equation~\eqref{eq: all equal} implies that $F_{1,0} <_{LG} F_{1,1}$ (as the first condition of Definition~\ref{def: left greedy} is satisfied with $i_0 = i$), which is a contradiction. If $i' < i$, we obtain the contradiction $F_{0,1} <_{LG} F_{0,0}$ by noting that the first condition of Definition~\ref{def: left greedy} is satisfied with $i_0 = i'$ (again,  by considering the location of the point $a$).

Finally consider the case when $i' = i$. Arguing similarly as above, we must have 
$j_0=j'_0$ and $j_1=j'_1$. Since all the deleted points $D$ are contained in the $i$-th column,  \eqref{eq: three conditions 2} implies that 
\begin{equation}\label{eq: counts equal}
    |(F_{0,0})_{u,v}| = |(F_{0,1})_{u,v}| \text{ and }|(F_{1,1})_{u,v}| = |(F_{1,0})_{u,v}|
\end{equation}
for every $(u,v)$ with $u \leq i$. Since we assumed that $F_{1,1}<_{LG}F_{1,0}$, it follows that this relation must be witnessed at a column or inside a column whose index is strictly greater than $i$. However, using the fact that $(F_{0,0})_{u,v} = (F_{1,0})_{u,v}$ and $(F_{0,1})_{u,v} = (F_{1,1})_{u,v}$ holds for all $u >i$, this immediately leads to the contradiction $F_{0,1}<_{LG}F_{0,0}$. Thus  $F_{1,1}$ and $F_{1,0}$ must belong to the same equivalence class.
\end{proof}

\subsection{Properties of permutations in $\widetilde{\mathcal{O}}_k^{[d]}$}\label{subsection: properties of permutations in Okd}
\begin{definition}\label{monotone bipartition}
    Given a permutation $\sigma \in \mathfrak{S}_n$ and a set $I$ of consecutive points $a_1,\dots,a_\ell$ with respect to the $\leq_1$-order,  we say that $I$ can be partitioned into monotone segments if there exists a bipartition $I = A \sqcup B$ so that both $A$ and $B$ are monotone subpermutations of $\sigma$, and for all $a \in A$ and $b \in B$ we have $b <_2 a$. We say that $I$ can be partitioned into monotone segments with respect to the bipartition $I = A \sqcup B$ if the sets $A$ and $B$ satisfy the properties described above. 
\end{definition}
\begin{definition}\label{def: alternation}
    Given a permutation $\sigma \in \mathfrak{S}_n$ and a set of $\leq_1$-consecutive points $I$ inside $\sigma$
    that can be partitioned into monotone segments with respect to a bipartition $I = A \sqcup B$, 
    we define the number of alternations in $I$ with respect to $A \sqcup B$ to be the number of indices $i \in [|I|-1]$ for which $a_i$ and $a_{i+1}$ are not contained in the same part ($A$ or $B$) of the bipartition. Finally, we define the number of alternations in $I$ as the maximum number of alternations with respect to any such bipartition $A \sqcup B$ of $I$. 
\end{definition}
\noindent It is not too hard to check that if $I$ can be partitioned into monotone segments in multiple ways, then the number of alternations with respect to these bipartitions can differ by at most one. In the following Lemma, we prove that any $I$ that is contained in an element $\sigma \in \widetilde{\mathcal{O}}_{k}^{[d]}$ cannot contain too many alternations. 
\begin{lemma}\label{lemma: bound on number of alternations in O_k^d}
    For all $d,k\in\mathbb{Z}_{>0}$ there exists $D_{1}=D_{1}\left(d\right)\in\mathbb{Z}_{>0}$
    such that the following holds. Suppose that $\sigma\in\widetilde{\mathcal{O}}_{k}^{\left[d\right]}$ and that $I$ consists of consecutive points $a_1, \dots,a_\ell$ with respect to the $\leq_1$-order in $\sigma$ so that $I$ can be partitioned into monotone segments. Then $I$ contains fewer than $D_{1}$ alternations. 
\end{lemma}
\begin{proof}
    As $\widetilde{\mathcal{O}}_k^{[d]}$ is closed under taking subpatterns, we may for simplicity assume that $I$ coincides with $\sigma$. Let $A \sqcup B$ be an arbitrary bipartition of $I$ satisfying the conditions of Definition~\ref{monotone bipartition}. We prove by induction on $d\geq1$ that we may take $D_{1}\left(d\right)= 7d$. It is easy to check that any subpattern of either increasing or decreasing oscillating sequence contains at most 5 alternations. 
    
    For the base case $d=1$, observe that if $\sigma$ is a monotone permutation,  it can contain at most one alternation, as we cannot ever swap twice between $A$ and $B$. Otherwise, $\sigma$ must be a subpattern of the increasing or decreasing oscillating sequence, and hence contains at most $5$ alternations. Thus the base case holds. 
    
    For the inductive step, suppose that our inductive statement holds up to some $d\geq1$, and consider $\sigma=\tau\left[\theta_{1},\dots,\theta_{\left|\tau\right|}\right]\in\widetilde{\mathcal{O}}_{k}^{\left[d+1\right]}$, where $\tau\in\widetilde{\mathcal{O}}_{k}$ and $\theta_{1},\dots,\theta_{\left|\tau\right|}\in\widetilde{\mathcal{O}}_{k}^{\left[d\right]}$. We split the proof into subcases depending on whether $\tau$ is a monotone permutation or not. We slightly abuse the notation, and say that an index $i$ belongs to the permutation $\theta_{j}$ if the corresponding point $\left(i,\sigma\left(i\right)\right)$ belongs to $\theta_{j}$ in the drawing $\sigma=\tau\left[\theta_{1},\dots,\theta_{\left|\tau\right|}\right]$. We say that an alternation is internal if both of its points are contained in the same $\theta_{i}$, and external otherwise. 

    We prove that the number of external alternations must always be at most $7$. Since for each $a \in A$ and $b\in B$ we have $\sigma(b) < \sigma(a)$, there can be at most one index $i \in [|\tau|]$ for which $\theta_i$ contains points from both $A$ and $B$. This block can contribute at most 2 external alternations. 

    Let $\tau'$ denote the permutation obtained from $\tau$ by removing its $i$-th point. For each $j \in \tau'$ we may  assign a label $A$ or $B$ based on whether all the points in the block $\theta_j$ belong to $A$ or $B$. Thus, the number of other external alternations of $\sigma$ is clearly bounded above by the alternations of $\tau'$ with respect to this particular partition $A \sqcup B$, and thus the number is at most $5$ as $\tau'$ is monotone, a subpattern of $\omega_k$ or a subpattern of $\varpi_k$. Thus, the total number of external alternations is at most $7$. 

    \noindent \textbf{Case 1, $\tau$ is a monotone permutation:} let us assume that $\tau$ is monotone increasing (the case where it is monotone decreasing follows analogously). Since $\sigma\left(a\right)>\sigma\left(b\right)$ for all $a\in A$ and $b\in B$, it follows that there exists an index $i\in\left[\left|\tau\right|\right]$ so that every $\ell$ which belongs to $\theta_{j}$ for some $j>i$ is contained in $A$, and every $\ell$ which belongs to $\theta_{j}$ for some $j<i$ is contained in $B$. Thus, every internal alternation must be contained in $\theta_{i}$, and the external alternations can only occur between $\theta_{i-1}$ and $\theta_{i}$, or between $\theta_{i}$ and $\theta_{i+1}$. Thus, there are at most $2$ external alternations. Since $\theta_{i}\in\widetilde{\mathcal{O}}_{k}^{\left[d\right]}$, the induction hypothesis implies that $\sigma$ contains less than $D_{1}\left(d\right)$ internal alternations, and thus the total number of alternations is less than $D_{1}\left(d\right)+2\leq D_{1}\left(d+1\right)$,
    as required. 

    \noindent \textbf{Case 2, $\tau$ is a subpattern of $\omega_{k}$ or $\varpi_k$:}  let $a_{m}\in A$ be chosen so that $\sigma\left(a_{m}\right)\leq\sigma\left(a\right)$ for all $a\in A$, and let $i\in\left[\left|\tau\right|\right]$ be chosen so that $a_{m}\in\theta_{i}$. Since $\sigma\left(a\right)>\sigma\left(b\right)$ for all $a\in A$ and $b\in B$, we conclude that for all $j\neq i$, every index of a point in $\theta_{j}$ is contained in $A$ if $\tau\left(j\right)>\tau\left(i\right)$, and every index of a point in $\theta_{j}$ is contained in $B$ if $\tau\left(j\right)<\tau\left(i\right)$. Thus, as before, all internal alternations must be contained inside $\theta_{i}$, and the inductive hypothesis implies that the number of internal alternations in $\theta_{i}$ is at most $D_{1}\left(d\right)-1$. By the earlier observation, the number of external alternations is at most $7$. Thus the number of  alternations is at most 
    \[\left(D_{1}\left(d\right)-1\right) + 7 = 7(d+1) - 1 = D_{1}\left(d+1\right)-1.
    \]
\end{proof}
Our next lemma is a crucial ingredient when we construct the gridded figure $F_{0,1}$ corresponding to $\sigma$ which is the `lift' of  $F_{1,1}$ in our subsequent applications of Lemma~\ref{lemma: square comparision}. For convenience, we slightly extend the notation and define $\widetilde{\mathcal{O}}_{k}^{\left[0\right]} = \{1\}$ to be the permutation class that contains only one permutation, namely the singleton permutation. With respect to this notation we certainly have $\widetilde{\mathcal{O}}_{k}^{\left[1\right]} = \widetilde{\mathcal{O}}_{k}\left[\widetilde{\mathcal{O}}_{k}^{\left[0\right]}\right]$. The reason for introducing this notion is that we want to be able to apply Lemma~\ref{lemma: substitutions inside iteratively inflated alternating sequences} for $d' = 0$ as well. 
\begin{lemma}\label{lemma: substitutions inside iteratively inflated alternating sequences}
    For all $d,k\in\mathbb{Z}_{>0}$, there exists $D_{2}=D_{2}\left(d,k\right)\in\mathbb{Z}_{>0}$ such that the following holds. Suppose $\sigma\in\widetilde{\mathcal{O}}_{k}^{\left[d\right]}$ contains an interval $I$ such that the restriction $\sigma\vert_{I}$ of $\sigma$ to $I$ is order-isomorphic to a permutation $\mu=\tau\left[\theta_{1},\dots,\theta_{r}\right]$, where $\tau$ is a monotone permutation of length $r\geq D_{2} $ and $\theta_{1},\dots,\theta_{r}\in\mathcal{\widetilde{O}}_{k}^{\left[d'\right]}$ for some $0\leq d'\leq d-1$, and for which $\tau\left[\theta_{1},\dots,\theta_{r}\right]$ is also the left-greedy decomposition of $\mu$ when we view $\mu$ as an element of $\mathcal{\widetilde{O}}_{k}^{\left[d'+1\right]}$. Given arbitrary positive integer $r'$ and permutations $\phi_{1},\dots, \phi_{r'}\in\mathcal{\widetilde{O}}_{k}^{\left[d'\right]}$, write $\tau'$ for the monotone permutation of length $r'$ with the same monotonicity as $\tau$ and set $\mu'=\tau'\left[\phi_{1},\dots,\phi_{r'}\right]$. Then we may replace the copy of $\mu$ in $\sigma$ induced by the interval $I$ by a copy of $\mu'$ to obtain a new permutation $\sigma'$ which still belongs to $\mathcal{\widetilde{O}}_{k}^{\left[d\right]}$. 
\end{lemma}
We remark here that we often apply the result to a tuple $(\phi_1, \ldots, \phi_{r'})$ of the form 
\[(\theta_1,\dots,\theta_s,\varphi_1,\dots,\varphi_t,\theta_{s+1},\dots,\theta_r)\] --- that is, we re-insert some of the deleted permutations $\varphi_1,\dots,\varphi_{t}$ at the appropriate gap in the middle of the direct sum.
\begin{proof}
    To simplify the presentation, we assume that $\tau$ is a monotone increasing permutation, as the case when $\tau$ is decreasing follows analogously. For any given $d,k\in\mathbb{Z}_{>0}$, and fixed $d'\in\left\{ 0,\dots,d-1\right\} $, we claim that $D_{2}^{*}\left(d,d',k\right)=k\cdot\left(2^{d+1-d'}-1\right)$ satisfies the required properties, and in particular, we may choose $D_{2}\left(d,k\right)=k\cdot2^{d+1}$. Suppose for the sake of contradiction that a tuple $\left(d,d',k\right)$ violating the conditions exists for this particular choice of $D_{2}^{*}\left(d,d',k\right)$, and choose a pair $\left(d,d'\right)$ among those for which a counterexample exists that minimizes $d-d'$, and among those with the same minimum value we choose one with minimal $d$.
    
    Let $\sigma\in\mathcal{\widetilde{O}}_{k}^{\left[d\right]}$ and $I$ be an interval chosen so that they satisfy the conditions of the lemma, but violate its conclusion. Let $\tau_{\sigma}\in\mathcal{\widetilde{O}}_{k}$ and $\rho_{1},\dots,\rho_{|\tau_{\sigma}|}\in\mathcal{\widetilde{O}}_{k}^{\left[d-1\right]}$ be chosen so that $\tau_{\sigma}\left[\rho_{1},\dots,\rho_{|\tau_{\sigma}|}\right]$ is the left-greedy decomposition of $\sigma \in \mathcal{\widetilde{O}}_{k}^{\left[d\right]}$. We also write $a$ and $b$ for the greatest and least indices chosen respectively so that every point in the interval $I$ is contained among the permutations $\rho_{a},\dots,\rho_{b}$. Since $I$ is an interval, it follows that all the points in $\rho_{a+1},\dots,\rho_{b-1}$ are contained in $I$. 

    Our proof proceeds in two phases. In the first phase, we insert all the points in $\mu'$ in a suitable gap in the middle of the interval $I$. To execute this step, we need to split into several cases depending on the size of $b-a$ and on how large $d-d'$ is. This execution is often either very simple due to the properties of the permutation class $\mathcal{\widetilde{O}}_{k}^{\left[d''\right]}$, or via a (somewhat technical) application of the inductive hypothesis. In the second phase, we simply delete all the points which used to belong to the interval $I$ --- here we are using the fact that permutation classes are closed under taking subpatterns. 
    
    We now give a more detailed informal description of our proof techniques for the first phase in various subcases. When $d'=d-1$, the fact that $\tau\left[\theta_{1},\dots,\theta_{r}\right]$ is the left-greedy decomposition of $\mu$ can be used to prove that two of the permutations $\theta_{i}$ cannot both be fully contained inside the same permutation $\rho_{j}$. As a consequence, it is not too hard to check that $\tau_{\sigma}$ must be monotone, and must have the same direction of monotonicity as $\tau$. Thus, it is fairly easy to insert the points of $\mu'$ inside $I$ by simply increasing the length of $\tau_{\sigma}$. This is handled in Case 1.

    When $d'\leq d-2$, the structure of the proof depends on the size of $b-a$, breaking down into several subcases. If $b\geq a+2$, i.e.\ if there exists at least one $\rho_{a+1}$ which is fully contained in $I$, the argument becomes easy. As we trivially have $\mu'\in\widetilde{\mathcal{O}}_{k}^{\left[d'+1\right]}\subseteq\mathcal{\widetilde{O}}_{k}^{\left[d-1\right]}$,
    we take a slight short-cut: we may simply replace $\rho_{a+1}$ with $\mu'$, causing the deletion of the points of $I \cap \rho_{a+1}$ already at this stage. We then proceed to delete the remaining points of $I$, which completes the proof. On the other hand if $b=a$, it is not too hard to conclude that the pair $\rho_{a}$ and $I$ would give a contradiction for the smaller pair $\left(d-1,d'\right)$. Finally, in the case when $b=a+1$, we note that at least one of $\rho_{a}$ or $\rho_{b}$ must fully contain many of the permutations $\theta_{i}$, in fact at least $D_{2}^{*}\left(d-1,d',k\right)$ of them. Thus, we may use the fact that the lemma holds for $\left(d-1,d',k\right)$ to construct such a permutation $\sigma'$. While this inductive approach is rather clear by itself, the actual construction is unfortunately fairly technical, due to the fact that we want to embed the permutations $\phi_{i}$ somewhat in the middle of the interval. We end up initially constructing slightly larger permutations by using the minimality assumption, within which we find the necessary building blocks to obtain the permutation $\sigma'$ as a subpermutation. Having thus outlined our approach, we now fill in the details.

    \noindent \textbf{Case 1, $d'=d-1$: }
    we then have $D_2^*(d,d-1,k) = 3k$. If $d = 1$, this implies the permutation $\sigma$ and the interval $\mu$ must both be monotone permutations with the same direction of monotonicity, and the result follows trivially. Assume therefore $d>1$.
    
\begin{claim}\label{claim: no rich index}
   There is no  $i\in\left[\left|\tau_{\sigma}\right|\right]$ such that $\rho_{i}$ contains all the points in both $\theta_{j}$ and $\theta_{j+1}$.
\end{claim}
\begin{proof}
Indeed, if $\theta_{j}\oplus\theta_{j+1}$ is a subpattern of $\rho_{i}$ for some $i$, then $\theta_{j}\oplus\theta_{j+1}\in\widetilde{\mathcal{O}}_{k}^{\left[d-1\right]}$ as $\widetilde{\mathcal{O}}_{k}^{\left[d-1\right]}$ is closed under taking subpatterns. However, this contradicts the fact that $\tau\left[\theta_{1},\dots,\theta_{r}\right]$ is the left-greedy decomposition of $\mu$ when $\mu$ is viewed as an element of $\widetilde{\mathcal{O}}_{k}^{\left[d'+1\right]}=\widetilde{\mathcal{O}}_{k}^{\left[d\right]}$. 
\end{proof}

We now split into subcases based on the structure of the permutation $\tau_{\sigma}$. 

    First suppose that $\left|\tau_{\sigma}\right|\leq k$. Since there are at most $\left|\tau_{\sigma}\right|-1$ indices $j\in\left[r\right]$ for which $\theta_{j}$ is not fully contained in one of the permutations $\rho_{i}$, we conclude that $r\leq2\left|\tau_{\sigma}\right|-1\leq2k-1<D_2^*(d,d-1,k)$, which contradicts the fact that $r\geq D_2^*(d,d-1,k)$. Thus we cannot have $\left|\tau_{\sigma}\right|\leq k$. 

    We may therefore assume that  $\left|\tau_{\sigma}\right|>k$, and hence that $\tau_{\sigma}$ is a monotone permutation. Since $r \geq D_2^*(d,d-1,k) \geq 3$, we conclude $\theta_1$ cannot contain any points in $\rho_b$, as else both $\theta_{r-1}$ and $\theta_r$ would be contained in $\rho_b$, contradicting Claim~\ref{claim: no rich index}. Thus $\tau_{\sigma}$ must in fact be a monotone permutation, which is easily seen to have the same direction of monotonicity as $\tau$, and we must have $b \geq a+1$. 
    
    Let $\tau^*$ denote the monotone increasing permutation of length $|\tau_\sigma| + r'$, and let $\sigma''$ denote the permutation 
    \[
        \sigma''=\tau^{*}\left[\rho_{1},\dots,\rho_{a},\phi_{1},\dots,\phi_{r'},\rho_{a+1},\dots,\rho_{\left|\tau_{\sigma}\right|}\right].
    \]
     Note that each of the permutations in the inflation given above is trivially contained in $\widetilde{\mathcal{O}}_{k}^{\left[d-1\right]}$. Since the original interval $I$ contains points from both $\rho_a$ and $\rho_{a+1}$, the interval corresponding to $\mu'$ is indeed embedded inside $I$. Deleting all the points from $\sigma''$ contained in the original interval $I$ gives us a permutation $\sigma'$ satisfying the required properties. Thus $\sigma$ satisfies the conclusion of the Lemma, which is a contradiction.

    \noindent \textbf{Case 2, $d'<d-1$ and $b=a$:} It is straightforward to check that in this case, the permutation $\rho_{a}\in\widetilde{\mathcal{O}}_{k}^{\left[d-1\right]}$ and the interval $I\subseteq\rho_{a}$ yields a contradiction for the pair $\left(d-1,d'\right)$ as well, thus contradicting the minimality of $d-d'$. 

    At this stage, we remark that whenever $\tau_{\sigma}$ is monotone decreasing, we must have $b = a$ and thus Case 2 applies. Indeed, given an arbitrary point $p\in I$, let $i$ and $j$ be chosen so that $p$ belongs to both $\theta_{i}$ and $\rho_{j}$. Note that for any $i'<i$, every point $q$ in $\theta_{i'}$ satisfies $q<_{1}p$ and $q<_{2}p$ as $\tau$ is monotone increasing. Thus the condition $q<_{1}p$ implies that we must have $q\in\rho_{j'}$ for some $j'\leq j$, but the condition $q<_{2}p$ implies that we must have $q\in\rho_{j'}$ for some $j\leq j'$. Thus every $\theta_{i'}$ with $i'<i$ must be fully contained in $\rho_{j}$. An analogous argument shows that in fact every $\theta_{i'}$ with $i'>i$ is also contained in $\rho_{j}$, and thus $I$ is contained in $\rho_{j}$. Thus whenever $\tau_{\sigma}$ is decreasing, we must have $a = b$ and hence Case 2 applies. 

    \noindent  \textbf{Case 3, $d'<d-1$ and $b\geq a+2$:} 
    Since $d'<d-1$, we have $\mu' \in \widetilde{\mathcal{O}}_{k}^{\left[d'+1\right]} \subseteq \widetilde{\mathcal{O}}_{k}^{\left[d-1\right]}$. Consider the permutation $\sigma''$ given by 
    \[\sigma'' = \tau_{\sigma}\left[\rho_1,\dots,\rho_{a},\mu',\rho_{a+2},\dots,\rho_{|\tau_\sigma|}\right].\]
    Since $b \geq a+2$, we have deleted those points of $I$ which belong to $\rho_{a+1}$, and replaced them with points in $\mu'$ --- and in particular, the added points of $\mu'$ must be located somewhere in the interior of the old interval $I$, regardless of the profile  $\tau_{\sigma}$. Since $\mu' \in \widetilde{\mathcal{O}}_{k}^{\left[d-1\right]}$, we have $\sigma'' \in \widetilde{\mathcal{O}}_{k}^{\left[d\right]}$. Considering the subpattern $\sigma'$ of $\sigma''$ obtained by deleting all the remaining points of $\sigma''$ (i.e.\ the points not corresponding to $\mu'$) which belong to $I$ we conclude that $\sigma$ satisfies the conclusion of the Lemma, which is a contradiction.    

    \noindent \textbf{Case 4, $d'<d-1$ and $b=a+1$:} since $\tau$ is increasing, $I$ is an interval, and $\rho_a$ contains a point from $\theta_1$ while $\rho_{a+1}$ contains a point from $\theta_r$,  we must have $\tau_{\sigma}\left(a+1\right)=\tau_{\sigma}\left(a\right)+1$. Let $i_{a}$ denote the largest index chosen so that $\theta_{i_{a}}$ contains a point that belongs to $\rho_{a}$.

    Since $r \geq D_2^*(d,d',k) = 2D_2^*(d-1,d',k) +k \geq 2D_2^*(d-1,d',k) +1$, we conclude that either $i_a \geq D_2^*(d-1,d',k) + 1$, or $r- i_a \geq D_2^*(d-1,d',k) + 1$ --- note that by the choice of $i_a$, it is evident that $\rho_{a+1}$ fully contains at least $r-i_a$ of the permutations $\theta_i$.

    We split the proof into two subcases based on whether $i_a$ is large or not. If $i_a$ is large, we use the minimality of the counterexample to perform the desired addition inside $I \cap \rho_a$. Note that $I \cap \rho_a$ can still be written as a direct sum $\theta_1 \oplus \dots \oplus \theta_{i_a-1} \oplus (\theta_{i_a} \cap \rho_a)$ of $i_a$ permutations in $\widetilde{\mathcal{O}}_{k}^{\left[d'\right]}$. However, it is not guaranteed that this is the left-greedy decomposition of $I \cap \rho_a$ when viewed as a permutation in $\widetilde{\mathcal{O}}_{k}^{\left[d'+1\right]}$. Despite this, an easy argument ensures that the left-greedy decomposition of $I \cap \rho_a$ is a direct sum of at least $i_a-1$ permutations, thus allowing us to apply our inductive hypothesis. A similar argument can be used when $r - i_a$ is large, although the exact details are somewhat different. We separate the remainder of the proof into subcases.

    \noindent  \textbf{Case 4.1, $i_a \geq D_2^*(d-1,d',k) + 1$:}
    write  $\tau''[\varphi_1,\dots,\varphi_{|\tau''|}]$ for the left-greedy decomposition of  $\rho_a \cap I$ when viewed as a permutation in $\widetilde{\mathcal{O}}_{k}^{\left[d'+1\right]}$.
    \begin{claim} \label{claim: leftgreedy_a_a+1_1}
        The left-greedy decomposition of $\rho_a \cap I$ as a permutation in $\widetilde{\mathcal{O}}_{k}^{\left[d'+1\right]}$ is a direct sum of at least $i_a - 1$ permutations in $\widetilde{\mathcal{O}}_{k}^{\left[d'\right]}$. 
    \end{claim}
    \begin{proof}
        We first exclude the case that the profile of $\rho_a \cap I$ is not monotone increasing. This follows by using the same argument from Case 1 which proved that rich indices cannot exist. First suppose that the profile $\tau''$ of $\rho_a \cap I$ is either decreasing or has size at most $k$. Since $\theta_1 \oplus \dots \oplus \theta_{i_a-1}$ is a subpermutation of $\rho_a \cap I$ and $i_a-1 \geq D_2^*(d-1,d',k) \geq 3k$, it follows that there exists $i$ and $j$ so that all the points in both $\theta_i$ and $\theta_{i+1}$ are contained in $\varphi_j$. Indeed, if $\tau''$ is decreasing, then all the $\theta_i$'s must be contained in $\varphi_1$, and if $|\tau''| \leq k$, the statement follows easily from the pigeonhole principle. In either case it follows that $\theta_i \oplus \theta_{i+1}$ is a subpattern of a permutation in $\widetilde{\mathcal{O}}_{k}^{\left[d'\right]}$ and hence itself must be in $\widetilde{\mathcal{O}}_{k}^{\left[d' \right]}$. This contradicts the fact that $\tau[\theta_1,\dots,\theta_r]$ is the left-greedy decomposition of $\mu$. Thus $\tau''$ must be monotone increasing. 

        Suppose for the sake of contradiction that $|\tau''| \leq i_a -2$. Note that $\mu$ can be rewritten as a direct sum as 
        \begin{equation} \label{eq: direct sum split}
            (\rho_a \cap I) \oplus (\rho_{a+1} \cap I) = \varphi_1\oplus \dots \oplus\varphi_{|\tau''|} \oplus (\rho_{a+1} \cap \theta_{i_a}) \oplus \theta_{i_a+1} \oplus \dots \oplus \theta_r.
        \end{equation}
        Note that each block in the right-hand side is contained in $\widetilde{\mathcal{O}}_{k}^{\left[d'\right]}$, and there are $|\tau''| + (r-i_a+1) \leq r-1$ of them. This contradicts the fact that $\theta_1\oplus \dots \oplus \theta_r$ is the left-greedy decomposition of $\mu$.
    \end{proof}
\noindent We thus have $|\tau''| \geq i_a-1$.  We now perform the necessary addition by applying the statement of the Lemma to the interval $\rho_a \cap I$ inside $\rho_a$. Indeed, since the profile of $\rho_a \cap I$ is a monotone permutation of length at least $i_a-1 \geq D_2^*(d-1,d',k)$ by Claim~\ref{claim: leftgreedy_a_a+1_1} and the minimality of our putative counterexample, the statement of the Lemma applies. Hence, we conclude that the permutation $\rho_a'$ obtained by replacing the interval $\rho_a \cap I$ in $\rho_a$ with $\mu'$ is contained in $\widetilde{\mathcal{O}}_{k}^{\left[d-1\right]}$. Let $\rho_{a+1}'$ denote the permutation obtained by deleting all the points from $\rho_{a+1}$ which belong to $I$. Then $\sigma'$ given by 
    \[\sigma' = \tau_{\sigma}\left[\rho_1,\dots,\rho_{a-1},\rho_a',\rho_{a+1}',\rho_{a+2},\dots,\rho_{|\tau_\sigma|}\right]\] 
    satisfies the required conditions. Hence $\sigma$ satisfies the conclusion of the Lemma, which is the required contradiction. This completes the proof of Case 4.1. 
    
    \noindent  \textbf{Case 4.2, $i_a \leq D_2^*(d-1,d',k)$:}
    similarly to the previous case, we write  $\tau''[\varphi_1,\dots,\varphi_{|\tau''|}]$ for the left-greedy decomposition of  $\rho_{a+1} \cap I$ when viewed as a permutation in $\widetilde{\mathcal{O}}_{k}^{\left[d'+1\right]}$. Since $r \geq D_2^*(d,d',k) \geq 2D_2^*(d-1,d',k) +1$, it follows that $r-i_a \geq D_2^*(d-1,d',k) +1$. 
    \begin{claim}\label{claim: leftgreedy_a_a+1_2}
        The left-greedy decomposition of $\rho_{a+1} \cap I$ as a permutation in $\widetilde{\mathcal{O}}_{k}^{\left[d'+1\right]}$ is a direct sum of at least $r-i_a$ permutations in $\widetilde{\mathcal{O}}_{k}^{\left[d'\right]}$. 
    \end{claim}
    \begin{proof}
        The proof is very similar to the proof of Claim~\ref{claim: leftgreedy_a_a+1_1}. Indeed, the fact that $r-i_a \geq D_2^* (d-1,d',k)$ implies that the profile $\tau''$ must be monotone increasing. By rewriting the right-hand side of equation~\eqref{eq: direct sum split} as 
        \[
        (\rho_a \cap I) \oplus (\rho_{a+1} \cap I) = \theta_1 \oplus \dots \oplus (\rho_{a} \cap \theta_{i_a}) \oplus \varphi_1\oplus \dots \oplus\varphi_{|\tau''|},
        \]
    we analogously conclude that we must have $i_a + |\tau''| \geq r$. Thus the claim follows. 
    \end{proof}
    We now perform the necessary addition by applying the statement of the Lemma to the interval $\rho_{a+1} \cap I$ inside $\rho_{a+1}$. Analogously to Case 4.1, since the profile of $\rho_{a+1} \cap I$ is a monotone permutation of length at least $r-i_a > D_2^*(d-1,d',k) $ by Claim~\ref{claim: leftgreedy_a_a+1_2} and the minimality of our putative counterexample, the statement of the Lemma applies. Hence, we conclude that the permutation $\rho_{a+1}'$ obtained by replacing the interval $\rho_{a+1} \cap I$ in $\rho_{a+1}$ with $\mu'$ is contained in $\widetilde{\mathcal{O}}_{k}^{\left[d-1\right]}$. Let $\rho_{a}'$ denote the permutation obtained by deleting all the points from $\rho_{a}$ which belong to $I$. Then $\sigma'$ given by 
    \[\sigma' = \tau_{\sigma}\left[\rho_1,\dots,\rho_{a-1},\rho_a',\rho_{a+1}',\rho_{a+2},\dots,\rho_{|\tau_\sigma|}\right]\] 
    satisfies the required conditions. Hence $\sigma$ satisfies the conclusion of the Lemma, which is the required contradiction. This completes the proof of Case 4.2, and thus the proof of Case 4. 
\end{proof}
\noindent We record below the special case $d' = 0$ of Lemma~\ref{lemma: substitutions inside iteratively inflated alternating sequences}, which will allow us to extend monotone intervals inside cells corresponding to $\star$-entries. 

\begin{lemma}\label{lemma: can replace monotone intervals by monotone intervals in O_k^d}
    For all $d, k\in \mathbb{Z}_{>0}$, there exists $D_3 = D_3(d,k)\in \mathbb{Z}_{>0}$ such that for every $\sigma \in\widetilde{\mathcal{O}}_{k}^{[d]}$, if $\sigma$ contains a monotone interval $I$ of length at least $D_3$, then the permutation $\sigma'$ obtained from $\sigma$ by replacing the monotone interval $I$ in $\sigma$ with another monotone interval $J$ with the same monotonicity as $I$ satisfies $\sigma'\in \widetilde{\mathcal{O}}_{k}^{[d]}$, regardless of the length of $J$.
\end{lemma}
\begin{proof}
    We may simply take $D_3(d,k) = D_2^*(d,0,k)$, and apply Lemma~\ref{lemma: substitutions inside iteratively inflated alternating sequences} with $d' = 0$. 
\end{proof}

\subsection{Constructing structured deletion functions}

By the pigeon-hole principle, for every sufficiently large permutation $\sigma\in\cP\subseteq\mathrm{Geom}(M)_{d,k}$, the equivalence class $\sigma^{\#}$ of its left-greedy minimal representations as gridded figures contains at least one of the following: a long monotone sub-interval in a single cell, a long consecutive block of points consisting of alternating monotone permutations (corresponding to points in two cells of same column), a long consecutive block of values consisting of alternating monotone permutations (corresponding to points in cells of same row), or many points in a cell indexed by $\star$ corresponding to an iterated inflation of depth at most $d$. Our overall strategy in Subsections \ref{subsection: monotone intervals}-\ref{subsection: star entries} corresponding to these respective cases remains the same, with technical differences arising from where exactly the structured part lies in the left-greedy representation of a permutation in $\cP$. In all three cases, we start by constructing a first-order sentence which can be used to identify a sufficiently long structured  sub-interval of the permutation, which  in the first two cases corresponds to a monotone or an alternating monotone sub-interval. We then delete a large number of points from the middle of the interval: the reason for choosing the deleted points to be located in the middle of the interval is to ensure that the deletion respects the left-greedy representation and hence induces a structured deletion function; the technical part of this step is provided by Lemma~\ref{lemma: substitutions inside iteratively inflated alternating sequences} or \ref{lemma: can replace monotone intervals by monotone intervals in O_k^d} respectively.

For those permutations which contain many  points in cells indexed by $\star$, we need one additional ingredient. As such points correspond to an iterated inflation of depth at most $d$, we start by backtracking the inflation process, choosing at each depth a sufficiently large block to dive into and reach the next depth. As an outcome, we either end up with a monotone subpermutation (which are the only sufficiently long permutations of depth 1, i.e.\ elements of $\widetilde{\mathcal{O}}_{k}^{[1]}$), or we end up with a long monotone inflation whose long 'initial segment'  consists of `short' permutations. We can apply our deletion strategy to such objects, and delete a number of consecutive points from the middle of such a monotone inflated block. To carry out our strategy, we construct a first-order sentence which keeps track both of the backtracking process within the inflation structure, and of properties of an initial segment of the desired long monotone inflation. In some sense, a lot of the heavy lifting will already be done when we define both the backtracking process and the suitable properties that a long monotone inflation of small permutations must satisfy; these definitions are, however, somewhat technical. 
\subsubsection{Monotone sub-intervals}\label{subsection: monotone intervals}
In this subsection, we deal with the permutations whose left-greedy decomposition contains a long monotone sub-interval in a single cell, and whose left end-point can be determined by using a variable that occurs in a first-order sentence. For this purpose, we introduce the following definition. (Recall that given $a\in \sigma^{\#}$, $a_{1,i}$ is the $i$-th successor of $a$ in $\sigma^{\#}$ with respect to the $\leq_1$-order.)
\begin{definition}\label{def: monotone subpermutations}
    Let $C_0$ and $C_1$ be positive constants. We say that a first-order sentence $\Psi$ contains a $(C_0,C_1)$-structured monotone sub-interval in the $(i,j)$-th cell if the following conditions are satisfied: 
    \begin{enumerate}[(i)]
        \item there exists a first-order formula $\psi(x)$ so that we have $\Psi = \exists x: \psi(x)$, and for any $\sigma \models \Psi$ there exists a unique point $a$ in $\sigma$ for which we have $\sigma \models \psi (a)$; 
        \item the unique point $a$ is the $C_0$-th least point in $\sigma^\#_{i,-}$ with respect to the $\leq_1$-order;
        \item the points in $[a_{1,0}, a_{1,C_1-1}]_1$ all belong to $\sigma^\#_{i,j}$, and form a monotone sub-interval of $\sigma$. 
    \end{enumerate}
\end{definition}

Our main goal in this subsection is to prove that for any constant $L$, there exists a constant $C_1$ so that for any $C_0$ and for any first-order sentence $\Psi$ that contains a $(C_0,C_1)$-structured monotone sub-interval, there exists a structured deletion function which removes $L$ points. 
\begin{proposition}
\label{prop: structured deletion for monotone sub intervals}

    Let $L_1$ and $C_0$ denote given positive integers. Let $C_1 = \ell m \cdot D_3(d,k) + L_1 + 1$,  where $D_3(d,k)$ is the constant given by Lemma~\ref{lemma: can replace monotone intervals by monotone intervals in O_k^d}. Suppose that $\Psi$ is a first-order sentence which contains a $(C_0,C_1)$-structured monotone sub-interval in the $(i,j)$-th cell. Define $s_1$ to be the constant given by $s_1 = \ell m \cdot D_3(d,k)$, and let $f_\Psi$ denote the deletion function which removes the points $a_{1,s_1},\dots,a_{1,s_1+L_1-1}$. Then $f_{\Psi}$ is a structured deletion function. 
\end{proposition}

Combining Proposition~\ref{prop: structured deletion for monotone sub intervals} with the following Proposition allows us to easily deal with those permutations $\sigma$ which contain many points in a cell corresponding to an isolated $\pm 1$-entry. 
\begin{proposition} \label{prop: isolated monotone FOS}
	Let $(i,j)$ be chosen so that $M_{i,j} \in \{\pm 1\}$, and that it is the only non-zero entry in its respective row and column. Let $C_1$ be an arbitrary positive integer. Then there exists a first-order formula $\psi_{i,j}(x)$ so that we have $\sigma \models \psi_{i,j}(a)$ if and only if $a$ is the least point in $\sigma^\#_{i,j}$ with respect to $\leq_1$-order, and we have $|\sigma^\#_{i,j}| \geq C_1$.  
\end{proposition}
As the proof of Proposition~\ref{prop: isolated monotone FOS} involves quite technical constructions of first-order formulae, we defer the proof to the Appendix. We now move on to prove Proposition~\ref{prop: structured deletion for monotone sub intervals}.
\begin{proof}[Proof of  Proposition~\ref{prop: structured deletion for monotone sub intervals}.]
For convenience, we write $I$ for the monotone interval $[a_{1,0}, a_{1,C_1-1}]_1$, and we write $J_0$ for the interval consisting of the deleted points $[a_{1,s_1},a_{1,s_1+L_1-1}]_1$, and we set $J = I\setminus J_0$. We start by noticing that since $C_1 = \ell m \cdot D_3(d,k) + L_1 + 1$, it follows that all the deleted points are contained in $I$. Let $f_\Psi^\#$ denote the gridded version of the deletion function. Given $\sigma \in \cP^\Psi$, we note that $\sigma^\#$ and $f_\Psi^\#(\sigma)$ differ only in one cell, namely the one that contains the point $a$. Furthermore, the deleted points are part of a longer monotone sub-interval of $\sigma^\#$ which is fully contained in one specified cell, and the first deleted point is the $(s_1 + C_0)$-th point in the respective column. Thus, it easily follows that $f_\Psi^\#$ is an injection. 

Thus, it remains to prove that we have $\Pi^\# (f_\Psi(\sigma)) = f_\Psi^\#(\sigma)$ for every $\sigma \in \cP^\Psi$. In order to prove the statement, we consider four equivalence classes of gridded figures of specific interest that are closely related to $\sigma$ or $f_{\Psi}(\sigma)$ respectively. 

For convenience, we write $\theta := f_{\Psi}(\sigma)$. First, we choose $(F_{\sigma,\#},G)$ to be a representative from the equivalence class of $\sigma^\#$, $(F_{\theta},G)$ to be a representative from the equivalence class of $f_{\Psi}^\#(\sigma)$ and $(F_{\theta,\#},G)$ to be a representative from the equivalence class of $\Pi^\# (f_{\Psi}(\sigma))$. Observe that $\theta$ is the underlying permutation for both $(F_{\theta},G)$ and $(F_{\theta,\#},G)$, and by construction of the function $\Pi^\#$, $(F_{\theta,\#},G)$ is a representative of the equivalence class of $\theta^\#$. As we always use the standard gridding $G$, from now on we simply write $F$ for the gridded figure $(F,G)$.

Before constructing our fourth gridded figure, we make some structural observations. By the third property, we note that all three sets $I$, $J_0$ and $J$ are contained in the cell indexed by $(i,j)$. Thus, by construction of the function $f_\Psi^\#$, we certainly have 
\begin{align} \label{eq:first comparision}
\left(F_{\sigma,\#}\right)_{u,v} = (F_\theta)_{u,v}
\end{align}
for every $(u,v) \neq (i,j)$, where the equality means that the points in the respective cells are contained in identical positions. In addition, we have    
\begin{align} \label{eq: LG theta}
    F_{\theta, \#} \leq_{LG} F_{\theta}    
\end{align}
since $\theta$ is the underlying permutation for both of these gridded figures, and by definition, $F_{\theta,\#}$ is a representative of the equivalence class corresponding to $\theta^\#$. 

Let $(i_1,j_1)$ be the least cell under lexicographic order for which either $(F_{\theta,\#})_{i_1,j_1} \cap I$ contains at least two points and $M_{i_1,j_1} \in \{\pm 1\}$, or for which $(F_{\theta,\#})_{i_1,j_1}  \cap I$ contains at least $D_3(d,k)$ points and $M_{i_1,j_1} = \star$. Since $|I| = C_1 > \ell m \cdot D_3(d,k) + L_1$, and $F_{\theta,\#}$ contains exactly $|I| - L_1$ points of $I$ that were not deleted during the process, it follows that such $(i_1,j_1)$ indeed exists. 

Let $J'$ denote the sub-interval of $F_{\theta,\#}$ consisting of those points that belong to $J \cap \left(F_{\theta,\#}\right)_{i_1,j_1}$, and let $J''$ denote those points of $J$ in $F_{\theta,\#}$ that belong to cells $(u,v)$ for which we have $(i_1,j_1) <_{lex} (u,v)$. 

\begin{claim}\label{claim: J'' empty}
    We must have $J'' = \emptyset$. 
\end{claim}
\begin{proof}
    Suppose for the sake of contradiction that $J'' \neq \emptyset$. Let $F_{\theta,1}$ be obtained from $F_{\theta,\#}$ by replacing $J'$ with a monotone interval of length $|J'|+|J''|$ with the same direction of monotonicity, and deleting the original points of $J''$. As the inserted points correspond to the deleted points of $J'\cup J''$ in their original relative order, we have $\pi^\#(F_{\theta,1})=\theta$. This is certainly possible if $M_{i_1,j_1} \in \{\pm 1\}$, as the diagonal in the $(i_1,j_1)$-th cell has the correct direction of monotonicity. If $M_{i_1,j_1} = \star$, Lemma~\ref{lemma: can replace monotone intervals by monotone intervals in O_k^d} implies that such extension is indeed possible, as $F_{\theta,1}$ contains at least $D_3(d,k)$ points of $J$ inside $(i_1,j_1)$-th cell. 

    Hence it follows that $F_{\theta,1} \in \text{Geom}^\#(M)_{d,k}$, while we clearly have $F_{\theta,1} <_{LG} F_{\theta,\#}$, as $J'' \neq \emptyset$. This completes the proof of Claim~\ref{claim: J'' empty}. 
\end{proof}

Now, we are ready to define the fourth gridded figure of specific interest, which we denote by $(F_{\sigma,\uparrow},G)$. Let $F_{\sigma,\uparrow}$ denote the gridded figure obtained from $F_{\theta,\#}$ by extending the monotone interval $J'$ in cell $(i_1,j_1)$ with $|J_0|$ new points of the same monotonicity which are inserted in the gap prescribed by their original labels in $\sigma$, while leaving all surviving labeled points in their existing cells. It is easy to verify that $\pi^\#(F_{\sigma,\uparrow}) = \sigma$, and verifying that $F_{\sigma,\uparrow} \in \text{Geom}^\#(M)_{d,k}$ follows again by a straightforward application of Lemma~\ref{lemma: can replace monotone intervals by monotone intervals in O_k^d}. One should view $F_{\sigma,\uparrow}$ as a `lift' of the gridded permutation $F_{\theta,\#}$ which `reverses' the action of the function $f_{\Psi}^\#$. However, since deleting points from a gridded figure is in general an operator that interacts very poorly with the left-greedy decompositions, we need the fairly technical preliminary results (in this case, Lemma~\ref{lemma: can replace monotone intervals by monotone intervals in O_k^d}) to prove that such a lift even exists. 

Since the addition of the points affects only the cell $(i_1,j_1)$, it follows that 
\begin{align} \label{eq:second comparision}
\left(F_{\theta,\#}\right)_{u,v} = (F_{\sigma,\uparrow})_{u,v}
\end{align}
for every $(u,v) \neq (i_1,j_1)$. In addition, we have
\begin{align} \label{eq: LG sigma}
    F_{\sigma, \#} \leq_{LG} F_{\sigma,\uparrow}
\end{align}
as $\sigma$ is the underlying permutation for both of these gridded figures, and $F_{\sigma, \#}$ is a representative of the equivalence class containing the left-greedy decomposition. 

Thus, the equations~\eqref{eq:first comparision}, \eqref{eq:second comparision}, \eqref{eq: LG sigma} and \eqref{eq: LG theta} verify that the counting conditions of Lemma~\ref{lemma: square comparision}  (a) are satisfied. It is evident that when viewing $F_{0,0}$ and $F_{0,1}$ as labeled gridded figures, then $F_{1,0}$ and $F_{1,1}$ contain the same set of labels, and that the deleted points are not the least points in their respective cell. Thus the gridded figures $F_{\theta,\#}$ and $F_{\theta}$ must belong to the same equivalence classes. As our argument is clearly independent of the chosen representatives, we conclude that $\Pi^\#(f_\Psi(\sigma)) = f_\Psi^\#(\sigma)$ for every $\sigma$ with $\sigma \models \Psi$, and hence $f_\Psi$ is a structured deletion function. \end{proof}

\subsubsection{Alternating monotone sub-intervals}\label{subsection: alternating intervals}
In this subsection, we handle columns and rows of the matrix $M$ which contain two non-zero entries. We start by stating the result that ensures that the property of our interest can indeed be expressed by using a first-order formula.

\begin{lemma}\label{lemma: FOS alternating monotone sub intervals}
Suppose that for some $i$, the $i$-th column of $M$ contains exactly two non-zero entries, $M_{i, j_0}$ and $M_{i,j_1}$, with $j_0<j_1$. Then for all $K\in \mathbb{Z}_{>0}$ and all $\mathbf{\epsilon}=(\epsilon_s)_{s=0}^{K-1} \in \{0,1\}^K$, there exists a first-order sentence $\Psi_{i,\mathbf{\epsilon}}=\exists x\ : \psi_{i,\epsilon}(x)$ such that, for every $\sigma\in\mathcal{P}$ and every
$a\in\sigma$, we have $\sigma\models\psi_{i,\epsilon}(a)$
if and only if $a$ is the $\leq_1$-least point of
$\sigma^\#_{i,-}$, this column contains at least $K$ points,
and for every $s\in\{0,\dots,K-1\}$ we have
$a_{1,s}\in\sigma^\#_{i,j_{\epsilon_s}}$.
\end{lemma}
\noindent Again, as the proof is fairly technical, we postpone it to the Appendix. We now move on to proving that structured deletion functions exist for such first-order formulae. 
\begin{proposition}\label{prop: monotone alternating intervals}
 Suppose that for some $i$, the $i$-th column of $M$ contains exactly two non-zero entries, $M_{i,j_0}$ and $M_{i,j_1}$, with $j_0<j_1$. Given a positive integer $L_2$, there exists a positive integer $C_2$ depending only on $d$, $k$, $L_2$ and $M$ for which the following holds. 

 Given a sequence $\mathbf{\epsilon} \in \{0,1\}^{C_2}$, let $\Psi_{i,\epsilon} = \exists x: \psi_{i,\epsilon} (x)$ be a first-order sentence satisfying the conditions of Lemma~\ref{lemma: FOS alternating monotone sub intervals}. Then there exists some $s_2 \in \{0,\dots,C_2-L_2\}$ for which the function $f_{\Psi_{i,\epsilon}}$ deleting the points $a_{1,s_2},\dots,a_{1,s_2+L_2-1}$ is a structured deletion function. 
\end{proposition}
\begin{proof}
Let $C_1(L_2) := \ell m\cdot D_3(d,k) + L_2 + 1$  denote the constant $C_1$ in Proposition~\ref{prop: structured deletion for monotone sub intervals} when the value of $L_1$ is chosen to be $L_2$. Let $C_2$ be defined to be the constant 
\begin{align}\label{eq:C_2}
    C_2 := ((3+D_1(d))\max(\ell,m)+2)\cdot C_1(L_2)+L_2+3,
\end{align}
and fix a constant $s_2 = ((3 + D_1(d))\max(\ell,m) + 2) \cdot C_1(L_2) + 1$. 

Fix a sequence $\epsilon \in \{0,1\}^{C_2}$. We say that an index $t \in \{0,\dots, C_2 - C_1\}$ is monotone if $\epsilon_u = \epsilon_t$ for every $t \leq u \leq t + C_1 - 1$. If a monotone index $t$ exists, it is easy to check that $\Psi_{i,\epsilon}$ contains a $(t,C_1)$-structured monotone sub-interval in the $(i,j_{\epsilon_t})$-th cell. Thus, an application of Proposition~\ref{prop: structured deletion for monotone sub intervals} with $L_1 := L_2$ implies that the function $f_{\Psi_{i,\epsilon}}$ constructed in the Proposition is a structured deletion function which deletes $L_2$ points. 

 Otherwise, the sequence $\epsilon$ can only contain constant sub-intervals of length at most $C_1 - 1$, and thus by the pigeonhole principle, the sequence $\epsilon$ must change its sign at least $\frac{s_2}{C_1-1} > (3 + D_1(d))\max(\ell,m) + 2$ times, and it is easy to check that for any $\sigma$ with $\sigma \in \cP^{\Psi_{i,\epsilon}}$, there are at least $(3 + D_1(d))\max(\ell,m) + 2$ alternations within the points $a_{1,0},\dots,a_{1,s_2-1}$. 

 Let $s_2^*$ denote the index among these points at which the $((3 + D_1(d))\max(\ell,m) + 1)$-th alternation occurs, and thus we must have $s_2^*  < s_2$. Let $f_{\Psi_{i,\epsilon}}$ denote the deletion function which deletes the points $a_{1,s_2},\dots,a_{1,s_2+L_2-1}$ from a given $\sigma \in \cP^{\Psi_{i,\epsilon}}$. We claim that $f_{\Psi_{i,\epsilon}}$ is a structured deletion function. As usual, we write $f^{\#}_{\Psi_{i,\epsilon}}$ for the gridded version of the function $f_{\Psi_{i,\epsilon}}$. 

Fix a permutation $\sigma \in \cP^{\Psi_{i,\epsilon}}$. As in the proof of Proposition~\ref{prop: structured deletion for monotone sub intervals}, we write $I$ for the points in $[a_{1,0},a_{1,C_2-1}]_1$, $J_0$ for the set of deleted points, and $J$ for $I\setminus J_0$. 

We now follow the approach used in the proof of Proposition~\ref{prop: structured deletion for monotone sub intervals}. We write $\theta := f_{\Psi_{i,\epsilon}}(\sigma)$, and choose $F_{\sigma,\#} \in \sigma^\#$, $F_{\theta} \in  f_{\Psi_{i,\epsilon}}^\#(\sigma)$ and $F_{\theta,\#} \in \Pi^\# (f_{\Psi_{i,\epsilon}}(\sigma))$. As usual, we use the standard gridding $G$, which we omit from the notation. Again, $\theta$ is the underlying permutation for both $F_{\theta}$ and $F_{\theta,\#}$, and by construction of the function $\Pi^\#$, $F_{\theta,\#}$ is a representative of the equivalence class of $\theta^\#$. 

Let $i'$ denote the least column that contains two non-zero entries $(i',j'_0)$ and $(i',j'_1)$ with $j'_0<j'_1$, and so that $M_{i,j_s} =M_{i',j'_s}$ for both $s \in \{0,1\}$, and for which $(F_{\theta,\#})_{i',-}$ contains at least two points in $J$ in both of its cells. To prove that such a column exists, we note that (i) a column $i''$ that contains a single non-zero $\pm 1$ entry can contribute at most $1$ alternation, (ii) a column $i''$ that  contains two non-zero entries whose sign-patterns do not align with the ones of the $i$-th column can contribute at most 2 alternations (else both cells contain at least two points, which contradicts the fact that at least one of their directions is incorrectly aligned), (iii) a column whose sign-patterns align with the $i$-th column but contains at most one point in one of its cells can contribute at most $3$ alternations, (iv) a column $i''$ that contains a $\star$-entry can contribute at most $D_{1}(d)$ alternations by Lemma~\ref{lemma: bound on number of alternations in O_k^d}, and (v) there can be one alternation occurring between every two consecutive columns. 

In particular, if no such $i'$ exists, the number of alternations in $J$ must be at most $(3 + D_1(d))\max(\ell,m)$, which contradicts the fact that such many alternations are already contained among the points in $[a_{1,0},a_{1,s_2}]_1$. Thus, such $i'$ must exist, and in fact the point $a_{1,s_2^*}$ must belong to the $i'$-th column or a column that occurs after the $i'$-th column. Also, recall that we have $s_2^*< s_2$. 

As in the proof of Proposition~\ref{prop: structured deletion for monotone sub intervals}, we write $J'$ for the intersection $J \cap (F_{\theta,\#})_{i',-}$  and $J''$ for all the points in $J\cap F_{\theta,\#}$ which are contained in a column indexed by $u > i'$. 
\begin{claim} \label{claim: second J'' claim}
    We must have $J'' = \emptyset$. 
\end{claim}
\begin{proof}
    The proof is analogous to the proof of Claim~\ref{claim: J'' empty}. Suppose for the sake of contradiction that $J'' \neq \emptyset$. Let $s' \in \{0,\dots,C_2-1\}$ denote the largest index of a point $a_{1,s'}$ which belongs to $(F_{\theta,\#})_{i',-}$. Let $F_{\theta,1}$ denote the gridded figure obtained by deleting all the points in $J''$, and placing each of the deleted points $a_{1,u}$ in the $(i',j'_{\epsilon_u})$-th cell of $F_{\theta,1}$ so that each newly inserted point is placed after $a_{1,s'}$, and so that internally they retain their relative relations with respect to $\leq_1$.
    
    We certainly have $\pi^\#(F_{\theta,1}) = \theta$, as the sign-patterns in the $i'$-th and $i$-th columns coincide, and $J\cap (F_{\theta,1})$ already contains at least two points in both of the cells. As we also have $F_{\theta,1} \in \text{Geom}^\#(M)_{d,k}$, we obtain the required contradiction analogously. 
\end{proof}
Thus, by Claim~\ref{claim: second J'' claim} and the observation preceding it, we conclude that the points $a_{1,s_2-1}$ and $a_{1,s_2+L_2}$, i.e.\ the points preceding and succeeding the points in $J_0$ with respect to $\leq_1$-order, must both belong to the column $i'$ in $F_{\theta,\#}$. Let $F_{\sigma,\uparrow}$ denote the gridded figure obtained by adding points $b_{s_2},\dots,b_{s_2+L_2-1}$ in this order in the $i'$-th column of $(F_{\theta,\#})$ so that they are placed in the gap between $a_{1,s_2-1}$ and $a_{1,s_2+L_2}$, and so that $b_{u}$ is placed in the cell $C_{i',j'_{\epsilon_u}}$ for all $u$. Since $F_{\theta,\#}$ already contained at least two points in both of the cells indexed by $(i',j'_0)$ and $(i',j'_1)$, their directions must be the correct ones, and thus it is easy to conclude that $\sigma$ is the underlying permutation for $F_{\sigma,\uparrow}$. 
 
As in the proof of Proposition~\ref{prop: structured deletion for monotone sub intervals}, it is easy to verify that the counting conditions of Lemma~\ref{lemma: square comparision}  (b) are satisfied with $\mathcal{S} = \{(i,j_0),(i,j_1)\}$ and $\mathcal{S'} = \{(i',j'_0),(i',j'_1)\}$ - and it is not too hard to verify that the other conditions are satisfied as well. Hence it follows that $f_{\Psi_{i,\epsilon}}$ is a structured deletion function. \end{proof}
\noindent An analogous argument can be used to yield a row-analogue of Proposition~\ref{prop: monotone alternating intervals}. However, the details need to be slightly adjusted as the left-greedy order is not symmetric under taking transpositions. We omit the details. 
\begin{corollary}\label{cor: monotone alternating intervals, row version}
The statement of Proposition~\ref{prop: monotone alternating intervals} holds when interchanging the roles of columns and rows and replacing the $\leq_1$-order by the $\leq_2$-order. 
\end{corollary}
\subsubsection{Iterated inflations of alternating monotone sub-intervals}\label{subsection: star entries}

In this subsection, we deal with permutations that contain a large sub-interval which corresponds to an iterative inflation. A crucial part of our argument is the notion of so-called \textit{$B$ iterative sequences}, which we will soon define. Intuitively, the aim is to  descend through the left-greedy decompositions, and always pass to a sub-interval $\sigma_i$ arising as a block in the left-greedy decomposition at a given depth which has many points. For technical reasons, we fix a large constant $B$, and we always choose the $\sigma_i$ among the first $B$ blocks with least $i$ that contains at least a certain number of points, where this threshold depends only on $B$ and the current depth. In addition, we also want to record certain properties of this backtracking by using first-order sentences, and we want the number of points occurring before $\sigma_i$ to be uniformly bounded. These conditions are the reason for introducing the cut-off at certain point $B$. More precisely, at a given depth, we record the profile $\tau$, the index $i$, and all the small permutations $\sigma_1,\dots,\sigma_{i-1}$ which precede $\sigma_i$. 

If we can carry out the argument described above until we reach depth $d = 1$, we simply end up with a long monotone sub-interval. However, there is an obvious challenge why executing such strategy may fail: if the profile $\tau$ is monotone permutation, it may be the case that all the permutations $\theta_1,\dots,\theta_B$ are too small for us to continue the descent through the iterative structure of permutations in $\mathcal{\widetilde{O}}_{k}^{\left[d\right]}$. Thus, we simply terminate the process if we end up with a sub-interval $\theta$ which is in $\mathcal{\widetilde{O}}_{k}^{\left[d'\right]}$ for some $d' \leq d$, and whose first $B$ blocks $\theta_1,\dots,\theta_B$ in the left-greedy decomposition $\tau[\theta_1,\dots,\theta_\ell]$ are small, and we develop appropriate tools to handle  such sub-structures in Lemma~\ref{lemma: structured deletion for stars}. 

We now move on to the formal definition. For a sequence $\underline{u}$ of length $i$ and for a constant $u'$, we write $\underline{u} + u'$ for the sequence of length $i+1$ whose first $i$ entries correspond to those of $\underline{u}$, and the last entry equals $u'$. Suppose that $\sigma\in\mathcal{\widetilde{O}}_{k}^{\left[d\right]}$ is a given permutation. For a certain collection of sequences, we iteratively define the permutations $\theta_{\underline{u}}^{\left[i\right]}\left(\sigma\right)\in\mathcal{\widetilde{O}}_{k}^{\left[i\right]}$ and $\tau_{\underline{u}}\left(\sigma\right)\in\widetilde{\mathcal{O}}_{k}$ for a sequence $\underline{u}$ of length $d-i$ as follows. Initially, we set $\theta_{\emptyset}^{\left[d\right]}\left(\sigma\right)=\sigma$. Having defined $\theta_{\underline{u}}^{\left[i\right]}\left(\sigma\right)\in\mathcal{\widetilde{O}}_{k}^{\left[i\right]}$ for some sequence $\underline{u}$, let $\tau\in\widetilde{\mathcal{O}}_{k}$ and $\phi_{1},\dots,\phi_{\left|\tau\right|}\in\widetilde{\mathcal{O}}_{k}^{\left[i-1\right]}$ be chosen so that $\tau\left[\phi_{1},\dots,\phi_{\left|\tau\right|}\right]$ is the left-greedy decomposition of $\theta_{\underline{u}}^{\left[i\right]}\left(\sigma\right)$ when it is viewed as an element of $\mathcal{\widetilde{O}}_{k}^{\left[i\right]}$. In this case, we set $\theta_{\underline{u} +1}^{\left[i-1\right]}\left(\sigma\right)=\phi_{1},\dots,$$\theta_{\underline{u}+\left|\tau\right|}^{\left[i-1\right]}\left(\sigma\right)=\phi_{\left|\tau\right|}$, and we set $\tau_{\underline{u}}\left(\sigma\right)=\tau$. Note that we define the permutations $\tau_{\underline{u'}}\left(\sigma\right)$ and $\theta_{\underline{u'}}^{\left[i\right]}\left(\sigma\right)$ for only those sequences $\underline{u}'$ that can be obtained via such process. Note that we use the convention that the left-greedy decomposition of $\tau\in\widetilde{\mathcal{O}}_{k}$ is $\tau=\tau\left[1,\dots,1\right]$, where $1$ denotes the trivial permutation.

Given a fixed integer $B>k$, we write $\widetilde{\mathcal{O}}_{k,B}$ for the union of $\mathcal{O}_{k}$ together with all monotone permutations of length at most $B$. Equivalently, $\widetilde{\mathcal{O}}_{k,B}$ consists of all the permutations in $\widetilde{\mathcal{O}}_{k}$ of length at most $B$.   
\begin{definition}
    Given $t\in\left\{ 1,\dots,d\right\} $, a $B$ iterative sequence of depth $t$ is a sequence $\epsilon = \left(\tau_{i},c_{i},\underline{\phi}_{i}\right)_{i=t}^{d}$ satisfying the following conditions
\begin{enumerate}
\item For all $i\geq t+1$, we have 
\begin{enumerate}
\item $\tau_{i}\in\widetilde{\mathcal{O}}_{k,B}\cup\left\{\mathrm{inc}, \mathrm{dec}\right\} $, and 
\item Whenever $\tau_{i}\not\in\left\{\mathrm{inc}, \mathrm{dec}\right\} $, we have $c_{i}\in\left\{ 1,\dots,\left|\tau_{i}\right|\right\} $ and whenever $\tau_{i}\in\left\{\mathrm{inc}, \mathrm{dec}\right\} $, we have $c_{i}\in\left\{ 1,\dots,B\right\} $. 
\end{enumerate}
\item When $i=t$, we have $\tau_{t}\in\left\{\mathrm{inc}, \mathrm{dec}\right\} $ and $c_{t}=B+1$.
\item In either case, the sequence $\underline{\phi}_{i}$ contains $c_{i}-1$ permutations $\phi_{i,1}^{\left[i-1\right]},\dots,\phi_{i,c_{i}-1}^{\left[i-1\right]}$, which are all contained in $\widetilde{\mathcal{O}}_{k}^{\left[i-1\right]}$ and have at most $B^{i-1}$ points. In particular, the sequence is empty if $c_{i}=1$. 
\end{enumerate}
\end{definition}
Before moving on, we remark some technicalities regarding how monotone profiles are treated under this definition, and briefly motivate our choice which may look slightly odd. Note that separating the profiles based on whether their length exceeds $B$ or not has the slightly unintended consequence that when the profile $\tau$ is a monotone, we choose the representative $\tau_i$ slightly differently depending on the length of $\tau$. Indeed, we choose $\tau_i = \tau$ if $|\tau| \leq B$, and otherwise we simply let $\tau_i$ be the symbol $\mathrm{inc}$ or $\mathrm{dec}$ depending on the direction of monotonicity.  The reason for having universal labels $\mathrm{inc}$ and $\mathrm{dec}$ for all sufficiently long monotone permutations is to ensure that the set of $B$ iterative sequences is finite, which is needed to express certain conditions as first-order sentences. 

However, the reason why we separate short monotone profiles from long ones is that when we introduce the cut-off and consider only first $B$ blocks in the profile, it is somewhat more convenient to actually write down a formula with the assumption that the number of blocks is at least $B+1$, as we also need variables representing the extremal points for these blocks for bookkeeping purposes. In some sense, one should view $\mathrm{inc}$ and $\mathrm{dec}$ as the genuine representatives of monotone profiles, even though they are only used to represent the long monotone profiles, as for short monotone profiles we know a sufficiently large block must always exist due to a standard pigeonhole argument.

When the constant $B$ is clear from the context, we simply refer to $B$ iterative sequence as an iterative sequence. Given a $B$ iterative sequence $\epsilon = \left(\tau_{i},c_{i},\underline{\phi}_{i}\right)_{i=t}^{d}$, for each $i\in\left\{ t,\dots,d\right\} $ we write $\underline{c}_{>i}=\left(c_{d},\dots,c_{i+1}\right)$, with the convention that $\underline{c}_{>d}$ is the empty sequence. 

\begin{definition}
    We say that $\sigma$ \emph{respects the sequence} $\epsilon$ if for each $i\in\left\{ t,\dots,d\right\} $, the following conditions are satisfied. 
\begin{enumerate}
\item Whenever $\tau_i \not \in \{\mathrm{inc},\mathrm{dec}\}$, we have $\tau_{\underline{c}_{>i}}\left(\sigma\right)=\tau_{i}$ and $c_{i}$ is the least index $s$ for which we have $\left|\theta_{\underline{c}_{>i}+s}^{\left[i-1\right]}(\sigma)\right|\geq B^{i-1} + 1$. 
\item Whenever $\tau_{i}\in\left\{\mathrm{inc}, \mathrm{dec}\right\} $ and $c_{i}\neq B+1$, the permutation $\tau_{\underline{c}_{>i}}\left(\sigma\right)$ is a monotone permutation of length at least $B+1$ whose direction is given by $\tau_{i}$, and $c_{i}$ is the least index $s$ for which we have $\left|\theta_{\underline{c}_{>i}+s}^{\left[i-1\right]}(\sigma)\right|\geq B^{i-1} + 1$. 
\item Whenever $\tau_{i}\in\left\{\mathrm{inc}, \mathrm{dec}\right\} $ and $c_{i}= B+1$, the permutation $\tau_{\underline{c}_{>i}}\left(\sigma\right)$ is a monotone permutation of length at least $B+1$ whose direction is given by $\tau_{i}$, and we have $\left|\theta_{\underline{c}_{>i}+j}^{\left[i-1\right]}(\sigma)\right|\leq B^{i-1}$ for every $1\leq j\leq B$. 
\item For all $i\in\{t,\dots,d\}$ and
$1\leq j\leq c_i-1$, we have
$\theta_{\underline{c}_{>i}+j}^{[i-1]}(\sigma)
=\phi_{i,j}^{[i-1]}$.
\end{enumerate}
\end{definition}

It is not too hard to check that the number of $B$ iterative sequences of depth $t$ can be bounded as a function of $d$, $k$ and $B$. Indeed, this follows from the fact that the number of choices for each $\tau_{i},c_{i},\underline{\phi}_{i}$ is bounded as a function of $B$ for every $i$, and the depth is also bounded. We now prove that every $\sigma\in\widetilde{\mathcal{O}}_{k}^{\left[d\right]}$ with $\text{\ensuremath{\left|\sigma\right|}\ensuremath{\ensuremath{\geq B^{d}+1}}}$ respects some sequence $\epsilon$. 
\begin{lemma}\label{lemma: iterative sequences cover all}
Let $B>k$ be an integer. Given $\sigma\in\widetilde{\mathcal{O}}_{k}^{\left[d\right]}$ with $\text{\ensuremath{\left|\sigma\right|}\ensuremath{\ensuremath{\geq B^{d}+1}}}$, there exists a $B$ iterative sequence $\epsilon$ so that $\sigma$ respects $\epsilon$. 
\end{lemma}

\begin{proof}
We prove the statement by induction on $d$. When $d=1$, we have $\left|\sigma\right|\geq B+1>k$, and thus $\sigma$ is a monotone permutation. Choose  $\tau_{1}\in\left\{\mathrm{inc},\mathrm{dec}\right\} $ to match the direction of $\sigma$, $\phi_{1,j}$ to be the trivial permutation for all $1\leq j\leq B$, and $c_{1}=B+1$. Then the iterative sequence $\epsilon$ of depth $1$ consisting of this triple satisfies the required property.

Now suppose that $d\geq2$, $\sigma\in\widetilde{\mathcal{O}}_{k}^{\left[d\right]}$ satisfies $\text{\ensuremath{\left|\sigma\right|}\ensuremath{\ensuremath{\geq B^{d}+1}}}$ and that the statement holds for $d-1$, and let $\tau\left[\theta_{1},\dots,\theta_{\left|\tau\right|}\right]$ be the left-greedy decomposition of $\sigma$. If there exists $i\in\left\{ 1,\dots,\min\left(B,\left|\tau\right|\right)\right\}$ for which we have $\left|\theta_{i}\right|\geq B^{d-1}+1$, we choose minimal such $i$ and set $c_{d}=i$. We also choose $\tau_{d}=\tau$ or $\tau_{d}\in\left\{ \mathrm{inc},\mathrm{dec}\right\} $ to match the monotonicity of $\tau$ depending on whether we have $\left|\tau\right|\leq B$ or $\left|\tau\right|>B$ respectively. Finally, we set $\phi_{d,j}=\theta_{j}$ for every $1\leq j\leq c_{d}-1$. In either case, we may use the inductive hypothesis to construct the remaining part of the sequence $\left(\tau_{i},c_{i},\underline{\phi_{i}}\right)_{i=t}^{d-1}$. 

Otherwise, we must have $\left|\theta_{i}\right|\leq B^{d-1}$ for every $i\in\left\{ 1,\dots,\min\left(B,\left|\tau\right|\right)\right\} $,
and in particular, it follows that 
\[
\sum_{i=1}^{\min\left(B,\left|\tau\right|\right)}\left|\theta_{i}\right|\leq B\cdot B^{d-1}\leq B^{d}<\left|\sigma\right|.
\]
Thus, we must have $\left|\tau\right|>B$. Hence, we may choose $\tau_{d}$ to match the direction of $\tau$, $c_{d}=B+1$ and $\phi_{d,i}=\theta_{i}$ for every $1\leq i\leq B$. Then $\sigma$ respects the sequence $\left(\tau_{i},c_{i},\underline{\phi_{i}}\right)_{i=d}^{d}$, as required. 
\end{proof}
\begin{lemma} \label{lemma: FOS structure for stars}
Let $B > k$ be an integer, let $\left(i',j'\right)$ be chosen so that $M_{i',j'}=\star$, let $\epsilon=\left(\tau_{i},c_{i},\underline{\phi_{i}}\right)_{i=t}^{d}$ be a $B$ iterative sequence, and let $\mathcal{S}=\left\{ \left(u,v\right):u\in\left\{ t,\dots,d\right\} ,v\in\left\{ 1,\dots,\min\left(c_{u},B\right)\right\} \right\} $. Then there exists a first-order sentence 
\[
\Psi_{\epsilon,B,\left(i',j'\right)}=\exists_{\left(u,v\right)\in\mathcal{S}}x_{u,v},y_{u,v}:\psi_{\epsilon,B,\left(i',j'\right)}\left(x_{u,v},y_{u,v}\right)
\]
so that, for every $\sigma\in\mathcal{P}$, we have $\sigma\models\Psi_{\epsilon,B,(i',j')}$ if and only if $|\sigma^\#_{i',j'}|\geq B^d+1$ and $\sigma^\#_{i',j'}$ respects $\epsilon$. Whenever these conditions hold, there is a unique tuple $(a_{u,v},b_{u,v})_{(u,v)\in\mathcal{S}}$ satisfying $\psi_{\epsilon,B,(i',j')}$, and this tuple satisfies the following properties:
\begin{enumerate}
\item For each $\left(u,v\right)\in\mathcal{S}$ with $v\leq c_{u}-1$, the permutation $\theta_{\underline{c}_{>u}+v}^{\left[u-1\right]}\left(\sigma_{\left(i',j'\right)}^{\#}\right)$ (which equals $\phi_{u,v}$) consists of the points in $\left[a_{u,v},b_{u,v}\right]_{1}$.  
\item For each $u\in\{t+1,\dots,d\}$, the permutation
$\theta_{\underline{c}_{>u}+c_u}^{[u-1]}
\left(\sigma^\#_{i',j'}\right)$
(which contains at least $B^{u-1}+1$ points) consists of the
points in $[a_{u,c_u},b_{u,c_u}]_1$.
\end{enumerate}
\end{lemma}
\noindent Informally, Lemma~\ref{lemma: FOS structure for stars} states the following. By using first-order formulae, it is possible to track the endpoints of the `large' iterative permutations $\theta_{\underline{c_{>i}}}^{\left[i\right]}\left(\sigma_{\left(i',j'\right)}^{\#}\right)$ corresponding to an iterative sequence $\epsilon$, and it is also possible to track the endpoints of the preceding small permutations through the process. In addition, it is possible to track the end-points of the first $B$ permutations in the left-greedy decomposition occurring at the depth $t$ by using a first-order formula, and we use $\theta_{1},\dots,\theta_{B}$ to denote these permutations.

As the proof of Lemma~\ref{lemma: FOS structure for stars} is a rather technical construction, we defer it to the Appendix. 
\begin{lemma}\label{lemma: structured deletion for stars}
Let $L_{3}$ be a positive integer. Then there exists $B$ for which the following holds. Let $\epsilon=\left(\tau_{i},c_{i},\underline{\phi_{i}}\right)_{i=t}^{d}$ be a $B$ iterative sequence of depth $t$. Then there exist a constant $s_{3}$ which is determined by $\epsilon$, and which can be bounded above as a function of $d$, $k$ and $B$, and a structured deletion function $f$ for the first-order sentence $\Psi_{\epsilon,B,\left(i,j\right)}$ which deletes at least $L_{3}$ points, the first one being the $s_{3}$-th least point in the $\sigma_{\left(i,j\right)}^{\#}$ with respect to $\leq_{1}$-order. 
\end{lemma}

\begin{proof}
We claim that 
\[B=\max\left(\left(5+D_{2}\left(d,k\right)\right)\max\left(\ell,m\right)+L_{3}+4, \ell m \cdot D_3(d,k)+L_3+1\right)\]
works, but we remark that we made no effort to actually optimize this constant. First assume that $t \geq 2$; in this case we only need the fact that $B \geq \left(5+D_{2}\left(d,k\right)\right)\max\left(\ell,m\right)+L_{3}+4$. Throughout the proof, whenever we say that $p$ is the $t^{th}$ least  point in a cell or a column, we mean that $p$ is the $t^{th}$ least point in a cell or a column with respect to order $\leq_{1}$. Fix a permutation $\sigma \in \cP^{\Psi_{\epsilon,B,\left(i,j\right)}}$. Let $s_0$ denote the position of $a_{t,1}$ in $\sigma^\#_{i,j}$ with respect to the $\leq_1$-order. Since $\epsilon$ specifies all the small blocks preceding the larger block, by backtracking through the iterative structure it is not too hard to check that 
\[
s_0
=
1+\sum_{u=t+1}^{d}\sum_{v=1}^{c_u-1}
|\phi_{u,v}^{[u-1]}|.
\]
In particular, $s_0$ is independent of $\sigma$ (and is
completely determined by $\epsilon$).

Let $\tau\in\widetilde{\mathcal{O}}_{k}$ and $\theta_{i}\in\widetilde{\mathcal{O}}_{k}^{\left[t-1\right]}$ be chosen so that $\tau\left[\theta_{1},\dots,\theta_{\left|\tau\right|}\right]$ is the left-greedy decomposition of $\theta_{\underline{c}_{>t}}^{\left[t\right]}\left(\sigma_{\left(i,j\right)}^{\#}\right)$. Thus, $\tau$ must be a monotone permutation whose direction is given by $\tau_{t}$, and $\theta_{1},\dots,\theta_{B}$ are uniquely determined, as they must coincide with $\phi_{t,1},\dots,\phi_{t,B}$. Let $t_3=(5+D_2(d,k))\max(\ell,m)+1$, and put
\[
s_3=s_0+\sum_{v=1}^{t_3-1}|\phi_{t,v}^{[t-1]}|.
\]
Then the least point in $\theta_{t_3}$ is the $s_3$-th least
point in $\sigma^\#_{i,j}$ with respect to the $\leq_1$-order.
Again, $s_3$ is determined by $\epsilon$, and is bounded above
by a function of $d$, $k$ and $B$.

Let $f$ denote the function which deletes all the points from $\sigma$ that are contained in one of the blocks corresponding to $\theta_{t_{3}},\dots,\theta_{t_{3}+L_{3}-1}$, and note that this corresponds to deleting a sub-interval of $\sigma$ containing the $t$-th least points in $\sigma_{\left(i,j\right)}^{\#}$ for $s_{3}\leq t\leq s_{3}+L'_{3}-1$, where $L'_{3}=\sum_{j=t_{3}}^{t_{3}+L_{3}-1}\left|\theta_{j}\right|\ge L_{3}$. Thus, $L'_3$ is determined by $\underline{\phi_{t}}$, and hence can be bounded by a function of $d$, $k$ and $B$. Let $f^{\#}$ denote the appropriate gridded deletion function. 

As in the proof of Proposition~\ref{prop: structured deletion for monotone sub intervals}, we write $\theta:=f\left(\sigma\right)$, and define the gridded figures $F_{\sigma,\#}$, $F_{\theta}$ and $F_{\theta,\#}$ by choosing $F_{\sigma,\#}\in\sigma^{\#}$, $F_{\theta}\in f^{\#}\left(\sigma\right)$ and $F_{\theta,\#}\in\theta^{\#} = \Pi^\#(f(\sigma))$. Let $i'$ denote the least column in $\theta^{\#}$ that fully contains at least $\max\left(D_{2}(d,k)+2,4\right)$ of the blocks $\theta_u$ with $u<t_{3}$. We now prove that such an index $i'$ indeed exists. 

Suppose for the sake of contradiction that no such $i'$ exists. Note that (i) at most $\max\left(\ell,m\right)$ of $\theta_{u}$ can contain points in two columns. (ii) Each column not containing a $\star$-entry can contain at most $3$ of them completely. Indeed, otherwise as $\tau$ is a monotone permutation, it would follow that there exists $\theta_u$ and $\theta_{u+1}$ so that $\theta_u$ is contained in some monotone cell $C$ and $\theta_{u+1}$ intersects $C$ non-trivially, and thus the direction of $C$ must also coincide with the direction of $\tau$. However, this means that at least one of the points from $\theta_{u+1}$ could be moved to the permutation $\theta_{u}$ contradicting the fact that $\tau[\theta_1,\dots,\theta_{|\tau|}]$ is a left-greedy decomposition. This contradicts the fact that $\tau\left[\theta_{1},\dots,\theta_{\left|\tau\right|}\right]$ is a left-greedy decomposition of the underlying permutation. (iii) Each column containing a $\star$-entry can contain at most $D_{2}\left(d,k\right)+1$ of them. 

Thus the number of $\theta_{u}$'s with $u<t_{3}$ would be at most 
\[\max\left(\ell,m\right) + \max\left(\ell,m\right)\cdot\max\left(3,D_{2}\left(d,k\right)+1\right)<t_{3}-1,\] which is a contradiction. Thus, such $i'$ must exist, and by the observations above, must contain a $\star$-entry. Let $\left(i',j'\right)$ be chosen so that $M_{i',j'}=\star$. In addition, we write $\theta_u,\dots,\theta_v$ for the blocks that are completely contained in the cell $C_{i',j'}$, and we recall that $v-u+1 \geq D_2(d,k)$. Write $\varphi = \tau_{v-u+1}[\theta_u,\dots,\theta_v]$, where $\tau_{v-u+1}$ is a monotone permutation of length $v-u+1$ whose direction coincides with the direction of $\tau$. Then $\tau_{v-u+1}[\theta_u,\dots,\theta_v]$ is the left-greedy decomposition of $\varphi$ - indeed, if a better inflation of blocks were possible, its profile would have to be clearly a monotone permutation whose direction coincides with the direction $\tau$, and thus we could simply replace the appropriate gap in the left-greedy decomposition $\tau[\theta_1,\dots,\theta_{|\tau|}]$ to obtain the desired contradiction. 

Let $I$ denote the set of points contained in $\theta_{1},\dots,\theta_{B}$, let $J_{0}$ denote the set of deleted points, and let $J=I\setminus J_{0}$. We write $J'$ for $J\cap\left(F_{\theta,\#}\right)_{i',j'}$, and $J''$ for the set of points in $J$ that are contained in $\left(F_{\theta,\#}\right)_{\left(u',v'\right)}$ with $u'>i'$. 
\begin{claim}\label{claim: again J'' claim}
    We have $J''=\emptyset$.
\end{claim}
\begin{proof}
Suppose for the sake of contradiction that $J''\neq\emptyset$. Let $I'$ denote the sub-interval of $\left(F_{\theta,\#}\right)_{\left(i',j'\right)}$ corresponding to the points located in the blocks $\theta_u,\dots,\theta_v$. As $v-u+1 \geq D_2(d,k)$ and $\tau_{v-u+1}[\theta_u,\dots,\theta_v]$ is the left-greedy decomposition of $\varphi$, the conditions of Lemma~\ref{lemma: substitutions inside iteratively inflated alternating sequences} are satisfied when $I$ is the interval corresponding to $\varphi$. Define
\[\mu' = \tau'\left[\theta_u,\dots,\theta_{t_3-1},\theta_{t_3+L_3},\dots, \theta_B\right]\]
where $\tau'$ is a monotone permutation of appropriate length whose direction of monotonicity coincides with $\tau$, and let $F_{\theta,1}$ denote the gridded figure obtained by replacing the interval $I$ with the permutation $\mu'$. Note that $F_{\theta,1} \in \text{Geom}^{\#}\left(M\right)_{d,k}$ by Lemma~\ref{lemma: substitutions inside iteratively inflated alternating sequences}. 

Let $F_{\theta,2}$ denote the sub-figure of $F_{\theta,1}$ obtained by deleting the points that belong to the original blocks $\theta_{u+1},\dots,\theta_{t_3-1},\theta_{t_3},\dots,\theta_B$ in $F_{\theta,\#}$. Then $F_{\theta,2} \in \text{Geom}^{\#}\left(M\right)_{d,k}$, and it is also easy to check that $\pi^{\#}\left(F_{\theta,2}\right)=\theta$ and $F_{\theta,2}<_{LG}F_{\theta,\#}$. This gives the required contradiction, and thus we must have $J'' = \emptyset$. 
\end{proof}

Now suppose that $J'' = \emptyset$. Define
\[\mu' = \tau'\left[\theta_u,\dots, \theta_B\right]\]
where $\tau'$ is a monotone permutation of appropriate length whose direction of monotonicity coincides with $\tau$. As in the proof of Claim~\ref{claim: again J'' claim}, we apply Lemma~\ref{lemma: substitutions inside iteratively inflated alternating sequences} to the interval $I'$, which is defined as in the proof of Claim~\ref{claim: again J'' claim}, i.e.\ consists of the blocks $\theta_u,\dots,\theta_v$. Let $F_{\theta,1}$ denote the gridded figure obtained by replacing interval $I'$ with the permutation $\mu'$. As in the proof of Claim~\ref{claim: again J'' claim}, it is easy to check that the conditions of Lemma~\ref{lemma: substitutions inside iteratively inflated alternating sequences} apply, and thus $F_{\theta,1} \in \text{Geom}^{\#}\left(M\right)_{d,k}$. Let $F_{\sigma,\uparrow}$ denote the  subpattern of $F_{\theta,1}$ obtained by deleting the points that belong to the original blocks $\theta_{v+1},\dots,\theta_{t_3-1},\theta_{t_3+L_3},\dots,\theta_B$. Then it is easy to see that $F_{\sigma,\uparrow}$ induces the permutation $\sigma$. 

Since $J'' = \emptyset$, it is easy to check that the gridded figures $F_{\sigma,\#}$, $F_{\theta}$, $F_{\theta,\#}$ and $F_{\sigma,\uparrow}$ satisfy the conditions of Lemma~\ref{lemma: square comparision}, and hence $F_{\theta}$ and $F_{\theta,\#}$ must be equivalent. This completes the proof when $t \geq 2$. 

When $t = 1$, we may follow similar approach - however, we have to be careful with the fact that $i'$ may not necessarily correspond to a column with $\star$. Using the fact that $B \geq \ell m \cdot D_3(d,k)+L_3+1$ and using Proposition~\ref{prop: structured deletion for monotone sub intervals} as appropriate to deal with those columns containing monotone cells, it is not difficult to modify the proof and conclude the result. We omit the details. 
\end{proof}

\subsection{Putting it together}\label{subsection: putting it together}
\begin{proof}[Proof of Theorem~\ref{theorem: technical cover}]
Fix $\varepsilon>0$ and set $\gamma := \frac{1+\sqrt{5}}{2}$. Let $L$ be a positive integer chosen so that $\gamma^{-L} < \varepsilon$. Let $C_1$ be the constant obtained when Proposition~\ref{prop: structured deletion for monotone sub intervals} is applied with $L_1 = L$, let $C_2$ be the constant obtained when Proposition~\ref{prop: monotone alternating intervals}
is applied with $L_2 = L$, and let $B$ be the constant obtained when Lemma~\ref{lemma: structured deletion for stars} is applied with $L_3 = L$. 

Let $\mathcal{C}$ denote the collection of columns which contain at least two non-zero non-star entries, let $\mathcal{R}$ denote the collection of rows which contain at least two non-zero non-star entries, let $\mathcal{I}$ denote the collection of cells that contain $\pm 1$-entry and are the only non-zero entry in their respective row and column, and let $\mathcal{S}$ denote the collection of cells containing $\star$-entry. 

Let $\mathbf{P}$ denote a finite collection of first-order sentences constructed as follows. First, for each $(i,j) \in \mathcal{I}$, we include a copy of $\Psi_{(i,j)} = \exists x: \psi_{i,j}(x)$, where $\psi_{i,j}$ is the formula constructed in Proposition~\ref{prop: isolated monotone FOS} with our choice of $C_1$. Secondly, for each $i \in \mathcal{C}$ and $j \in \mathcal{R}$, and for every $\epsilon \in \{0,1\}^{C_2}$, we include a copy of $\Psi_{i,\epsilon}$ and $\hat{\Psi}_{j,\epsilon}$ in $\mathbf{P}$, where $\Psi_{i,\epsilon}$ is the first-order sentence provided by Lemma~\ref{lemma: FOS alternating monotone sub intervals}, and $\hat{\Psi}_{j,\epsilon}$'s denote the row analogues of these sentences. 

Finally, let $\mathcal{E}$ denote the collection of all $B$ iterative sequences for our choice of $B$. For each $(i,j) \in \mathcal{S}$ and $\epsilon \in \mathcal{E}$, we include a copy of $\Psi_{\epsilon, B, (i,j)}$ in $\mathbf{P}$, and we recall that the number of such first-order sentences is indeed bounded in terms of $d$, $k$ and $B$. 

Let $C = \max(C_1,C_2,B^d+1)$, and let $N_2 = \ell m\cdot C$. Given $n \geq N_2$, we claim that for every $\sigma \in \cP_n$ there exists $\Psi \in \mathbf{P}$ for which we have $\sigma \in \cP^{\Psi}$. 

Since $n \geq N_2$, by the pigeonhole principle there exists $(i,j)$ for which we have $|\sigma^\#_{i,j}| \geq C$. If $(i,j) \in \mathcal{I}$, we clearly have $\sigma \models \Psi_{i,j}$. If the column $i$ or the row $j$ contains two non-zero entries, we certainly have $\sigma \models \Psi_{i,\epsilon}$ or $\sigma \models \hat{\Psi}_{j,\epsilon}$ respectively for some $\epsilon \in \{0,1\}^{C_2}$. Finally, if $(i,j) \in \mathcal{S}$, Lemma~\ref{lemma: iterative sequences cover all} implies that there exists $\epsilon \in \mathcal{E}$ for which we have $\sigma \models \Psi_{\epsilon,B,(i,j)}$.

Thus the collection $\mathbf{P}$ satisfies the condition (ii). By Propositions~\ref{prop: structured deletion for monotone sub intervals}, ~\ref{prop: monotone alternating intervals} and Lemma~\ref{lemma: structured deletion for stars}, for every $\Psi \in \mathbf{P}$ there exists a structured deletion function $f$ deleting at least $L$ entries. Hence Lemma~\ref{lemma: injective} implies that for each $\Psi \in \mathbf{P}$, we have $\lim_{n\rightarrow \infty} \mathbb{P}\left(\sigma_n \models \Psi\right)  \leq \gamma^{-L} < \varepsilon$, which completes the proof. 
\end{proof}

\section*{Acknowledgments}
We thank the referees for very detailed comments which substantially improved the presentation in the manuscript. The manuscript was submitted to a journal in December 2025, and all mathematical ideas and proofs were solely obtained by the authors. Originally, GPT-4 was used to generate Figure~\ref{fig:fig 2}, and during revisions, GPT-6 Astra was used for proofreading purposes and to generate Figure~\ref{fig:inflation-example}.

\bibliographystyle{plain}
\bibliography{permutationsbiblio}
\section*{Appendix}
\subsection{A winning strategy for Duplicator: proof of Lemma~\ref{lemma: infinite model ER game}}
\begin{proof}[Proof of Lemma~\ref{lemma: infinite model ER game}]
By the discussion after Lemma~\ref{lemma: infinite model ER game}, since $\Psi$ has quantifier depth at most $k$, it is enough to show Duplicator has a winning strategy in the $k$-round Ehrenfeucht--Fra\"iss\'e game on $(F^{\infty}, G_{m\times m})$ and $(F', G_{m\times m})$. Duplicator's strategy, described below, will ensure that in every round $i$, the points $p_i\in F^{\infty}$ and $q_i\in F'$ picked in that round both lie in the interior of the same cell of $G_{m\times m}$.

In particular, this means that whenever $p_{j}$ and $p_{j'}$ lie in distinct cells $C$ and $C'$ respectively, the points $q_j$ and $q_{j'}$ will lie in $C$ and $C'$ respectively, and in particular will satisfy exactly the same $\leq_1$ and $\leq_2$ relations (here note it is crucial that $M$ is a permutation matrix, i.e.\ that there is exactly one non-empty cell per row and column). Thus Duplicator can view the global Ehrenfeucht--Fra\"iss\'e game on $F^{\infty}$ and $F'$ as the union of $m$ independent games, one for each of the pairs of $C_{i,\tau(i)}$-restricted point-sets $F^{\infty}_{i,\tau(i)}$ and $F'_{i,\tau(i)}$, $i\in [m]$. Provided that for every $i\in [m]$ Duplicator has a winning strategy in the $k$-round Ehrenfeucht--Fra\"iss\'e game on $F^{\infty}_{i,\tau(i)}$ and $F'_{i,\tau(i)}$, Duplicator then clearly has a winning strategy for the $k$-round Ehrenfeucht--Fra\"iss\'e game on $F^{\infty}$ and $F'$.

We may thus focus our efforts on a single cell $C_{i,\tau(i)}$ and describe Duplicator's winning strategy in that restricted setting, setting $P=F^{\infty}_{i,\tau(i)}$ and $Q=F'_{i,\tau(i)}$. We distinguish two cases.

\textbf{If $i\in [m]\setminus I$ (the easy case)}, then Duplicator's strategy is the obvious one: both $P$ and $Q$ contain precisely $a_i$ points. If Spoiler on their $i$-th move picks the $j$-th point of $P$ in the $\leq_1$-order, then Duplicator responds by picking the $j$-th point of $Q$ in the $\leq_1$-order, and vice-versa. This is clearly seen to be a winning strategy: $p_i\leq_1 p_{i'}$ will hold if and only if $q_i\leq_1 q_{i'}$ holds, so the ordered tuples $(p_1,\ldots p_k)$ and $(q_1, \ldots, q_k)$ satisfy identical $\leq_1$-relations, which in turn implies they satisfy identical $\leq_2$-relations (since $P$, $Q$ are $M$-compatible and $M_{i,\tau(i)}\in\{-1,+1\}$).

\textbf{If $i\in I$ (the hard case)}, then Duplicator's strategy to respond to Spoiler's moves is less obvious: $P$ contains infinitely many points, while $Q$ only contains finitely many points. Thus Duplicator must take some care when replying to Spoiler's move, as they no longer have an `obvious' response of picking the unique matching point. Duplicator's strategy will be instead to ensure that gaps between successively picked points in the two sets $P$ and $Q$ are either equal in size or both large.

To be more precise, after $r\geq 0$ rounds, the players have together selected sets of $r$ points from each of $P$ and $Q$. Let $x_1\leq_1x_2\leq_1\ldots \leq_1 x_r$ and $y_1\leq_1 y_2 \leq_1\ldots  \leq_1 y_r$ be the points picked inside $P$ and $Q$ respectively, listed in the $\leq_1$ order (rather than in the order in which they were picked in the game). These points  give rise to natural partitions of $P$ and $Q$ into $r+1$ consecutive intervals: for $1\leq i\leq r-1$, set $P_i=P_i^{(r)}:=\{x\in P:\  x_{i} \leq_1 x<_1x_{i+1}\}$, and set $P_0=P_0^{(r)}:=\{x\in P: \ x<_1x_1\}$, $P_r=P_r^{(r)}:=\{x\in P: \ x_r\leq_1 x\}$. If $r = 0$, we simply set $P_0 = P_0^{(0)} = P$. We define the intervals $Q_i$, $0\leq i \leq r$, in $Q$ mutatis mutandis. Observe that, crucially for our argument, the key property of our construction of $F^{\infty}$ ensures that every (potentially infinite) interval $P_i$ has both an $\leq_1$-minimum element and an $\leq_1$-maximum element, just as the finite sets $Q_i$ do.

Duplicator will follow an inductive strategy ensuring that for every $r$: $0\leq r\leq k-1$ and every $i$: $0\leq i \leq r$, we have $\min\{\vert P_i\vert, 2^{k-r}\} =\min\{\vert Q_i\vert, 2^{k-r}\}$ --- in other words, corresponding intervals in the partitions of $P$ and $Q$ have either the same size or are both large.

Our assumption guarantees that $\vert Q\vert \geq 2^{k}$, so this is certainly satisfied when $r=0$. Suppose in round $r+1\geq 1$ Spoiler picks a point $p_{r+1}$ inside the interval $P_i$. If $\vert P_i\vert = \vert Q_i\vert  < 2^{k-r}$ and $p_{r+1}$ is the $j$-th least point of $P_i$ in the $\leq_1$-order, then Duplicator answers with the obvious move, namely playing as $q_{r+1}$ the $j$-th least point of $Q_i$ in the $\leq_1$-order.

Otherwise, Duplicator's inductive strategy up to that point implies that $\vert P_i\vert$ and $\vert Q_i\vert$ both have size at least $2^{k-r}$. Spoiler's move splits $P_i$ into two intervals $P'_i=\{x\in P_i: \ x<_1 p_{r+1}\}$ and $P''_i:=\{x\in P_i: \ p_{r+1}\leq_1 x\}$, at least one of which must have size at least $2^{k-r-1}$. Duplicator responds by picking a point $q_{r+1}\in Q_i$ as follows:
\begin{itemize}
    \item if $\vert P'_i\vert =j < 2^{k-r-1}$, then Duplicator plays as $q_{r+1}$ the $(j+1)$-th least point of $Q_i$ in the $\leq_1$-order;
    \item if $\vert P''_i\vert =j < 2^{k-r-1}$, then Duplicator plays as $q_{r+1}$ the $j$-th greatest point of $Q_i$ in the $\leq_1$-order;
    \item otherwise if $\vert P'_i\vert$ and $\vert P''_i\vert$ are both at least $2^{k-r-1}$, then Duplicator plays as $q_{r+1}$ the $2^{k-r-1}$-th greatest point of $Q_i$ in the $\leq_1$-order.
\end{itemize}
The above can be done as long as $r\leq k-1$. Duplicator's move $q_{r+1}$ splits $Q_i$ into two intervals $Q'_i=\{y\in Q_i: \ y<_1 q_{r+1}\}$ and $Q''_i:=\{y\in Q_i: \ q_{r+1}\leq_1 y\}$, and ensures that 
\begin{align*}
    \min\{2^{k-r-1}, \vert P'_i\vert \} =\min\{2^{k-r-1}, \vert Q'_i\vert  \}&& \textrm{ and } &&   \min\{2^{k-r-1}, \vert P''_i\vert \} =\min\{2^{k-r-1}, \vert Q''_i\vert  \}
\end{align*}
hold, as desired, so that our inductive hypothesis is satisfied at the end of round $r+1$ provided $r+1\leq k$. Observe that $P_i$ and $Q_i$ play symmetrical roles in the above (all we are using is the fact that $P_i$ and $Q_i$ are either of the same size or both of size at least $2^{k-r}$). Thus if Spoiler instead picks a point $q_{r+1}\in Q_i$ in round $r+1$, then Duplicator can respond with the same strategy as above, interchanging the roles of $P$ and $Q$.

It is easy to see that this constitutes a winning strategy for Duplicator in the $k$-round Ehrenfeucht--Fra\"iss\'e game on $P=F^{\infty}_{i,\tau(i)}$ and $Q=F'_{i,\tau(i)}$. This in turn implies the full statement of the lemma.
\end{proof}
\subsection{Existence of first-order formulae picking out specific points in left-greedy decompositions of permutations}
In this section, we prove a number of technical results which involve constructing first-order formulae with certain structural properties that are closely related to the left-greedy decomposition. In addition to Proposition~\ref{prop: eeroklas}, we also need the following technical result from \cite{eeroklas}. We remark that their proof is based on describing a result of Albert, Ru{\v{s}}kuc and Vatter~\cite[Proposition 2.4]{albert2015inflations} in terms of first-order formulae. 
\begin{proposition} \label{prop: eeroklas second}
	Let $B, d$ and $k$ be given positive integers. Then
	\begin{enumerate}
	\item for every $\tau \in \widetilde{\mathcal{O}}_{k,B}$, there exists a first-order formula $\psi_{\tau}(x_1,\dots,x_{|\tau|},y_1,\dots,y_{|\tau|})$ for which we have $\sigma \models \psi_{\tau}(a_1,\dots,a_{|\tau|},b_1,\dots,b_{|\tau|})$ if and only if the left-greedy decomposition of $\sigma \in \widetilde{\mathcal{O}}_{k}^{[d]}$ is given by $\sigma = \tau[\theta_1,\dots,\theta_{|\tau|}]$, and each $\theta_i$ consists of precisely those points $p$ which satisfy $p \in [a_i,b_i]_1$. 
	\item  for both $\tau \in \{\mathrm{inc},\mathrm{dec}\}$, there exists a first-order formula $\psi_{\tau}(x_1,\dots,x_{B},y_1,\dots,y_{B})$ for which we have $\sigma \models \psi_{\tau}(a_1,\dots,a_{B},b_1,\dots,b_{B})$ if and only if the left-greedy decomposition of $\sigma \in \widetilde{\mathcal{O}}_{k}^{[d]}$ is given by $\sigma = \tau'[\theta_1,\dots,\theta_{|\tau'|}]$, where $\tau'$ is a monotone permutation of direction $\tau$ of length at least $B+1$, and for each $i \in \{1,\dots,B\}$, the permutation $\theta_i$ consists of precisely those points $p$ which satisfy $p \in [a_i,b_i]_1$. 
	\end{enumerate}
\end{proposition}

Having introduced the necessary background from \cite{eeroklas}, we introduce a few technical tools that turn out to be useful. Recall that we say that $y$ is a successor of $x$ if $x <_1 y$ and there exists no $z$ for which we have $x <_1 z <_1 y$.
\begin{lemma}
	There exists a first-order formula $s_1(x,y)$ so that $\sigma \models s_1(a,b)$ if and only if $b$ is the successor of $a$. 
\end{lemma}	
\begin{proof}
	We may take $s_1(x,y)$ to be $(x \leq_1 y) \land (x \neq y) \land (\forall z: (z \leq_1 x) \lor (y \leq_1 z))$. It is easy to check that $s_1$ satisfies the required properties. 
\end{proof}
We now prove a result which roughly speaking states the following. Suppose that $\psi$ is a first-order formula which contains variables $x_1,\dots,x_\ell$ and $y_1,\dots,y_\ell$, so that whenever 
\[\sigma \models \psi(a_1,\dots,a_\ell,b_1,\dots,b_\ell),\]
the points $a_1,\dots,a_\ell$ and $b_1,\dots,b_\ell$ are uniquely determined. Then, it is possible to further refine the formula $\psi$ in a way that the points in $[a_i,b_i]_1$ satisfy further structural properties.
\begin{lemma} \label{lemma: modifications}
	Given a first-order formula $\psi(x_1,\dots,x_\ell,y_1,\dots,y_\ell)$, a  sequence of positive integers $\underline{c}$ of length $\ell$, a sequence $\epsilon \in \{=,\geq\}^\ell$, and a sequence $\underline{\tau}$ of permutations $\tau_1,\dots,\tau_\ell$, there exist first-order formulae 
	\begin{enumerate}
		\item $\psi_{\epsilon, \underline{c}}(x_1,\dots,x_\ell,y_1,\dots,y_\ell)$, for which we have $\sigma \models \psi_{\epsilon ,\underline{c}}(a_1,\dots,a_\ell,b_1,\dots,b_\ell)$ if and only if $\sigma \models \psi(a_1,\dots,a_\ell,b_1,\dots,b_\ell)$, and for each $1 \leq i \leq \ell$ there are either exactly $c_i$ or at least $c_i$ points in $[a_i,b_i]_1$ depending on whether $\epsilon_i$ equals $=$ or $\geq$ respectively
		\item $\psi_{\underline{\tau}}(x_1,\dots,x_\ell,y_1,\dots,y_\ell)$, for which we have $\sigma \models \psi_{\underline{\tau}}(a_1,\dots,a_\ell,b_1,\dots,b_\ell)$ if and only if $\sigma \models \psi(a_1,\dots,a_\ell,b_1,\dots,b_\ell)$, and for each $1 \leq i \leq \ell$ the points in $[a_i,b_i]_1$ form an interval of $\sigma$, and the subpattern induced by this interval coincides with $\tau_i$. 
	\end{enumerate}
\end{lemma}
\begin{proof}

	Let $\psi_{\epsilon, \underline{c}}(x_1,\dots,x_\ell,y_1,\dots,y_\ell)$ denote the first-order formula given by 
	\begin{align*}
		 \psi (x_1,\dots,x_\ell,y_1,\dots,y_\ell)  \land \bigwedge_{i=1}^\ell \left(\exists z_1,\dots,z_{c_i}: (z_1 = x_i) \land r_i(z_{c_i},y_i) \land \bigwedge_{j = 1}^{c_i - 1}  s_1(z_{j},z_{j+1})\right),
	\end{align*}
    where $r_i(z_{c_i},y_i)$ is $z_{c_i} = y_i$ whenever $\epsilon_i$ equals $=$, and $z_{c_i} \leq_ 1y_i$ whenever $\epsilon_i$ equals $\geq$. It is straightforward to check that $\psi_{\epsilon, \underline{c}}$ satisfies the required properties, we omit the details. 

    In order to construct $\psi_{\underline{\tau}}(x_1,\dots,x_\ell,y_1,\dots,y_\ell)$, write $p_i=\tau_i^{-1}(1)$ and $q_i=\tau_i^{-1}(|\tau_i|)$, and define the auxiliary formulae $\phi_i(x_i,y_i)$ as 
\begin{align*} 
&\exists z_1,\dots,z_{|\tau_i|}:\Bigl[(z_1=x_i)\land(z_{|\tau_i|}=y_i) \land\bigwedge_{j=1}^{|\tau_i|-1}s_1(z_j,z_{j+1})\\ &\quad\land \bigwedge_{j=1}^{|\tau_i|-1} \bigwedge_{h=j+1}^{|\tau_i|} \left(\tau_i(j)\leq\tau_i(h)\iff z_j\leq_2 z_h\right)\\ &\quad\land \left(\forall z: \Bigl(((z_{p_i}\leq_2 z)\land(z\leq_2 z_{q_i})) \implies\bigvee_{r=1}^{|\tau_i|}(z=z_r) \Bigr)\right)\Bigr].
\end{align*}
    Define $\psi_{\underline{\tau}}(x_1,\dots,x_\ell,y_1,\dots,y_\ell)$ to be given by 
	\begin{align*}
		 \psi (x_1,\dots,x_\ell,y_1,\dots,y_\ell)  \land \bigwedge_{i=1}^{\ell} \phi_i(x_i,y_i). 
	\end{align*}
	Note that the first line of $\phi_{i}(x_i,y_i)$ ensures that there are exactly $|\tau_i|$ points in $[x_i,y_i]_1$, the second line implies that the subpermutation induced by these points coincides with $\tau_i$, and the last line implies that these points form a sub-interval. Thus $\psi_{\underline{\tau}}$ satisfies the required properties. 
\end{proof}
We now move on to proving the technical results whose proofs were previously omitted.
\begin{proof}[Proof of Lemma~\ref{lemma: indicator FOSs}]

    Let $I'$ denote the collection of those $i$ for which we either have $i \in I$ or $a_i > 0$, and let $\psi_{I'}$ denote the first-order formula given in Proposition~\ref{prop: eeroklas}. Let $\underline{c}$ be the sequence indexed by $I'$ chosen so that $c_i = a_i$ whenever $i \in I^c \cap I'$ and $c_i = C$ for $i \in I$. Let $\epsilon \in \{=,\geq \}^{I'}$ be the sequence chosen so that $\epsilon_i$ equals $=$ whenever $i \in I'\cap I^c$ and $\epsilon_i$ equals $\geq$ whenever $i \in I$. Then, it is easy to check that the sentence $\Psi = \exists_{u \in \{1,2\}, v\in \{0,1\}, i \in I'} \text{ }x_{u,v,i,\tau(i)}: \psi_{\epsilon, \underline{c}, I'}(x_{u,v,i,\tau(i)})$ satisfies the required properties. 
\end{proof}
\begin{proof}[Proof of Proposition~\ref{prop: isolated monotone FOS}]
	Let $\mathcal{S}$ denote all the possible subsets of the non-zero entries of $M$ that contain $(i,j)$. Define $\phi_{i,j}(x,y)$ to be the first-order formula 
	\begin{align*}
		\bigvee_{I \in \mathcal{S}} \left(\exists_{u \in \{1,2\}, v\in \{0,1\}, (i',j') \in I}\text{ } x_{u,v,i',j'}: \psi_I(x_{u,v,i',j'}) \land (x = x_{1,0,i,j}) \land (y = x_{1,1,i,j})\right),
	\end{align*}
	where the formulae $\psi_{I}$ are given as in the statement of Proposition~\ref{prop: eeroklas}. Thus, by Proposition~\ref{prop: eeroklas}, we have $\sigma \models \phi_{i,j}(a,b)$ if and only if $\sigma^\#_{i,j}$ is non-empty, and its smallest and largest points with respect to order $\leq_1$ are $a$ and $b$ respectively. Thus, we may take $\psi(x)$ to be the formula $\exists y: \phi_{\geq, C_1}(x,y)$, where $\phi_{\geq, C_1}$ is constructed as in the statement of Lemma~\ref{lemma: modifications}. 
\end{proof}
\begin{proof}[Proof of Lemma~\ref{lemma: FOS alternating monotone sub intervals}]
    For technical reasons, we only prove the statement when $\epsilon$ is neither of the two constant sequences. In those cases, the proof is very similar, but small technical modifications are needed. 	Let $\mathcal{S}$ denote all the possible subsets of the non-zero entries of $M$ that contain both $(i,j_0)$ and $(i,j_1)$. 

     Define $\phi_{i}(y_{w,u,v}: w\in \{0,1\}, u\in \{1,2\}, v\in \{0,1\})$ to be the first-order formula given by 
     \begin{align*}
     	\bigvee_{I\in \mathcal{S}}\left(\exists_{u \in \{1,2\}, v\in \{0,1\}, (i',j')\in I} \text{ } x_{u,v,i',j'}: \psi_{I}(x_{u,v,i',j'}) \land \bigwedge_{u \in \{1,2\}, v\in \{0,1\}} \left((y_{0,u,v} = x_{u,v,i,j_0}) \land (y_{1,u,v} = x_{u,v,i,j_1})\right)  \right)
     \end{align*}
    It is easy to check that we have $\sigma \models \phi_{i}(a_{w,u,v}: w\in \{0,1\}, u\in \{1,2\}, v\in \{0,1\})$ if and only if both $\sigma^\#_{i,j_0}$ and $\sigma^\#_{i,j_1}$ are non-empty, and for both of these cells, the $4$ extremal points are given by the constants $a_{w,u,v}$ with $w \in \{0,1\}$ respectively. 
    
    Let $\psi_{\epsilon}(x)$ denote the first-order formula
    \begin{align*}
    		& \exists_{w \in \{0,1\}, u \in \{1,2\}, v\in \{0,1\}} y_{w,u,v} \text{ } \exists z_0, \dots, z_{K-1}: \phi_i(y_{w,u,v}) \\ 	  & \land \left(\left((y_{0,1,0} \leq_1 y_{1,1,0}) \implies (x = y_{0,1,0})\right) \land \left((y_{1,1,0} \leq_1 y_{0,1,0}) \implies (x = y_{1,1,0})\right)\right) \\
    		&\land (z_0 = x) \land \left((z_{K-1} \leq_1 y_{1,1,1}) \lor (z_{K-1} \leq_1 y_{0,1,1})\right) \land \bigwedge_{i=0}^{K-2} s_1(z_i,z_{i+1}) \\
    		& \land \bigwedge_{i=0}^{K-1} \left((y_{\epsilon_i,1,0} \leq_1 z_i) \land (z_i \leq_1 y_{\epsilon_i,1,1}) \land (y_{\epsilon_i,2,0} \leq_2 z_i) \land (z_i \leq_2 y_{\epsilon_i,2,1}) \right)
    \end{align*}
    We claim that $\psi_{\epsilon}(x)$ satisfies the required properties. Indeed, the construction of $\phi_i$ guarantees that whenever $\sigma \models \psi_\epsilon(a)$, then the values of $y_{w,u,v}$ must be chosen so that $\sigma^{\#}_{i,j_s}$ contains precisely those points $p$ which satisfy the conditions $p \in [y_{s,1,0},y_{s,1,1}]_1$ and $p \in [y_{s,2,0}, y_{s,2,1}]_2$. Thus, the second line implies that $a$ must be the least point in $\sigma^\#_{i,-}$ with respect to $\leq_1$. Finally, the second last line implies that $\sigma^\#_{i,-}$ must contain at least $K$ points, and the $K$ least such are represented by the values of the variables $z_0,\dots,z_{K-1}$. Finally, the last line implies that we must have $z_s \in \sigma^\#_{i,j_{\epsilon_s}}$ for all $0 \leq s \leq K-1$, as required. 
    
When $\epsilon$ is one of the two constant sequences, the proof follows similarly, but we have to adapt to the possibility that one of $\sigma^{\#}_{i,j_s}$ might be empty. We omit the details. 
\end{proof}
\begin{proof}[Proof of Lemma~\ref{lemma: FOS structure for stars}]
Given a first-order formula $\varphi$, we use the shorthand $\varphi^{[a,b]_1}$ to denote the formula obtained by restricting all its  bound variables to $[a,b]_1$. Let $\eta_{r,B}$ denote a first-order sentence chosen so that $\sigma \models \eta_{r,B}$ if and only if $|\sigma| \geq B^r+1$. 

Note that it suffices to prove the statement for $\widetilde{\mathcal{O}}_k^{[d]}$. Indeed, the statement for a cell $(i',j')$ indexed by $\star$ then follows by applying Proposition~\ref{prop: eeroklas}, identifying the extremal points of the cell $(i',j')$ by using a first-order sentence, and finally restricting all the variables to be contained between these extremal points.  We prove the statement by induction on $d$; note that the case $t=d$ also covers the base case $d=1$.

Given $\epsilon=(\tau_i,c_i,\underline{\phi_i})_{i=t}^d$, set $r_d=\min(c_d,B)$, and define $h_d=|\tau_d|$ if $\tau_d\notin\{\mathrm{inc},\mathrm{dec}\}$ and $h_d=B$ otherwise. Let $\psi_{\tau_d}$ be the formula provided by Proposition~\ref{prop: eeroklas second}, and define
\[
\chi(x_1,\dots,x_{r_d},y_1,\dots,y_{r_d})
=
\exists_{r_d<j\leq h_d}x_j,y_j:
\psi_{\tau_d}(x_1,\dots,x_{h_d},y_1,\dots,y_{h_d}).
\]
Combining this with the formulae constructed in the proof of Lemma~\ref{lemma: modifications}, we can construct a formula $\chi_{\underline{\phi}_d}(x_1,\dots,x_{r_d},y_1,\dots,y_{r_d})$ which also ensures that for every $1\leq j\leq c_d-1$ the interval $[x_j,y_j]_1$ induces $\phi_{d,j}$. 

If $t=d$, then $c_d=B+1$ and $r_d=B$. We take
\[
\psi_{\epsilon,B}
=
\eta_{d,B}\land
\chi_{\underline{\phi}_d}(x_1,\dots,x_B,y_1,\dots,y_B),
\]
and relabel the variables $x_j,y_j$ as $x_{d,j},y_{d,j}$. The formula $\psi_{\tau_d}$ ensures that the profile is a monotone permutation of length at least $B+1$ and whose direction matches the parameter $\tau_d$, and  Lemma~\ref{lemma: modifications} ensures that the first $B$ blocks match the first $B$ blocks described in $\underline{\phi}_d$. This completes the proof of the case when $t = d$. 

If $t<d$, let $\epsilon'$ be obtained by deleting the $d$-th entry of $\epsilon$, and let $\psi_{\epsilon',B}$ be the formula provided by the inductive  hypothesis. Let $\mathcal{S'}$ denote the appropriate set of indices for $\epsilon'$ as described in the statement of the Lemma. Since $r_d=c_d$, define 
\begin{align*}
\psi_{\epsilon,B}:={}&\eta_{d,B} \land\chi_{\underline{\phi_d}} (x_1,\dots,x_{c_d},y_1,\dots,y_{c_d})\\ &\land \bigl(\psi_{\epsilon',B}\bigr)^{[x_{c_d},y_{c_d}]_1} (x_{u,v},y_{u,v}:(u,v)\in\mathcal{S}')\\ &\land\bigwedge_{(u,v)\in\mathcal{S}'} \left((x_{c_d}\leq_1 x_{u,v}) \land(x_{u,v}\leq_1 y_{u,v}) \land(y_{u,v}\leq_1 y_{c_d})\right).
\end{align*}
We again relabel the variables $x_j,y_j$ as $x_{d,j},y_{d,j}$. The inductive hypothesis ensures that the $c_d$-th block respects $\epsilon'$, and in particular contains at least $B^{d-1}+1$ points.  Furthermore, the construction of the formula ensures that the total number of points is at least $B^d+1$, that the blocks preceding the $c_d$-th block have at most $B^{d-1}$ points in them and they match the given profile $\underline{\phi}_{d,i}$ for each $1 \leq i \leq c_d-1$, and that their endpoints are uniquely determined. This completes the proof. 
\end{proof}
\end{document}